\documentclass[10pt,reqno]{amsart}
\usepackage{amsmath,amssymb,latexsym,soul,mathrsfs}
\usepackage{ulem}
\usepackage{color,enumitem,graphicx}
\usepackage[english]{babel}
\usepackage[left=1.8cm,right=1.8cm,top=2.1cm,bottom=2.1cm]{geometry}
\usepackage{comment}
\usepackage{booktabs}
\usepackage{array}
\usepackage{multirow}
\usepackage{tabularx}
\usepackage{cite}

\makeatletter
\newcommand{\leqnomode}{\tagsleft@true}
\newcommand{\reqnomode}{\tagsleft@false}
\makeatother

\def\nd{\noindent}
\def\thend{\rule{3mm}{3mm}}
\newtheorem{theorem}{Theorem}[section]

\newtheorem{cor}{Corollary}[section]
\newtheorem{prop}{Proposition}[section]
\newtheorem{lem}{Lemma}[section]
\newtheorem{rmk}{Remark}[section]

\newcommand{\eps}{\epsilon}

\usepackage[colorlinks=true,
            urlcolor=blue,
            citecolor=red,
            linkcolor=blue,
            linktocpage,
            pdfpagelabels,
            bookmarksnumbered,
            bookmarksopen,
            backref=page]{hyperref}

\begin{document}
	\title[Extremal parameters, multiplicity and nonexistence for a singular elliptic problem]{Extremal parameters, multiplicity and nonexistence for a singular elliptic problem with competing nonlinearities}
	\vspace{1cm}
	
	\author{Jos\'e Carlos de Albuquerque$^{*}$}
	\address{J. C. de Albuquerque \newline Department of Mathematics, Federal University of Pernambuco, Brazil }
	\email{\tt josecarlos.melojunior@ufpe.br}
    \thanks{The first three authors were partially supported by CNPq with grant 301849/2025-0, 300411/2025-1 and 310440/2025-4, respectively.}
	
	\author{M. L. M. Carvalho}
	\address{M. L. M. Carvalho \newline Universidade Federal de Goi\'as, IME, Goi\^ania-GO, Brazil }
	\email{\tt marcos$\_$leandro$\_$carvalho@ufg.br}
	
	\author{Edcarlos D. Silva}
	\address{Edcarlos D. da Silva \newline  Universidade Federal de Goi\'as, IME, Goi\^ania-GO, Brazil}
	\email{\tt edcarlos@ufg.br}
	
	\author{M. L. Silva}
	\address{M. L. Silva \newline Universidade Federal de Goi\'as, IME, Goi\^ania-GO, Brazil }
	\email{\tt maxwell@ufg.br}

    \thanks{$^{*}$Corresponding author: J. C. de Albuquerque}
	
		

\begin{abstract}
We consider the following class of singular elliptic problems
\[
\begin{cases}
-\Delta u
=
-\dfrac{u}{|u|^{\beta+1}}\chi_{\{|u|>0\}}
+\lambda |u|^{p-1}u,
& \text{in }\Omega,\\[1mm]
u=0,
& \text{on }\partial\Omega,
\end{cases}
\]
where $\Omega\subset\mathbb R^N$ is a bounded smooth domain, $N\geq3$, $0<\beta,p<1$, and $\lambda>0$. The energy functional is not of class $C^1$ in $H_0^1(\Omega)$, and the standard critical point theory does not apply. We prove that a nontrivial nonnegative weak solution exists if and only if $\lambda$ exceeds a critical value $\bar\lambda$, which we identify and locate strictly between two Rayleigh-type parameters $0<\lambda_*<\bar\lambda<\lambda^*<+\infty$. This gives a sharp existence threshold. We also prove that for every $\lambda$ in a neighborhood on the left of $\lambda^*$, there exist at least two distinct nontrivial nonnegative weak solutions. The first is a minimizer on the Nehari manifold, with an energy changing sign at $\lambda^*$. The second has a strictly positive energy. For $\lambda<\lambda^*$ sufficiently close to $\lambda^*$, both solutions have positive energy. This contrasts with
previous multiplicity results for the same problem, in which the two
solutions obtained have energies of opposite signs.
\end{abstract}
	
\maketitle

\section{Introduction}
Singular elliptic equations arise naturally in several
models involving singular reactions. In fact, problems with singular nonlinearities have appeared, among
other contexts, in fluid mechanics, pseudoplastic flows,
non-Newtonian fluids, heterogeneous catalytic processes, and
biological pattern formation; see, for instance,
\cite{CrandallRabinowitzTartar,DiazMorelOswald,zhang1995,ShiYao,GherguRadulescu}.
In this work, we study the following class of singular elliptic
problems
\begin{equation}\label{P}
    \left\{
    \begin{array}{ll}
        -\Delta u
        =
        -\dfrac{u}{|u|^{\beta+1}}
        \chi_{\{|u|>0\}}
        +
        \lambda |u|^{p-1}u,
        & \text{in }\Omega,
        \\[2mm]
        u=0,
        & \text{on }\partial\Omega,
    \end{array}
    \right.
    \tag{$P_\lambda$}
\end{equation}
where $\Omega\subset\mathbb R^N$ is a smooth bounded domain, $N\geq3$,
$0<\beta,p<1$,  $\lambda>0$ is a real parameter and $\chi_{\{|u|>0\}}$ is the characteristic function for the set ${\{|u|>0\}}:=\{x\in \Omega\;;\;|u(x)|>0\}.$ From the
variational point of view, the same competition is reflected in the
formal energy associated with \eqref{P} given by
\[
    I_\lambda(u)
    =
    \frac12\|u\|^2
    +
    \frac{1}{1-\beta}
    \int_\Omega |u|^{1-\beta}\,dx
    -
    \frac{\lambda}{p+1}
    \int_\Omega |u|^{p+1}\,dx.
\]
The three terms above have different homogeneities, $1-\beta<p+1<2$. As a consequence, the behavior of $I_\lambda$ along one-dimensional fibers is determined by the interplay between the singular, sublinear and Dirichlet terms. This structure is one of the
main reasons why nonlinear Rayleigh quotients and Nehari-type
decompositions are particularly well suited to the analysis of
\eqref{P}. Furthermore, the singularity creates a substantial
analytical difficulty. Although $I_\lambda$ is formally associated
with \eqref{P}, $I_\lambda$ is not of class $C^1$ on
$H_0^1(\Omega)$, once it possesses the non differentiable mapping 
\[ u\longmapsto
    \int_\Omega |u|^{1-\beta}\,dx \]
at functions that vanish on sets of positive
Lebesgue measure. Therefore, the usual critical point theory cannot be applied
directly. In particular, the variational geometry suggested by the
formal functional can be combined with an approximation procedure
that allows one to construct critical points. After that,
we recover solutions of the original singular problem. Moreover, we
provide a complete variational description which depends on the parameter $\lambda > 0$. Namely, we identify the exact threshold for
existence of nontrivial nonnegative weak solutions, the change of sign
of the minimizing energy, and the structure of the solution set in a
full neighborhood of the critical parameter.

The dependence on the parameter $\lambda$ for related singular
problems has been studied from different points of view. In
\cite{zhang1995}, positive solutions were obtained for sufficiently
large values of $\lambda>0$. Later, in \cite{ShiYao}, the authors
considered
\[
    -\Delta u+K(x)u^{-\beta}=\lambda u^p,
\]
where $K$ is positive. They identified a critical parameter separating
the nonexistence and existence regimes for positive solutions. In
particular, the case $K\equiv1$ corresponds to the positive version
of \eqref{P}. In the same paper, the authors also obtained a maximal positive solution and
studied its dependence on $\lambda$.
Further contributions to singular elliptic equations include the
seminal work of Lazer and McKenna~\cite{LazerMcKenna}, where existence
and boundary behavior for singular Dirichlet problems were studied.
We also refer to D\'avila~\cite{Davila2004} for global regularity and
local $H^1$-minimizers associated with a nondifferentiable singular
functional, and to Figueiredo and Montenegro
\cite{FigueiredoMontenegro} for singular elliptic equations involving
critical nonlinearities.

A different and relevant point of view was developed in
\cite{DavilaMontenegro2003,Davila}. They considered equations of the
form
\[
    -\Delta u
    =
    \left(
    -u^{-\beta}
    +\lambda f(x,u)
    \right)\chi_{\{u>0\}},
\]
and studied maximal nonnegative solutions together with the transition
between positive and free-boundary behavior. Since $f(u)=u^p$ is
nonnegative, nondecreasing, concave, and sublinear, their framework
contains the nonnegative weak solutions for the problem  \eqref{P}. As a consequence, critical
parameters associated with positive solvability and with the
positive/free-boundary transition are already present in the
literature. These parameters are concerned with positivity properties
of the corresponding solutions. They should be distinguished from the
threshold for the existence of arbitrary nontrivial nonnegative weak
solutions for the problem considered above.

A variational treatment of \eqref{P} was developed in \cite{Montenegro}. In that paper, the authors introduced the regularization
\[
    g_\varepsilon(s)
    =
    \frac{s^q}{(s+\varepsilon)^{q+\beta}},
    \qquad
    0<q<p<1,
\]
and studied the corresponding nonsingular problems
\[
    -\Delta u+g_\varepsilon(u)=\lambda u^p.
\]
For $\lambda$ sufficiently large, they obtained two critical points
of the regularized energy, one by global minimization and another one by
the mountain-pass theorem. A delicate passage to the limit
$\varepsilon\to0^+$, based on uniform local gradient
estimates, implies the existence of two distinct nontrivial nonnegative weak
solutions of the singular problem. They also developed a second
approach based on perturbations of the domain. Hence, the authors in \cite{Montenegro} studied the existence of a positive
solution for large $\lambda$ and, under some extra restrictions on
$p$, they proved a uniqueness result within a suitable class of positive functions.

The same equation can be used in the framework considered in \cite{AnelloFaraci2015}. They proved
the existence of two distinct nontrivial nonnegative weak solutions
for sufficiently large values of $\lambda$, for the equation
\[
    -\Delta u
    =
    \left(
    \lambda u^{p}-u^{-\beta}
    \right)\chi_{\{u>0\}},  
\] one of them being
positive and a local minimizer of the energy, and obtained
nonexistence of positive solutions below a critical parameter. Later, a multiplicity result for a different range of singular
exponents was obtained in \cite{AnelloFaraci2017}.

The approximation procedure introduced for singular elliptic problems
has also been developed in several works. For example, in \cite{MontenegroStapenhorst} the authors obtained positive solutions by the method
of perturbation of domains. Furthermore, in \cite{StapenhorstExp,StapenhorstCritical} more general nonlinearities were considered. In those works, the authors also constructed solutions
through suitable regularization of the singular term. The main feature is to consider a procedure taking a limit in order to control the singular nonlinear term. In fact, we apply a regularization procedure, and the passage
$\varepsilon\to0^+$ is now a well-known tool.

The fibering approach for semilinear elliptic problems was
developed in \cite{yavdat2} where the parameter $\lambda > 0$ is controlled by an extremal value $\lambda^*$ which
is characterized through variational principles. This mathematical viewpoint
was later systematically formulated in terms of nonlinear generalized
Rayleigh quotients, see \cite{yavdat1}. In that work, the extremal parameters $\lambda_*$ and $\lambda^*$ are introduced
through nonlinear Rayleigh quotients and used in the Nehari manifold
analysis. These parameters were used in order to ensure existence and multiplicity of critical values of the energy functional associated with the semilinear elliptic problem. More recently, nonlinear
Rayleigh quotient techniques have been applied to several singular
nonlocal problems, see \cite{CardosoSilvaCarvalhoYang, SilvaLeiteSilva, SilvaRochaSilva}.

It is important to stress that existence and 
multiplicity of positive solutions for Problem \eqref{P} and related
problems have been widely studied in recent years. However, some important issues are not considered in the literature. Here, we collect some interesting questions for the study of singular elliptic problems.
\begin{itemize}
    \item [(i)] the exact threshold for the
existence of nontrivial nonnegative weak solutions;
\item [(ii)] the change of
sign of the minimizing energy as a function of $\lambda$;
\item [(iii)] the
full Nehari geometry in a neighborhood of the critical parameter.
\end{itemize}
\subsection{Contribution of the present work}\label{sec:contribution}
The purpose of the present work is to provide existence and multiplicity of solutions for singular elliptic problems that depend on a parameter $\lambda > 0$. Our approach combines the regularization method with
nonlinear Rayleigh quotients and the Nehari method. A key feature of our approach is the use of nonlinear Rayleigh
quotients to identify sharp extremal parameters governing the Nehari
geometry. This analysis also allows us to apply the Lagrange multiplier theorem
and to show that components of the Nehari manifold are
natural constraints. Recall that the regularized
singular term is not homogeneous. Furthermore, for the regularized problem, we need to consider the fibering maps, Rayleigh quotients, and extremal values which depend on a parameter $\varepsilon > 0$. The main challenge here is to show that the variational and
geometric properties for the fibering maps are stable as $\varepsilon\to0^+$. The last assertion, together with some careful analysis, implies that the original singular problem has the same properties as the regularized problem. More precisely, we identify two Rayleigh-type parameters:
$\lambda_*$, associated with the Nehari geometry, and $\lambda^*$,
associated with the change of sign of the minimizing energy. We
prove the multiplicity of solutions in a whole neighborhood of $\lambda^*$, including
parameters below $\lambda^*$. Hence, we classify the corresponding energy
levels. We further identify an extreme sharp parameter $\bar\lambda$,
such that the existence of weak nontrivial nonnegative 
solutions satisfies $\lambda_*<\bar\lambda<\lambda^*$. The main contribution is not only to recover multiplicity of solutions for
large parameters $\lambda > 0$. In fact, we provide a finer variational description of
the parameter $\lambda > 0$ and find some extra properties on existence and nonexistence of solutions for Problem \eqref{P}. Furthermore, using the nonlinear
Rayleigh quotient method combined with regularization, we give the
geometric information on singular elliptic problems.

Our results also provide the exact threshold
$\bar\lambda$ separating existence from nonexistence of nontrivial
nonnegative weak solutions. Furthermore, we give the location of $\bar\lambda$ inside
the interval $(\lambda_*,\lambda^*)$. Moreover, we prove the existence of two
distinct nontrivial nonnegative weak solutions for every
$\lambda\in(\lambda^*-\epsilon,+\infty)$, including parameters below
$\lambda^*$. On the other hand, we give the complete classification of the energy of the minimizing
solution as positive, zero, or negative, according to the position of
$\lambda$ relative to $\lambda^*$. In particular, for $\lambda$ in a
left neighborhood of $\lambda^*$, there exist two solutions with positive energies. This phenomenon does not appear in the previous results given by \cite{Montenegro} and \cite{AnelloFaraci2015}.

The main technical novelty is to consider the nonlinear Rayleigh quotient
analysis, which is developed at the level of the regularized problems. For this problem, 
the standard arguments are not available due to the loss of homogeneity. This
requires a separate analysis of the fibering map, the Rayleigh
levels, and the associated minimizers, together with some uniform estimates
with respect to the regularization parameter $\varepsilon  > 0$. Hence, we show that the
variational and geometric information is stable as
$\varepsilon\to0^+$ and can be used for the original singular elliptic
problem. To the best of our knowledge, the present work is the first one
considering the nonlinear Rayleigh quotient method for singular elliptic problems where the perturbation techniques are discussed taking into account the presence of competing nonhomogeneous nonlinearities.

It is worthwhile to mention that our main results extend the work \cite{Montenegro} in several directions. First, we identify the sharp
existence threshold $\bar\lambda$. We also prove the multiplicity of solutions for
$\lambda$ below $\lambda^*$. Furthermore, we classify the sign of the
minimizing energy. Moreover, we show that the two solutions have energies of
the same sign for $\lambda$ close to $\lambda^*$ from below.
On the other hand, we also mention that the present work generalizes some results discussed in \cite{AnelloFaraci2015}. Indeed, we do not require
$\lambda$ to be large and we provide the sharp existence threshold which describes the full Nehari geometry in a neighborhood of the
critical parameter. These improvements are summarized in
Figures~\ref{fig-EM1} and~\ref{fig-EM2} below.

The main differences between the present work and the existing literature
are summarized in Table~\ref{tab:comparison}. We emphasize that the
comparison with \cite{DavilaMontenegro2003} concerns the threshold for
positive solvability, which is a different (though related) question from
the sharp threshold for nontrivial nonnegative weak solutions considered here.

\begin{table}[htbp]
\centering
\caption{Comparison of the present work with previous contributions.}
\label{tab:comparison}
\footnotesize
\renewcommand{\arraystretch}{1.0}
\begin{tabularx}{\textwidth}{@{}l *{4}{>{\raggedright\arraybackslash}X} @{}}
\toprule
\textbf{Feature} 
& \textbf{\cite{Montenegro}} 
& \textbf{\cite{AnelloFaraci2015}} 
& \textbf{\cite{DavilaMontenegro2003}} 
& \textbf{Present work} \\
\noalign{\global\arrayrulewidth=0.01pt}\hline

Existence regime 
& $\lambda$ large 
& $\lambda$ large 
& Positive solvability threshold 
& $\lambda \geq \bar\lambda$ (sharp) \\
\noalign{\global\arrayrulewidth=0.01pt}\hline
Multiplicity 
& Two solutions\newline for $\lambda$ large 
& Two solutions \newline for $\lambda$ large 
& Not addressed 
& Two solutions for $\lambda \in (\lambda^*-\epsilon,+\infty)$ \\
\noalign{\global\arrayrulewidth=0.01pt}\hline
Sharp existence threshold 
& No 
& No 
& Partial (for positive solutions) 
& Yes ($\bar\lambda$) \\
\noalign{\global\arrayrulewidth=0.01pt}\hline
Sign of minimizing energy 
& Not classified 
& Not classified 
& Not classified 
& Changes sign according to $\lambda \lessgtr \lambda^*$ \\
\noalign{\global\arrayrulewidth=0.01pt}\hline
Nehari geometry 
& Not described 
& Not described 
& Not described 
& Complete description via $\lambda_*$, $\bar\lambda$, $\lambda^*$ \\
\noalign{\global\arrayrulewidth=0.01pt}\hline
Strict positivity 
& Not established 
& Yes (one solution) 
& Yes (positive branch) 
& Not addressed \\
\noalign{\global\arrayrulewidth=0.01pt}\hline
Energy of the two solutions 
& Opposite signs 
& Opposite signs \newline (for large $\lambda$) 
& Not applicable 
& Same sign for $\lambda$ close to $\lambda^*$ from below \\
\noalign{\global\arrayrulewidth=0.01pt}\hline
Method 
& Regularization $+$ variational 
& Variational 
& Monotone iteration $+$ free boundary 
& Regularization and Nonlinear Rayleigh quotients \\
\bottomrule
\end{tabularx}
\end{table}
\subsection{Strategy of the proof}\label{sec:strategy}
The proof combines three main ingredients. First, we develop the
nonlinear Rayleigh quotient machinery directly for the singular
problem and identify the threshold parameters $\lambda_*$ and
$\lambda^*$, together with the associated fibering geometry and the
decomposition of the Nehari set. This first step identifies the
natural thresholds and the geometry of the Nehari set, and provides
the variational framework for the subsequent analysis. Second,
following the regularization procedure of Montenegro and Silva
\cite{Montenegro}, we introduce a family of nonsingular problems and
study the corresponding Rayleigh quotients $\lambda_{n,\eps}$ and
$\lambda_{e,\eps}$. Since the regularized singular term is not
homogeneous along the fibers, the threshold levels, the minimizers,
and the associated energies must be analyzed separately, and we
establish uniform estimates with respect to the regularization
parameter, both for the Rayleigh levels and for the associated
minimizers. This second step provides a family of nonsingular problems
for which the standard critical point theory is available. Third, we
let $\varepsilon\to0^+$ and show that the variational information, the
critical levels, and the solutions constructed at the regularized
level persist in the singular elliptic problem. This last step allows us to
transfer the Nehari decomposition, the energy classification, and the
extremal parameter $\bar\lambda$ to the original problem \eqref{P}.

\subsection{Notation and weak solutions}\label{sec:notation}

Before stating our main results, we specify the notion of weak solution
for the present work. Due to the presence of the singular
term, the weak formulation requires some careful analysis. In particular, the
singular term is understood only on the set $\{|u|>0\}$. Moreover, we consider the corresponding local integrability condition. Hence, we
use the following notion of weak solution: We say that
$u\in H_0^1(\Omega)$ is a weak solution of \eqref{P} whenever
\(
    \frac{1}{|u|^\beta}\chi_{\{|u|>0\}}
    \in L^1_{\mathrm{loc}}(\Omega)
\)
and
\[
    \int_\Omega \nabla u\nabla\varphi\,dx
    =
    \int_{\{|u|>0\}}
    \left(
        -\frac{u}{|u|^{\beta+1}}
        +\lambda |u|^{p-1}u
    \right)\varphi\,dx
\]
for every $\varphi\in C_c^\infty(\Omega)$. In fact, the last assertion is equivalent to the following identity  
\[
    \int_\Omega \nabla u\nabla\varphi\,dx
    +
    \int_{\{|u|>0\}}
    \frac{u}{|u|^{\beta+1}}\varphi\,dx
    =
    \lambda
    \int_\Omega |u|^{p-1}u\,\varphi\,dx,
    \qquad
    \forall\,\varphi\in C_c^\infty(\Omega).
\]
This approach aligns with \cite{Montenegro}, where nonnegative solutions are considered for the problem
\[
    -\Delta u
    =
    \left(
        -u^{-\beta}
        +\lambda u^p
    \right)\chi_{\{u>0\}}
\]
where the weak formulation is given by
\[
    \int_\Omega \nabla u\nabla\varphi\,dx
    =
    \int_{\{u>0\}}
    \left(
        -\frac{1}{u^\beta}
        +\lambda u^p
    \right)\varphi\,dx,
    \qquad
    \forall\,\varphi\in C_c^1(\Omega).
\]
Indeed, whenever $u\geq0$, we deduce that 
\(
    \{|u|>0\}=\{u>0\},
    ~
    \frac{u}{|u|^{\beta+1}}
    =
    \frac{1}{u^\beta}
    ~\text{on }\{u>0\},
\)
and
\(
    |u|^{p-1}u=u^p.
\)
Therefore, for nonnegative functions, our notion of weak solution
coincides with that used in \cite{Montenegro}. In fact, we use the odd extensions in the following form:
\[
    s\longmapsto \frac{s}{|s|^{\beta+1}}
    \quad\text{and}\quad
    s\longmapsto |s|^{p-1}s.
\]
The last assertion allows us to formulate problem $(P_\lambda)$ in order to consider a more general setting.  Hence, our formulation also solves the problem considered in~\cite{Montenegro}. Notice also that nonnegativity is obtained
through the variational  procedures using the
invariance of the functionals under the maps $u \mapsto |u|$.

Now, for each $u\in H_0^1(\Omega)\setminus\{0\}$, we consider the fibering map
$\phi_u:(0,+\infty)\to\mathbb{R}$ defined by
\begin{equation}\label{fibering}
	\phi_u(t):=I_\lambda(tu)
	=\frac{t^2}{2}\|u\|^2
	+\frac{t^{1-\beta}}{1-\beta}\|u\|_{1-\beta}^{1-\beta}
	-\frac{\lambda t^{p+1}}{p+1}\|u\|_{p+1}^{p+1},
	\qquad t>0.
\end{equation}
Notice that $\phi_u\in C^\infty((0,+\infty),\mathbb{R})$ for each
$u\in H_0^1(\Omega)\setminus\{0\}$.
Throughout the paper, we use the notation
$$\|u\|_r:=\|u\|_{L^r(\Omega)} \quad \mbox{and} \quad \|u\|^2:=\|u\|^2_{H_0^1(\Omega)}=\int_\Omega|\nabla u|^2\,dx,$$ where $1\leq r\leq\infty$.
Moreover, $S_{p+1}$ and $S_{2^*}$ denote the Sobolev constants for
the embeddings of $H_0^1(\Omega)$ into $L^{p+1}(\Omega)$ and
$L^{2^*}(\Omega)$, respectively. More specifically, we obtain 
\(\|u\|_{p+1}\leq S_{p+1}\|u\|\) and
\(\|u\|_{2^*}\leq S_{2^*}\|u\|\) for every $u\in H_0^1(\Omega)$.

At this stage, by using the same ideas from \cite{nehari1,nehari2}, we
consider the Nehari method for our Problem \eqref{P} as follows:
	\begin{equation}
			\mathcal{N}_{\lambda}:=\left\{ u \in H_0^1(\Omega)\setminus \{0\} : \phi_u'(1)=0 \right\}
			=\left\{ u\in  H_0^1(\Omega)\setminus \{0\}:\Vert u\Vert^2+\|u\|^{1-\beta}_{1-\beta}=\lambda\|u\|^{p+1}_{p+1}\right\}.
		\label{nehari}
	\end{equation}
	Under these conditions, following the same ideas explored in \cite{tarantello}, we split the Nehari set $\mathcal{N}_\lambda$ into three disjoint subsets;
\[
\begin{aligned}
\mathcal{N}^+_{\lambda}
&=\{u\in\mathcal{N}_{\lambda}: \phi_u''(1)>0\},\\
\mathcal{N}^-_{\lambda}
&=\{u\in\mathcal{N}_{\lambda}: \phi_u''(1)<0\},\\
\mathcal{N}^0_{\lambda}
&=\{u\in\mathcal{N}_{\lambda}: \phi_u''(1)=0\}.
\end{aligned}
\]
In order to apply the nonlinear Rayleigh quotient, we also need to
consider some definitions. First, we shall define 
	\begin{eqnarray}\label{conjE}
		\mathcal{E}_{\lambda}\!\!&\!\!=\!\!&\!\!\{u\in H_0^1(\Omega)\setminus \{0\}:I_{\lambda}(u)=0 \}
		=\left\{ u\in  H_0^1(\Omega)\setminus \{0\}:\dfrac{1}{2}\Vert u\Vert^2+\frac{1}{1-\beta}\|u\|^{1-\beta}_{1-\beta} =\dfrac{\lambda}{p+1}\|u\|^{p+1}_{p+1}\right\}.
	\end{eqnarray}
This work aims to find weak solutions for the main problem by solving the following minimization problems:
	\begin{equation}\label{ee1}
		c_{\mathcal{N}^-_{\lambda}}:=\inf\{I_{\lambda}(u): u\in\mathcal{N}_{\lambda}^-\},
	\end{equation}
	\begin{equation}\label{ee2}
		c_{\mathcal{N}^+_{\lambda}}:=\inf\{I_{\lambda}(u): u\in\mathcal{N}_{\lambda}^+\}.
	\end{equation}
	Notice also that
	\begin{equation}\label{a0i}
		u \in \mathcal{N}_{\lambda} \quad \mbox{if and only if} \quad \lambda = \dfrac{\|u\|^2 +\|u\|^{1-\beta}_{1-\beta}}{\|u\|^{p+1}_{p+1}}.
	\end{equation}
	Similarly, we deduce the following assertion:
	\begin{equation}\label{a11i}
		u\in\mathcal{E}_{\lambda} \quad \mbox{if and only if} \quad \lambda = \dfrac{\frac{1}{2}\|u\|^2 +\frac{1}{1-\beta}\|u\|^{1-\beta}_{1-\beta}}{\frac{1}{p+1}\|u\|^{p+1}_{p+1}}.
	\end{equation}
	In this case, we define the functionals $R_n,R_e: H_0^1(\Omega)\setminus\{0\} \to \mathbb{R}$ by
	\begin{eqnarray}\label{R-pert}
		R_{n}(u)=\dfrac{\|u\|^2 +\|u\|^{1-\beta}_{1-\beta}}{\|u\|^{p+1}_{p+1}}\qquad\mbox{and}\qquad R_{e}(u)= \dfrac{\frac{1}{2}\|u\|^2 +\frac{1}{1-\beta}\|u\|^{1-\beta}_{1-\beta}}{\frac{1}{p+1}\|u\|^{p+1}_{p+1}}.
	\end{eqnarray}
The functionals $R_n$ and $R_e$ are continuous on
$H_0^1(\Omega)\setminus\{0\}$. However, due to the
singular term $\|u\|_{1-\beta}^{1-\beta}$, these functionals are not 
Fr\'echet differentiable on $H_0^1(\Omega)\setminus\{0\}$ in general.
Nevertheless, for every fixed $u\in H_0^1(\Omega)\setminus\{0\}$, the
maps
\[
t\mapsto R_n(tu)
\qquad\text{and}\qquad
t\mapsto R_e(tu)
\]
are of class $C^\infty$. As a consequence,  taking into account 
the singular nonlinear term, derivatives of the
Rayleigh quotients along fibers are understood with respect to the
scalar variable $t>0$. At this stage, we also consider
	\begin{equation}\label{mun}
	   \lambda_*:=\inf_{u\in H_0^1(\Omega)\setminus\{0\}}\inf_{t>0}R_n(tu) \quad \mbox{and} \quad \lambda^*:=\inf_{u\in H_0^1(\Omega)\setminus\{0\}}\inf_{t>0}R_e(tu).
    \end{equation}
We emphasize that all functionals considered in this work are
invariant under the function $u \mapsto |u|$. Namely, 
\[
I_\lambda(|u|)=I_\lambda(u),\qquad
R_n(|u|)=R_n(u),\qquad
R_e(|u|)=R_e(u),
\]
Analogously, these types of identities hold for the perturbed functionals.
Consequently, the minimizing sequences and constrained minimizers can be
replaced by their absolute values without changing the corresponding
energy levels. This symmetry is also used in the proof of our main results.

\subsection{Main results}\label{sec:main-results}

The main results of this work can be stated as follows.
\begin{theorem}\label{theorem1}
    Suppose $0<\beta,p<1$. Then
    \(0<\lambda_*<\lambda^*<+\infty\).
    Furthermore, there exists $\epsilon_1>0$ such that Problem
    \eqref{P} admits a nontrivial nonnegative weak solution
    $u_\lambda\in H_0^1(\Omega)$ whenever
    \(\lambda\in(\lambda^*-\epsilon_1,+\infty).\)
    Moreover, for each such $\lambda$, there exists $A>0$ such that
    $I_\lambda(u_\lambda)\geq A$.
\end{theorem}

\begin{theorem}\label{theorem2}
	Suppose $0<\beta,p < 1$. Then there exists $\epsilon_2>0$ such that
	Problem \eqref{P} admits a nontrivial nonnegative weak solution
	$v_{\lambda} \in H_0^1(\Omega)$, distinct from the solution
	$u_\lambda$ of Theorem~\ref{theorem1}, whenever
	$\lambda \in (\lambda^*-\epsilon_2, +\infty)$, satisfying the
	following statements:
	\begin{itemize}
		\item[(i)] It holds that $\phi_{v_\lambda}''(1)>0$.
		\item[(ii)] Assume that $\lambda\in (\lambda^*-\epsilon_2, \lambda^*)$. Then $I_{\lambda}(v_{\lambda})>0$;
		\item[(iii)] For $\lambda=\lambda^*$ we obtain that $I_{\lambda^*}(v_{\lambda^*})=0$;
		\item[(iv)] Assume that $\lambda>\lambda^*$. Then $I_\lambda(v_{\lambda})<0$.
	\end{itemize}
\end{theorem}

	As a consequence, by using Theorems \ref{theorem1} and
	\ref{theorem2}, we obtain the following result.
\begin{cor}\label{coro} Suppose $0<\beta,p < 1$. Then Problem \eqref{P} admits at least two distinct nontrivial nonnegative weak solutions for each $\lambda \in (\lambda^*-\eps_2,+\infty)$.
   \end{cor}
	Now, we define
\begin{equation}\label{lambda-bar}
    \mathcal{S}
    :=
    \left\{
        \lambda>0:
        \eqref{P}\text{ admits a nontrivial nonnegative weak solution}
    \right\},
    \qquad
    \bar{\lambda}:=\inf\mathcal{S}.
\end{equation}
\begin{theorem}\label{theorem4a} If $0<\beta,p<1$, then $\lambda_*<\bar\lambda<\lambda^*$ and $(P_{\;\bar\lambda})$ admits a nontrivial nonnegative weak solution.
\end{theorem}
Now, we establish the sharp existence threshold for Problem \eqref{P}.
\begin{theorem}\label{theorem4}
Suppose $0<\beta,p<1$. Then Problem~\eqref{P} admits a nontrivial
nonnegative weak solution if and only if
\(
\lambda\geq\bar\lambda.
\)
Equivalently,
\(
\mathcal S=[\bar\lambda,+\infty).
\)
\end{theorem}

We  emphasize that although $\mathcal{N}_{\lambda}\not=\emptyset$ for all $\lambda\geq \lambda_{\ast},$ we only have nonnegative solutions for $\lambda\geq \bar\lambda.$

\begin{rmk}
The variational approach developed in \cite{Yijing} cannot be applied
directly to Problem~\eqref{P}. Indeed, the negative singular term
$-1/u^\beta$ competes with the positive sublinear term $\lambda u^p$,
and hence the positivity and monotonicity arguments used there are not
directly available in the present setting. This provides an additional
motivation for combining the regularization procedure with the
nonlinear Rayleigh quotient method.
\end{rmk}

\begin{rmk}
In \cite{Montenegro}, the authors proved the existence of at
least two distinct nontrivial nonnegative solutions for sufficiently
large values of $\lambda$. They introduced a family of regularized
problems and provided a sequence procedure using some standard arguments. However, the authors do not determine the sharp threshold for the existence of
nontrivial nonnegative weak solutions. Notice also that the authors do not describe the
transition of the Nehari geometry and the corresponding energy levels
which depend on the size of $\lambda$. Moreover, strict positivity of solutions was not established.
\end{rmk}

\begin{rmk}
Now, by using the nonlinear Rayleigh quotient and the regularization
method, we obtain a finer description of the parameter dependence of
the problem. The Nehari threshold $\lambda_*$ and the energy-change
threshold $\lambda^*$ describe, respectively, the changes in the
fibering geometry and in the sign of the minimizing energy, whereas the
existence threshold $\bar\lambda$ determines the sharp lower boundary
for the existence of nontrivial nonnegative weak solutions. In
particular, we establish existence at $\lambda=\bar\lambda$ and
nonexistence of nontrivial weak solutions for
$\lambda<\bar\lambda$, together with the multiplicity and energy
conclusions stated in the main theorems.
\end{rmk}

\begin{figure}[!ht]
\centering
\includegraphics[scale=0.6]{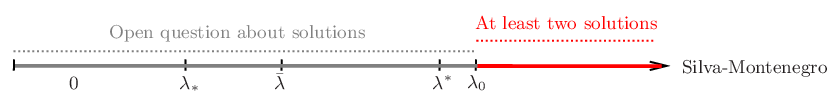}
\caption{Range of parameters for the existence of solutions obtained
in \cite{Montenegro}.}
\label{fig-EM1}
\end{figure}
\begin{figure}[!ht]
\centering
\includegraphics[scale=0.6]{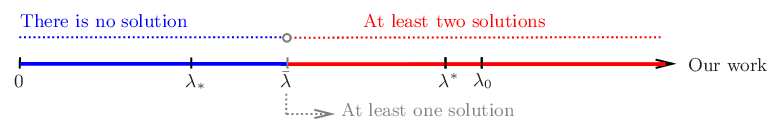}
\caption{Our Range of parameters for the existence and nonexistence of
solutions.}
\label{fig-EM2}
\end{figure}

\begin{rmk}\label{rmk:novelty-montenegro}
It is important to distinguish what is adapted from \cite{Montenegro}
and what is genuinely new in the present work. The regularization
scheme, based on the family
$g_\epsilon(s)=s^q/(s+\epsilon)^{q+\beta}$ with $0<q<p<1$, and the
passage to the singular limit $\epsilon\to0^+$, which relies on the
uniform $L^\infty$ estimates and the local gradient estimates of
\cite{Montenegro}, are the same as those employed there. In
particular, the existence of two distinct nontrivial nonnegative weak
solutions for sufficiently large $\lambda$ was already established in
\cite{Montenegro}, together with a perturbation-of-domains approach
and, under an additional restriction on $p$, a uniqueness result
within a suitable class of positive solutions. Thus, the use of a
regularized problem and the passage $\epsilon\to0^+$ are not the novelty of the present contribution.
\end{rmk}

\begin{rmk}\label{rmk:regularity-classical} We briefly comment on the regularity of the solutions constructed in
this work. For each fixed $\epsilon>0$,
\[
    g_\epsilon(s)
    =
    \frac{|s|^{q-1}s}{(|s|+\epsilon)^{q+\beta}}
\]
is continuous and bounded. However, it is not
differentiable at the origin because $0<q<1$. More precisely,
$g_\epsilon$ is locally H\"older continuous with exponent $q$. The solutions
$v_{\lambda,\epsilon}$ and $u_{\lambda,\epsilon}$ obtained in
Propositions~\ref{min-real-N-eps} and~\ref{MPS} are nonnegative.
Therefore, the uniform $L^\infty$ estimate for the regularized
problem established in \cite[Lemma~2.4]{Montenegro} applies to these
solutions. In particular, $v_{\lambda,\epsilon},u_{\lambda,\epsilon}\in L^\infty(\Omega)$. Hence the right-hand side of
\[
    -\Delta u
    =
    -g_\epsilon(u)
    +
    \lambda |u|^{p-1}u
\]
belongs to $L^\infty(\Omega)$. Since $\Omega$ is smooth, the global
$L^r$-regularity theory for the Dirichlet problem yields $v_{\lambda,\epsilon},u_{\lambda,\epsilon}\in W^{2,r}(\Omega)$, for every $1<r<\infty$. Consequently, by the Sobolev embedding theorem, $v_{\lambda,\epsilon},u_{\lambda,\epsilon}\in C^{1,\alpha}(\overline\Omega)$, for every $\alpha\in(0,1)$. Furthermore, the maps
\[
    s\longmapsto |s|^{p-1}s
    \qquad\text{and}\qquad
    s\longmapsto g_\epsilon(s)
\]
are locally H\"older continuous with exponents $p$ and $q$,
respectively. Since $q<p$, the right-hand side of the regularized
equation is locally H\"older continuous once
$u\in C^{1,\alpha}(\overline\Omega)$. Standard Schauder estimates
imply that $v_{\lambda,\epsilon},u_{\lambda,\epsilon}\in C^{2,\gamma}(\overline\Omega)$, for every $0<\gamma<q$. Thus the regularized solutions are classical.

We now turn to the singular problem~\eqref{P}. The uniform
$L^\infty$ estimate and the local gradient estimate of
\cite[Lemma~2.4 and Lemma~3.1]{Montenegro}, applied to the
regularized solutions, imply that their limits
$u_\lambda$ and $v_\lambda$ are bounded and locally Lipschitz
continuous in $\Omega$. The truncation and Fatou arguments used in the passage
$\epsilon\to0^+$, as in
\cite[Lemma~4.1 and the proof of Theorem~1.1]{Montenegro}, yield $u_\lambda^{-\beta}\chi_{\{u_\lambda>0\}},\,v_\lambda^{-\beta}\chi_{\{v_\lambda>0\}}\in L^1_{\mathrm{loc}}(\Omega)$, which is precisely the local integrability required in the notion of
weak solution adopted in Section~\ref{sec:notation}.
\end{rmk}

\subsection{Outline}\label{sec:outline}
The remainder of this paper is organized as follows.
In Section~\ref{R-nt}, we study the Rayleigh quotients $R_n$ and $R_e$
associated with Problem~\eqref{P}. We establish the relations with
the Nehari set and the threshold parameters $\lambda_*$ and
$\lambda^*$. We also derive the main geometric and minimization
properties associated with the decomposition of the Nehari set. In
Section~\ref{trunc}, following the regularization procedure introduced
in \cite{Montenegro}, we study the perturbed problem
\eqref{P-pertubed}. We define the corresponding perturbed Rayleigh
quotients, analyze the associated threshold levels, and examine their
dependence on the regularization parameter $\varepsilon$. Section
\ref{Pass-lim} contains the most technical part of the proof, where
the passage to the limit is analyzed, and the extremal parameter is
identified. We then complete the proofs of the existence,
multiplicity, and energy-classification results. Finally, we establish
existence at the extremal parameter $\bar\lambda$ and prove the
corresponding nonexistence result below this threshold parameter.

\section{Analysis of the singular problem via Rayleigh quotients}\label{R-nt}
    
This section establishes properties of the energy functional $I_{\lambda }$ and nonlinear Rayleigh quotients for Problem \eqref{P}, followed by two remarks on the functionals $R_n$ and $R_e$ defined in \eqref{R-pert}.
    
\begin{rmk}\label{rmk1}
Let $u\in H_0^1(\Omega)\setminus\{0\}$ be fixed, and define \( \phi_u(t):=I_\lambda(tu),~ t>0.\)
Then
\[
\phi_u'(1)
=
\|u\|^2+\|u\|_{1-\beta}^{1-\beta}
-\lambda\|u\|_{p+1}^{p+1}.
\]
Consequently,
\begin{itemize}
    \item[(i)] $R_n(u)=\lambda$ if and only if $\phi_u'(1)=0$;
    \item[(ii)] $R_n(u)>\lambda$ if and only if $\phi_u'(1)>0$;
    \item[(iii)] $R_n(u)<\lambda$ if and only if $\phi_u'(1)<0$.
\end{itemize}
\end{rmk}
		
		\begin{rmk}\label{rmk11} 
			Let $u\in H_0^1(\Omega)\setminus\{0\}$ be fixed. The following conclusions hold:
			\begin{itemize}
				\item[(i)] ${R}_e(u)=\lambda$ if and only if $I_\lambda(u)=0$; 
				\item[(ii)] $R_e(u) > \lambda$ if and only if $I_\lambda (u) > 0$;
				\item[(iii)] $R_e(u) < \lambda$ if and only if $I_\lambda (u) < 0$.
			\end{itemize}
		\end{rmk}
        
		Let $u\in H_0^1(\Omega)\setminus\{0\}$ be fixed. Since
\[
R_e(tu)-\lambda
=
\frac{p+1}{t^{p+1}\|u\|_{p+1}^{p+1}}
\phi_u(t),
\]
if $R_e(tu)=\lambda$, then $\phi_u(t)=0$, and differentiation with
respect to $t$ yields
\begin{equation}\label{aee}
\frac{d}{dt}R_e(tu)
=
\frac{p+1}{t^{p+1}\|u\|_{p+1}^{p+1}}
\phi_u'(t).
\end{equation}
As a consequence, we obtain the following result.

\begin{prop}\label{importante}
Suppose that $u\in H_0^1(\Omega)\setminus\{0\}$ and that
$R_e(tu)=\lambda$ for some $t>0$. Then:
\begin{itemize}
    \item[(i)]
    $\dfrac{d}{dt}R_e(tu)>0$ if and only if $\phi_u'(t)>0$;
    \item[(ii)]
    $\dfrac{d}{dt}R_e(tu)<0$ if and only if $\phi_u'(t)<0$;
    \item[(iii)]
    $\dfrac{d}{dt}R_e(tu)=0$ if and only if $\phi_u'(t)=0$.
\end{itemize}
\end{prop}

\begin{proof}
The result follows immediately from \eqref{aee}, since
\( \displaystyle
\frac{p+1}{t^{p+1}\|u\|_{p+1}^{p+1}}>0.
\)
\end{proof}
		
		Analogously, let $u\in H_0^1(\Omega)\setminus\{0\}$ be fixed. Since
\[
R_n(tu)-\lambda
=
\frac{\phi_u'(t)}
{t^p\|u\|_{p+1}^{p+1}},
\]
if $R_n(tu)=\lambda$, then $\phi_u'(t)=0$. Differentiating with
respect to $t$, we obtain
\begin{equation}\label{save}
\frac{d}{dt}R_n(tu)
=
\frac{\phi_u''(t)}
{t^p\|u\|_{p+1}^{p+1}}.
\end{equation}
In particular, we have the following result.

\begin{prop}\label{der-Rn}
Suppose that $u\in H_0^1(\Omega)\setminus\{0\}$ and that
$R_n(tu)=\lambda$ for some $t>0$. Then:
\begin{itemize}
    \item[(i)]
    $\dfrac{d}{dt}R_n(tu)>0$ if and only if $\phi_u''(t)>0$;
    \item[(ii)]
    $\dfrac{d}{dt}R_n(tu)<0$ if and only if $\phi_u''(t)<0$;
    \item[(iii)]
    $\dfrac{d}{dt}R_n(tu)=0$ if and only if $\phi_u''(t)=0$.
\end{itemize}
\end{prop}

\begin{proof}
The result follows immediately from \eqref{save}, since
\(
t^p\|u\|_{p+1}^{p+1}>0.
\)
\end{proof}

    \begin{lem}\label{lam-n-e}
     For each $u\in H_0^1(\Omega)\backslash\{0\}$, there exist unique $t_{n}(u)>0$ and $t_{e}(u)>0$ such that
      \[
       R_{n}(t_{n}(u)u)=\min_{t>0}R_{n}(tu) \quad \mbox{and} \quad R_{e}(t_{e}(u)u)=\min_{t>0}R_{e}(tu).
      \]
    \end{lem}
    \begin{proof}
        First, ${\lim\limits_{t\to 0^+}R_n(tu)}=\lim\limits_{t\to+\infty}R_n(tu)=+\infty.$	A direct computation shows that $\frac{d}{dt}R_n(tu)=0$ if and only if
	\begin{eqnarray}\label{t-n}
		t=t_n(u):=\left[\frac{(\beta+p)\|u\|_{1-\beta}^{1-\beta}}{(1-p)\|u\|^2}\right]^{\frac{1}{1+\beta}}.
	\end{eqnarray}
    Similarly, $\lim\limits_{t\to 0^+}R_e(tu)=\lim\limits_{t\to+\infty}R_e(tu)=+\infty$
		and $\frac{d}{dt}R_e(tu)=0$ if and only if
		$$t=t_e(u):=\left(\frac{2}{1-\beta}\right)^{\frac{1}{1+\beta}}t_n(u),$$
    which finishes the proof.
    \end{proof}
In view of the previous lemma, we define
	\begin{equation}\label{lambda_n}
		\lambda_n(u):=R_n(t_n(u)u)=c_{\beta,p}\frac{\|u\|^{2\frac{p+\beta}{\beta+1}}\|u\|_{1-\beta}^{(1-\beta)\frac{1-p}{\beta+1}}}{\|u\|_{p+1}^{p+1}},\;\;\;\text{with}\;\;\;c_{\beta,p}:=\frac{1+\beta}{1-p}\left(\frac{1-p}{\beta+p}\right)^{\frac{\beta+p}{\beta+1}}
	\end{equation}
	and
		\begin{equation}\label{lambda_e}
			\lambda_e(u):=R_e(t_e(u)u)=\tilde{c}_{\beta,p}\frac{\|u\|^{2\frac{p+\beta}{\beta+1}}\|u\|_{1-\beta}^{(1-\beta)\frac{1-p}{\beta+1}}}{\|u\|_{p+1}^{p+1}},\;\;\;\text{where}\;\;\;\tilde{c}_{\beta,p}:=\frac{p+1}{1-\beta}\left(\frac{1-\beta}{2}\right)^{\frac{\beta+p}{\beta+1}}{c}_{\beta,p}.
		\end{equation}

    \begin{rmk}
		We point out that $\tilde{c}_{\beta,p}>{c}_{\beta,p}$. Indeed, this assertion is equivalent to proving that $f(p)>f(-\beta)$, where $f(t)=\left(\frac{t+1}{2}\right)^{\frac{1}{1-t}}$. Since $f$ is increasing in $[-\beta,p]$, the desired result follows. 
    \end{rmk}

   In view of Lemma \ref{lam-n-e} together with the identity
    $$R_n(tu) - R_e(tu) = \frac{t}{p+1} \frac{d}{dt} R_e(tu), \quad t>0,$$    
    we can show how the geometry of the fibering maps associated with $R_n$ and $R_e$ is organized, see Figure \ref{fig2}.
    \begin{figure}[!ht]
    \center{\includegraphics[scale=0.7]{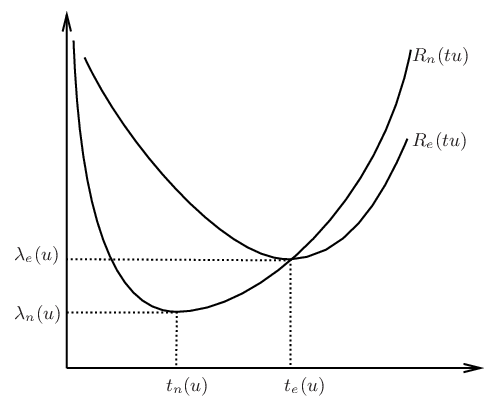}}
			\caption{The fibering functions $R_n(tu)$ and $R_e(tu)$.}
			\label{fig2}
	\end{figure}

        \begin{prop}\label{N-way-origin} If $0<\beta,p < 1$, and $\lambda\geq \lambda_*$, then $\mathcal{N}_{\lambda}$ is bounded away from the origin, where $\lambda_*$ is given by \eqref{mun}.

		\end{prop}
		\begin{proof}
			For $u\in \mathcal{N}_{\lambda}$, we have
			\begin{equation}\label{des-1}
				\|u\|_{1-\beta}^{1-\beta}\leq\displaystyle \lambda\|u\|_{p+1}^{p+1}\qquad\mbox{and}\qquad\|u\|\leq\displaystyle \lambda^{\frac{1}{2}}\|u\|_{p+1}^{\frac{p+1}{2}}.
			\end{equation}		
		Now, in view of \eqref{des-1}, the embedding $H_0^1(\Omega)\hookrightarrow L^{2^{\ast}}(\Omega)$ and the interpolation $p+1= s(1-\beta)+(1-s)2^{\ast}$, we obtain
			 \begin{align*}
			 		\|u\|_{p+1}^{p+1}\leq & \, {\|u\|_{1-\beta}^{s(1-\beta)}}\|u\|_{2^{\ast}}^{(1-s)2^{\ast}}\\
                    \leq & \, \lambda^s\|u\|_{p+1}^{s(p+1)} S_{2^{\ast}}^{(1-s)2^{\ast}}{\|u\|^{{(1-s)2^{\ast}}}}\\
                    \leq & \, \lambda^s S_{2^{\ast}}^{(1-s)2^{\ast}}\|u\|_{p+1}^{s(p+1)} \left({\lambda^{\frac{1}{2}}\|u\|_{p+1}^{\frac{p+1}{2}}}\right)^{(1-s)2^{\ast}}\\
                    = & \, \lambda^{\frac{2s+(1-s)2^{\ast}}{2}} S_{2^{\ast}}^{(1-s)2^{\ast}}\|u\|_{p+1}^{s(p+1)+\frac{(p+1)(1-s)2^{\ast}}{2}}.
		 	\end{align*}
	 		 Hence, by using the fact that 
	 		 $$\frac{2s(p+1)+(p+1)(1-s)2^{\ast}}{2}-(p+1)=\frac{p+1}{2}(1-s)(2^{\ast}-2),$$
	 		 we obtain  
            $$
	 		 \frac{1}{\lambda^{\frac{2s+(1-s)2^{\ast}}{2}}S_{2^{\ast}}^{(1-s)2^{\ast}}}\leq\|u\|_{p+1}^{\frac{2s(p+1)+(p+1)(1-s)2^{\ast}}{2}-(p+1)}=\|u\|_{p+1}^{\frac{(p+1)}{2}[(1-s)(2^{\ast}-2)]}.
	 		 $$
		Therefore,		
			\begin{equation}\label{dis-ori}
				\|u\|\geq  S_{p+1}^{-1}\|u\|_{p+1}\geq\frac{1}{S_{p+1}\lambda^{\frac{2s+(1-s)2^{\ast}}{2(1-s)(2^{\ast}-2)}\;}S_{2^{\ast}}^{\frac{2^{\ast}}{2^{\ast}-2}}}=:\delta_\lambda,
			\end{equation}
        and the proof is complete.
		\end{proof}
		
		\begin{prop} \label{fechada}	
			Suppose $0<\beta,p < 1$ and $\lambda\geq\lambda_*$. Then there exists $C=C_\lambda > 0$ such that $\|u\| \leq C$ for every $u \in \mathcal{N}_{\lambda}$. Furthermore, the Nehari sets $\mathcal{N}_{\lambda}$ and $\mathcal{N}_{\lambda}^0$ are closed.
		\end{prop}
		\begin{proof}
			 For $u\in \mathcal{N}_\lambda$, we have 
            \begin{equation}\label{C_lambda}
             \|u\|^{1-p}\leq\lambda S_{p+1}^{p+1}:=C_{\lambda}.
            \end{equation} 
        Moreover, the continuous functionals
				$F_1, F_2:H_0^1(\Omega)\to \mathbb{R}$ given by 
				$$F_1(u)=\|u\|^2+\|u\|^{1-\beta}_{1-\beta}-\lambda\|u\|^{p+1}_{p+1}\qquad\mbox{and}\qquad F_2(u)=2\|u\|^2+(1-\beta)\|u\|^{1-\beta}_{1-\beta}-\lambda(p+1)\|u\|^{p+1}_{p+1}$$
				imply that $\mathcal{N}_\lambda\cup \{0\}=F_{1}^{-1}(\{0\})$ and $\mathcal{N}_\lambda^0\cup \{0\}=F_{1}^{-1}(\{0\})\cap F_{2}^{-1}(\{0\})$ are closed sets. In view of Proposition \ref{N-way-origin}, the sets $\mathcal{N}_\lambda$ and $\mathcal{N}_\lambda^0$ are bounded away from zero, we conclude that $\mathcal{N}_\lambda$ and $\mathcal{N}_\lambda^0$ are closed.
		\end{proof}
	
		\begin{cor}\label{weak-conv}
			Suppose $0<\beta,p < 1$ and $\lambda\geq\lambda_*$. Then, for any sequence $(u_k)\subset \mathcal{N}_\lambda$, there exist a subsequence $(u_{k_j})\subset (u_k)$ and $u\in H_0^1(\Omega)\setminus\{0\}$ such that $u_{k_j}\rightharpoonup u$ as $j\rightarrow\infty$.
		\end{cor}
    \begin{proof}
In view of Proposition~\ref{fechada}, $(u_k)$ is bounded in
$H_0^1(\Omega)$. Hence, up to a subsequence, $u_k\rightharpoonup u$ weakly in $H_0^1(\Omega)$ and $u_k\to u$ strongly in $L^{p+1}(\Omega)$. On the other hand, estimate~\eqref{dis-ori} yields $\|u_k\|_{p+1}^{p+1}\geq d_\lambda>0$ for every $k$, where $d_\lambda$ is independent of $k$.
Thus, $u\neq0$, since
\[\|u\|_{p+1}^{p+1}
    =
    \lim_{k\to\infty}\|u_k\|_{p+1}^{p+1}
    \geq d_\lambda>0.\]
This completes the proof.
\end{proof}
	
		\begin{prop}\label{N0-dist-0}
			Suppose $0<\beta,p < 1$ and $\lambda\geq\lambda_*$. Then there exists $\sigma_\lambda>0$ such that $I_{\lambda}(u)\geq \sigma_\lambda$ for all $u\in \mathcal{N}_\lambda^0$.
		\end{prop}
		\begin{proof}
			Let $u\in \mathcal{N}_{\lambda}^0$ be fixed. Then
            $$
            \left\{
            \begin{array}{l}
            \|u\|^2+\|u\|_{1-\beta}^{1-\beta}=\lambda \|u\|_{p+1}^{p+1},\vspace{0,2cm} \\
            2\|u\|^2+(1-\beta)\|u\|_{1-\beta}^{1-\beta}=(p+1)\lambda \|u\|_{p+1}^{p+1}.
            \end{array}
            \right.
			$$ 
            Thus,		
			$$
            \left\{
            \begin{array}{l}
            (1-p)\|u\|^2-(\beta+p)\|u\|_{1-\beta}^{1-\beta}=0,\vspace{0,2cm}\\
            (\beta+1)\|u\|_{1-\beta}^{1-\beta}=(1-p)\lambda\|u\|_{p+1}^{p+1},
            \end{array}
            \right.
            \quad \Longrightarrow \quad 
            \left\{
            \begin{array}{l}
                \|u\|_{1-\beta}^{1-\beta} =\frac{1-p}{\beta+p}\|u\|^2\\ \\
                \lambda\|u\|_{p+1}^{p+1} =\frac{1+\beta}{\beta+p}\|u\|^2.
            \end{array}
            \right.
            $$		
			Hence,
\[
I_\lambda(u)
=
\frac{(1-p)(1+\beta)}
{2(1-\beta)(p+1)}\|u\|^2.
\]
In view of \eqref{dis-ori}, we obtain
\begin{equation}\label{sigm-lambda}
I_\lambda(u)
\geq
\frac{(1-p)(1+\beta)}
{2(1-\beta)(p+1)}
\left(
\frac{1}{
S_{p+1}^{p+1}
\lambda^{\frac{2s+(1-s)2^*}{(1-s)(2^*-2)}}
S_{2^*}^{\frac{2\cdot2^*}{2^*-2}}
}
\right)^{\frac{2}{p+1}}
=:\sigma_\lambda>0.
\end{equation}		
This ends the proof.
\end{proof}

	\begin{lem}\label{N0}
		If $0<\beta,p < 1$, and  $\lambda_n : H_0^1(\Omega) \setminus \{0\} \to \mathbb{R}$ given by \eqref{lambda_n}, then the following assertions hold:
		\begin{itemize}
			\item[(i)] $\lambda_n$ is $0$--homogeneous, i.e., $\lambda_n(tu)=\lambda_n(u)$ for every $t>0$ and $u\in H_0^1(\Omega)\setminus\{0\}$;
			\item[(ii)] The function $\lambda_n: H_0^1(\Omega)\setminus\{0\} \to \mathbb{R}$ is continuous and weakly lower semicontinuous;
			\item[(iii)] $\lambda_n$ is bounded from below by a positive constant, i.e., there exists $C = C(\beta,p,N,\Omega) > 0$ such that $\lambda_n(u) \geq C > 0$ for all $u\in H_0^1(\Omega)$. Furthermore, $\lambda_n$ is unbounded from above, i.e., there exists a sequence $(w_k) \in H_0^1(\Omega) \setminus \{0\}$ such that $\lambda_n(w_k) \to \infty$ as $k \to \infty$;
			\item[(iv)] There exists $u\in H_0^1(\Omega)\setminus\{0\}$ such that $$\lambda_n(u)=\displaystyle\lambda_*=\inf_{v\in H_0^1(\Omega)\setminus\{0\}}\lambda_n(v)>0.$$
		\end{itemize}	
	\end{lem}
	\begin{proof}
     For $(i)$, we notice that
      \[
          \lambda_n(tu) =  c_{\beta,p}\frac{\|tu\|^{2\frac{p+\beta}{\beta+1}}\|tu\|_{1-\beta}^{(1-\beta)\frac{1-p}{\beta+1}}}{\|tu\|_{p+1}^{p+1}}= t^{\,2\frac{p+\beta}{\beta+1}+(1-\beta)\frac{1-p}{\beta+1}-(p+1)}\lambda_n(u)=\lambda_n(u).
      \]
    Item $(ii)$ follows from the compact Sobolev embeddings and the weak lower semicontinuity of the $H_0^1(\Omega)$-norm. To prove $(iii)$, we use the interpolation $p+1={2}t+{(1-\beta)}(1-t)$ jointly with the Sobolev embedding to obtain 
		$$
		\|u\|_{p+1}^{p+1}=\int_{\Omega}|u|^{2t}|u|^{(1-\beta)(1-t)}\leq\|u\|_{2}^{2t}\|u\|_{1-\beta}^{(1-\beta)(1-t)}\leq\frac{1}{\Lambda_1^{t}}\|u\|^{2t}\|u\|_{1-\beta}^{(1-\beta)(1-t)},
		$$
	where $\Lambda_1$ denotes the first eigenvalue of the Laplacian. Consequently,
		$$
    		\lambda_n(u):=c_{\beta,p}\frac{\|u\|^{2t}\|u\|_{1-\beta}^{(1-\beta)(1-t)}}{\|u\|_{p+1}^{p+1}}\geq c_{\beta,p}\Lambda_1^{t}>0, \quad \forall u\in H_0^1(\Omega)\setminus\{0\},
		$$	
		  which proves that $\lambda_n$ is bounded from below by a positive constant. On the other hand, let $(\varphi_k)$ be a sequence of $L^2$-normalized eigenfunctions associated with the eigenvalues $(\Lambda_k)$. Then
          \begin{eqnarray}\label{coer-ln}
              \lambda_n(\varphi_k)=c_{\beta,p}\frac{\Lambda_k^{t}\|\varphi_k\|_{1-\beta}^{(1-\beta)(1-t)}}{\|\varphi_k\|_{p+1}^{p+1}}\geq c_{\beta,p}\frac{\Lambda_k^{t}\|\varphi_k\|_{1-\beta}^{(1-\beta)(1-t)}}{\|\varphi_k\|_{1-\beta}^{(1-\beta)(1-t)}}= c_{\beta,p}\Lambda_k^{t}\to \infty, \quad \mbox{as } k\rightarrow\infty.
          \end{eqnarray}

    \medskip
    \noindent \textbf{Proof of $(iv)$.}
    Let $(u_k)\subset H_0^1(\Omega)\setminus\{0\}$ be a minimizing sequence for $\lambda_*$. In view of item $(i)$, we may assume, without loss of generality, that $\|u_k\|_{p+1}=1$ for every $k\in\mathbb{N}$. Thus,
\[
\lambda_n(u_k)
=
c_{\beta,p}
\|u_k\|^{2\frac{p+\beta}{1+\beta}}
\|u_k\|_{1-\beta}^{
(1-\beta)\frac{1-p}{1+\beta}}
\longrightarrow\lambda_*.
\]
Let $\theta\in(0,1)$ be determined by
\(
p+1=\theta(1-\beta)+(1-\theta)2^*.
\)
By interpolation and the Sobolev embedding, we obtain
\[
\begin{aligned}
1
=
\|u_k\|_{p+1}^{p+1}
\leq
\|u_k\|_{1-\beta}^{\theta(1-\beta)}
\|u_k\|_{2^*}^{(1-\theta)2^*}\leq
S_{2^*}^{(1-\theta)2^*}
\|u_k\|_{1-\beta}^{\theta(1-\beta)}
\|u_k\|^{(1-\theta)2^*}.
\end{aligned}
\]
Consequently,
\[
\|u_k\|_{1-\beta}^{
(1-\beta)\frac{1-p}{1+\beta}}
\geq
C
\|u_k\|^{-
\frac{(1-\theta)2^*}{\theta}
\frac{1-p}{1+\beta}}.
\]
Hence,
\(
\lambda_n(u_k)
\geq
C\|u_k\|^\kappa,
\)
where
\(
\kappa
:=
2\frac{p+\beta}{1+\beta}
-
\frac{(1-\theta)2^*}{\theta}
\frac{1-p}{1+\beta}
=
\frac{(p+\beta)(p+1)(2^*-2)}
{(1+\beta)(2^*-(p+1))}
>0.
\)
Since $(\lambda_n(u_k))$ is bounded, we conclude that $(u_k)$ is
bounded in $H_0^1(\Omega)$. Thus, up to a subsequence,
\(
u_k\rightharpoonup u
~\text{in }H_0^1(\Omega)
\)
and, by the compact Sobolev embedding,
\(
u_k\to u
\quad\text{in }L^{p+1}(\Omega).
\)
Therefore,
\(
\|u\|_{p+1}=1,
\)
and hence $u\neq0$. Moreover,
\(
\|u_k\|_{1-\beta}^{1-\beta}
\rightarrow
\|u\|_{1-\beta}^{1-\beta}.
\)
Indeed, taking into account that $0<1-\beta<p+1$ and $\Omega$ is bounded, the strong
convergence in $L^{p+1}(\Omega)$ implies convergence in
$L^{1-\beta}(\Omega)$ by interpolation, and hence
\[
\int_\Omega
\left|
|u_k|^{1-\beta}-|u|^{1-\beta}
\right|\,\mathrm{d}x
\leq
\int_\Omega |u_k-u|^{1-\beta}\,\mathrm{d}x
\rightarrow0.
\]
Thus, by the weak lower semicontinuity of the norm, we conclude that
\[
\lambda_n(u)
\leq
\liminf_{k\to\infty}\lambda_n(u_k)
=
\lambda_*.
\]
On the other hand, by the definition of $\lambda_*$, we have
\(
\lambda_*\leq\lambda_n(u).
\)
Therefore,
\(
\lambda_n(u)=\lambda_*,
\)
which proves item $(iv)$.
\end{proof}
	
	\begin{lem}\label{gamma}
		If $0<\beta,p < 1$, and $\lambda_e : H_0^1(\Omega) \setminus \{0\} \to \mathbb{R}$ comes from \eqref{lambda_e}, then the following statements hold:
		\begin{itemize}
			\item[(i)] $\lambda_e : H_0^1(\Omega)\setminus\{0\} \to \mathbb{R}$ is $0$--homogeneous,  continuous, weakly lower semicontinuous, and bounded from below;
			\item[(ii)] $\lambda_e(u)>\lambda_n(u)$ for every $u \in H_0^1(\Omega) \setminus \{0\}$;
			\item[(iii)] There exists $u^*\in H_0^1(\Omega)\setminus\{0\}$ such that $$\displaystyle\lambda^*=\lambda_e(u^*)=\inf_{v\in H_0^1(\Omega)\setminus\{0\}}R_e(v)>0.$$ 
            In particular, $0<\lambda_*<\lambda^*$. Moreover, there exists $t^{+}_{\lambda^*} (u^*)=t_{e}(u^*)>0$ such that $v_{\lambda^*}=t^{+}_{\lambda^*} (u^*)u^*\in \mathcal{N}_{\lambda^*}^+$ and $I_{\lambda^*}(v_{\lambda^*})=c_{\mathcal{N}_{\lambda^*}^{+}}=0$.
		\end{itemize}
	\end{lem}
	\begin{proof}
		From \eqref{lambda_n} and \eqref{lambda_e}, we have
\(
\lambda_e(u)
=
\frac{\tilde c_{\beta,p}}{c_{\beta,p}}\lambda_n(u),
~
\frac{\tilde c_{\beta,p}}{c_{\beta,p}}>1.
\)
Hence, items (i)--(iii) follow directly from Lemma \ref{N0}. In particular,
the minimizers of $\lambda_e$ and $\lambda_n$ coincide and
\(
\lambda^*
=
\frac{\tilde c_{\beta,p}}{c_{\beta,p}}\lambda_*
>
\lambda_*.
\)
It remains only to prove the last statement of item (iii). First, we claim that $c_{\mathcal{N}_{\lambda^*}^+} \geq 0$. Indeed, for any $u \in \mathcal{N}_{\lambda^*}^+$, we have $\lambda_e(u) \geq \lambda^* = R_n(u)$. Since $t_e(u)$ is the global minimum of the map $t \mapsto R_e(tu)$, it follows that $$R_e(u) \geq \lambda_e(u) = R_e(t_e(u)u) \geq \lambda^* = R_n(u).$$ In view of Remark \ref{rmk11} (i), we have $I_{\lambda^*}(u) \geq 0$ for all $u \in \mathcal{N}_{\lambda^*}^+$. Therefore, $c_{\mathcal{N}_{\lambda^*}^+}\geq0$.

        To show that $c_{\mathcal{N}_{\lambda^*}^{+}}=0$, we prove that there exists a function $v_{\lambda^*}\in \mathcal{N}_{\lambda^*}^+$ such that $I_{\lambda^*}(v_{\lambda^*}) = 0$. By the first part of item (iii), there exists a nontrivial function $u^* \in H_0^1(\Omega) \setminus \{0\}$ such that $\lambda^* = \lambda_e(u^*) = R_e(t_e(u^*)u^*)$, which is equivalent to $I_{\lambda^*}(t_e(u^*)u^*) = 0$; see Remark \ref{rmk11} (i). On the other hand, it is straightforward to verify that 
\begin{equation}
    R_n(tu) - R_e(tu) = \frac{t}{p+1} \frac{d}{dt} R_e(tu).
\end{equation}
Since $t_e(u^*)$ is the unique critical point of the map $t \mapsto R_e(tu^*)$, we conclude that $\lambda^* = R_e(t_e(u^*)u^*)= R_n(t_e(u^*)u^*)$, which implies that $I_{\lambda^*}(t_{\lambda^*}^{n,+}(u^*)u^*) = 0$; see Remark \ref{rmk11} (i). Since $\lambda^*>\lambda_*=\lambda_n(u^*)$, the fibering map $\phi_{u^*}$ has exactly two critical points $t_{\lambda^*}^{n,-}(u^*)<t_n(u^*)<t_{\lambda^*}^{n,+}(u^*)$, and the minimum of $t\mapsto R_e(tu^*)$ is attained to the right of $t_n(u^*)$, so $t_e(u^*)>t_n(u^*)$. Hence, the identity above forces $t_e(u^*) = t_{\lambda^*}^{n,+}(u^*)$. Therefore, $c_{\mathcal{N}_{\lambda^*}^{+}}=I_{\lambda^*}(t_{\lambda^*}^{n,+}(u^*)u^*)=0$.
\end{proof}

	At this stage, we prove that for any fixed function $u\in H_0^1(\Omega)\setminus \{0\}$ the fibering map $\phi(t)=I_{\lambda} (tu)$ has exactly two distinct critical points whenever $\lambda > \lambda_n(u) $; see Figure \ref{fig1}. More precisely, we prove the following important result:  
    
\begin{prop} \label{compar} Suppose $0<\beta,p < 1$. Then, for each  $u\in H_0^1(\Omega)\setminus \{0\}$, the following assertions hold: 
\begin{itemize}
\item[(i)] If $\lambda>\lambda_n(u) =R_n(t_n(u)u)$, then  the fibering map $\phi_u$ has exactly two distinct critical points $0<t_\lambda^{n,-}(u)<t_n(u)<t_\lambda^{n,+}(u)$;
\item[(ii)] The functional $t_\lambda^{n,+}(u)$ is a local minimum point for the fibering map $\phi_u$ which satisfies $t_\lambda^{n,+}(u)u\in\mathcal{N}_{\lambda}^+$. Furthermore, the functional $t_\lambda^{n,-}(u)$ is a local maximum for the fibering map $\phi_u$ which satisfies $t_\lambda^{n,-}(u)u\in\mathcal{N}_{\lambda}^-$;
\item[(iii)] Define the open set $U_{\lambda} = \{ u \in H_0^1(\Omega) \setminus \{0\} : \lambda > \lambda_n(u) \}$. Then the functions $u \mapsto t_\lambda^{n,+}(u)$ and $u \mapsto t_\lambda^{n,-}(u)$ are continuous on $U_{\lambda}$;
\item[(iv)] If $\lambda=\lambda_n(u)$, then the fibering map $\phi_u$ has exactly one degenerate critical point $0<t_n(u)$ satisfying $t_n(u)u\in\mathcal{N}_{\lambda}^0$. Furthermore, $t\mapsto \phi_u(t)$ is an increasing function;
\item[(v)] If $\lambda<\lambda_n(u)$, then the fibering map $\phi_u$ has no critical points. Furthermore, $t\mapsto \phi_u(t)$ is an increasing function.
\end{itemize}
\end{prop}
\begin{proof}
		The proof is similar to \cite[Proposition 2.5]{Montenegro}. Hence, we only verify the continuity asserted in item (iii).

        Since $\lambda_n$ is continuous, $U_\lambda$ is open. Let $(u_k)\subset U_\lambda$ be such that
        $u_k\to u\in U_\lambda$ in $H_0^1(\Omega)$. From the definition of $R_n$ and the compact Sobolev embeddings, the map $t\mapsto R_n(tu_k)$ converges locally uniformly to $t\mapsto R_n(tu)$ on $(0,+\infty)$. Moreover, using the fact that $\lambda>\lambda_n(u)$, the equation $R_n(tu)=\lambda$ has exactly two solutions, one in $(0,t_n(u))$ and the other in $(t_n(u),+\infty)$. These solutions are simple, because at each of them the derivative of $t\mapsto R_n(tu)$ does not vanish. Indeed, $t_n(u)$ is the unique critical point of $R_n(\cdot\,u)$ and $R_n(tu)\to+\infty$ as $t\to0^+$ and as $t\to+\infty$). Hence, by the local uniform convergence and the simplicity of the roots, the Implicit Function Theorem yields
        \[
        t_\lambda^{n,\pm}(u_k)\to t_\lambda^{n,\pm}(u).
        \]
        Therefore, the maps $u\mapsto t_\lambda^{n,\pm}(u)$ are continuous on $U_\lambda$.
	\end{proof}    
    From now on, we observe that the properties analogous to those established in Proposition \ref{compar} also hold for the fibering maps associated with the functional $R_e$. In view of this, for any fixed $u\in H_0^1(\Omega)\setminus \{0\}$, we characterize the geometric behavior of the fibering map $\phi_u$. For convenience, we consider the following graphical representations:
	\begin{figure}[!ht]
		\center{\includegraphics[scale=0.7]{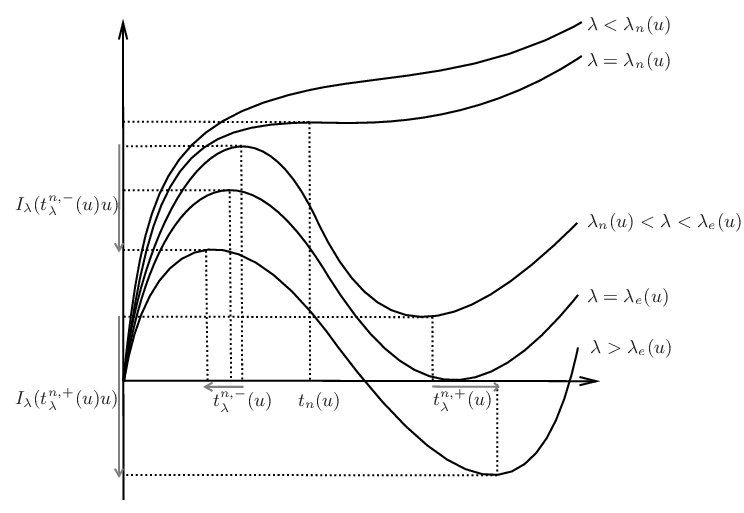}}
		\caption{Fibering geometry for the different parameter regimes.}
		\label{fig1}
	\end{figure}
	\begin{prop}\label{problema}
		Suppose  $0<\beta,p < 1$. Then the following assertions are true:
		\begin{itemize}
			\item[(i)] If $\lambda > \lambda_*$, then the sets $\mathcal{N}_{\lambda}^+$ and $\mathcal{N}_{\lambda}^-$ are not empty;
			\item[(ii)] If $\lambda = \lambda_*$, then $\mathcal{N}_{\lambda_*}=\mathcal{N}_{\lambda_*}^0$;
			\item[(iii)] If $\lambda \geq \lambda_*$, then the set $\mathcal{N}_{\lambda}^0$ is not empty; 
			\item[(iv)] If $\lambda<\lambda_*$, then $\mathcal{N}_{\lambda}$ is empty.  
		\end{itemize}	
	\end{prop}
	\begin{proof}From Lemma \ref{N0} $(iv)$, $\displaystyle\lambda_*:=\lambda_n(u)$ for some  $u\in H_0^1(\Omega)\setminus\{0\}$. Assume that $\lambda >\lambda_*=\lambda_n(u)$. It follows from Proposition \ref{compar} $(i)$ and $(ii)$ that there exists $0<t_\lambda^{n,-}(u)<t_\lambda^{n,+}(u)$ such that   $t_\lambda^{n,\pm}(u)u\in\mathcal{N}_{\lambda}^\pm$. This proves item $(i)$. 
		
		Now, we consider $\lambda=\lambda_*=\lambda_n(u)$. From Proposition \ref{compar} $(iv)$, we have that $\mathcal{N}_{\lambda_*}^0\neq \emptyset$. Assume, by contradiction, that $\mathcal{N}_{\lambda_*}^\pm\neq \emptyset$. Then there exists $v^\pm\in H_0^1(\Omega)\setminus \{0\}$ such that $R_n(v^\pm)=\lambda_*$ and
\(
\left.\frac{d}{dt}R_n(tv^\pm)\right|_{t=1}\gtrless0.
\)
Since $t\mapsto R_n(tv^\pm)$ has a global minimum $t_n(v^\pm)>0$, we conclude that $\lambda_*=R_n(v^\pm)>R_n(t_n(v^\pm)v^\pm)$. The last assertion gives us a contradiction. This finishes the proof of item $(ii)$. 
		
		Suppose that $\lambda \geq \lambda_* = \lambda_n(u)$. It follows from Lemma \ref{N0}-$(ii)$ and $(iii)$ that $\lambda_n$ is continuous and unbounded from above. According to the Intermediate Value Theorem, there exists $w \in H_0^1(\Omega) \setminus \{0\}$ such that $\lambda = \lambda_n(w)$. In particular, for $\lambda=\lambda_*$ and taking into account Lemma \ref{N0}-$(iv)$, there exists $u \in H_0^1(\Omega) \setminus \{0\}$ such that $\lambda_* = \lambda_n(u)$. In light of Proposition \ref{der-Rn} $(iii)$, we conclude that $t_n(w)w \in \mathcal{N}_\lambda^0$. In particular,  $\mathcal{N}_{\lambda}^0$ is not empty.
		
		Now, we assume that $\lambda<\displaystyle\lambda_*:=\lambda_n(u)=\inf_{v\in H_0^1(\Omega)\setminus\{0\}}\lambda_n(v)\leq R_n(tu),~\forall t>0$. Then the equation $R_n(tu)=\lambda$ has no solution. Consequently, $\mathcal{N}_\lambda=\emptyset$. This ends the proof.
	\end{proof}

	\begin{prop}\label{min-real-c+}
		Suppose $0<\beta,p < 1$ and $\lambda>\lambda_*$. Then there exists $v_\lambda\in \mathcal{N}^+_{\lambda}\cup \mathcal{N}^0_{\lambda}$ such that $$I_{\lambda}(v_\lambda)=\inf\limits_{\mathcal{N}^+_{\lambda}}I_{\lambda}=:c_{\mathcal{N}^+_{\lambda}}.$$ Moreover, 
		\begin{itemize}
			\item[(i)] For each $\lambda\in (\lambda_*, \lambda^*)$ we have $c_{\mathcal{N}^+_{\lambda}}>0$;
			\item[(ii)] For $\lambda=\lambda^*$ it follows that $c_{\mathcal{N}^+_{\lambda^*}}=0$;
			\item[(iii)] For each $\lambda>\lambda^*$ we have $c_{\mathcal{N}^+_{\lambda}}<0$.
		\end{itemize}		
    \end{prop}
\begin{proof}
    The case $\lambda=\lambda^*$ follows from Lemma \ref{gamma} (iii). Thus, for $\lambda\neq \lambda^*$, we consider a sequence $(v_k)\subset \mathcal{N}_{\lambda}^+$ which is a minimizing sequence for $c_{\mathcal{N}^+_{\lambda}}$. In view of Corollary \ref{weak-conv}, up to a subsequence, we have that $v_k\rightharpoonup v_\lambda\neq 0$. We claim that $v_k\to v_\lambda$ strongly in $H_0^1(\Omega)$. Indeed, due to Lemma \ref{N0} (ii), we deduce that $\lambda_n(v_\lambda)\leq \lambda$. Moreover, the map
    \(
    u\mapsto \phi_u'(1)
    \)
    is weakly lower semicontinuous. This fact implies that
    \(
    \phi_{v_\lambda}'(1)\leq 0.
    \)
    We shall split the proof into two cases:

	\noindent \textbf{Case 1.} $\lambda_n(v_\lambda)=\lambda$. Since $\phi_{v_\lambda}'(1)\leq0$, Remark \ref{rmk1} implies that $R_n(v_\lambda)\leq\lambda$. On the other hand,
    \[
    \lambda=\lambda_n(v_\lambda)\leq R_n(v_\lambda).
    \]
    Consequently, $R_n(v_\lambda)=\lambda=\lambda_n(v_\lambda)$. Since $t_n(v_\lambda)$ is the unique global minimum point of the map $t\mapsto R_n(tv_\lambda)$, we conclude that $t_n(v_\lambda)=1$. Thus,
    \(
    \left.\frac{d}{dt}R_n(tv_\lambda)\right|_{t=1}=0,
    \)
    and, in view of Proposition \ref{der-Rn}, we have
    \(
    \phi_{v_\lambda}''(1)=0.
    \)
    Therefore, $v_\lambda\in\mathcal{N}_\lambda^0$.

    Since $(v_k)\subset\mathcal{N}_\lambda^+$, we have
    \[
    \|v_k\|^2+\|v_k\|_{1-\beta}^{1-\beta}
    =\lambda\|v_k\|_{p+1}^{p+1}.
    \]
    By the compact Sobolev embedding and since $\Omega$ is bounded,
    \[
    \|v_k\|_{p+1}^{p+1}\to\|v_\lambda\|_{p+1}^{p+1}
    \quad\mbox{and}\quad
    \|v_k\|_{1-\beta}^{1-\beta}
    \to\|v_\lambda\|_{1-\beta}^{1-\beta}.
    \]
    Since $v_\lambda\in\mathcal{N}_\lambda^0\subset\mathcal{N}_\lambda$, it follows that
    \[
    \begin{aligned}
    \lim_{k\to\infty}\|v_k\|^2
    =
    \lambda\|v_\lambda\|_{p+1}^{p+1}
    -\|v_\lambda\|_{1-\beta}^{1-\beta}=\|v_\lambda\|^2.
    \end{aligned}
    \]
    Hence, $v_k\to v_\lambda$ strongly in $H_0^1(\Omega)$. Consequently,
    \[
    I_\lambda(v_\lambda)
    =\lim_{k\to\infty}I_\lambda(v_k)
    =c_{\mathcal{N}_\lambda^+}.
    \]

	\noindent \textbf{Case 2.} $\lambda_n(v_\lambda)<\lambda$. Since $\lambda_n(v_\lambda)<\lambda$, we have $v_\lambda\notin\mathcal{N}_\lambda$, hence $\phi_{v_\lambda}'(1)\neq0$. Combined with $\phi_{v_\lambda}'(1)\leq 0$, this gives $\phi_{v_\lambda}'(1)<0$, and therefore the fibering map $\phi_{v_\lambda}$ is strictly decreasing in a neighborhood of $t=1$. Since $R_n(v_\lambda)\leq\lambda$ and $\lambda_n(v_\lambda)<\lambda$, there exists $t_\lambda^{n,+}(v_\lambda)>1$ such that
    \(
    t_\lambda^{n,+}(v_\lambda)v_\lambda
    \in\mathcal{N}_\lambda^+.
    \)
    Using the fact that $\phi_{v_\lambda}$ is strictly decreasing on
    \(
    [1,t_\lambda^{n,+}(v_\lambda)],
    \)
    we obtain
    \[
    \begin{aligned}
    c_{\mathcal{N}_\lambda^+}
    \leq
    I_\lambda(t_\lambda^{n,+}(v_\lambda)v_\lambda)
    <
    I_\lambda(v_\lambda)
    \leq
    \liminf_{k\to\infty}I_\lambda(v_k)
    =
    c_{\mathcal{N}_\lambda^+},
    \end{aligned}
    \]
    which is a contradiction. Therefore,
    \(
    t_\lambda^{n,+}(v_\lambda)=1,
    \)
    and hence $v_\lambda\in\mathcal{N}_\lambda^+$.

    As in Case (1), using the Nehari identity together with the compact Sobolev embedding, we obtain
    $
    \|v_k\|\to\|v_\lambda\|.
    $
    Therefore, $v_k\to v_\lambda$ strongly in $H_0^1(\Omega)$ and
    $
    I_\lambda(v_\lambda)
    =c_{\mathcal{N}_\lambda^+}.
    $
     
    We now address the proof of items (i)--(iii). Consider $\lambda_*<\lambda<\lambda^*$. Given $v\in \mathcal{N}_\lambda^+$, we obtain $\lambda<\lambda^*\leq R_e(t_e(v)v)\leq R_e(v)$. This is equivalent to $I_\lambda(v)>0$. Assume, by contradiction, that $c_{\mathcal{N}^+_{\lambda}}=0$. Then there exists $(u_k)\subset \mathcal{N}_\lambda^+$ such that $I_\lambda(u_k)\to 0$. From Corollary \ref{weak-conv}, there exists $u\in H_0^1(\Omega)\backslash\{0\}$ such that $u_k\rightharpoonup u$. Since
		$$I_\lambda(u_k)=\frac{\|u_k\|_{p+1}^{p+1}}{p+1}\left[R_e(u_k)-\lambda\right]\to 0$$
		and $\|u_k\|_{p+1}\to \|u\|_{p+1}\neq0$, we conclude that $R_e(u_k)\to\lambda$. This implies that
		$$R_e(u)\leq \liminf_{k\to \infty}R_e(u_k)=\lambda<\lambda^*\leq R_e(t_e(u)u)\leq R_e(u).$$
		This is a contradiction. Therefore, $c_{\mathcal{N}^+_{\lambda}}>0$. This proves item $(i)$. The item $(ii)$ follows directly from Lemma \ref{gamma} $(iii)$. Now, assume that $\lambda>\lambda^*$. From Lemma \ref{gamma} $(iii)$, there exists $v_{\lambda^*}:=t_e(u^*)u^*\in H_0^1(\Omega)\setminus\{0\}$ such that $\lambda^*=\lambda_e(v_{\lambda^*})$ and $I_{\lambda^*}(v_{\lambda^*})=0$.
	As a consequence, since the minimum of the fibering map
$t\mapsto I_\lambda(tv_{\lambda^*})$ is attained at
$t_\lambda^{n,+}(v_{\lambda^*})$, we have
\[
c_{\mathcal{N}^+_{\lambda}}
\leq I_\lambda
\bigl(
t_\lambda^{n,+}(v_{\lambda^*})v_{\lambda^*}
\bigr)<I_{\lambda}(v_{\lambda^*})
=
I_{\lambda^*}(v_{\lambda^*})
-\frac{\lambda-\lambda^*}{p+1}
\|v_{\lambda^*}\|_{p+1}^{p+1}
<0.
\]
This finishes the proof. 		
\end{proof}
	\begin{lem}\label{c-lamb-cont}
	Suppose $0<\beta,p<1$. Then the function
	$\lambda\mapsto c_{\mathcal{N}^+_{\lambda}}$
	is left-continuous at $\lambda=\lambda^*$.
    \end{lem}
	\begin{proof}
        In view of Proposition \ref{min-real-c+}, there exists $v_{\lambda^*}\in \mathcal{N}_{\lambda^*}^+$ such that $0=c_{\mathcal{N}^+_{\lambda^*}}=I_{\lambda^*}(v_{\lambda^*})$. Since $\lambda_n(v_{\lambda^*})<\lambda<\lambda^*$, we have $v_{\lambda^*}\in U_{\lambda}$ for each $\lambda\in (\lambda_n(v_{\lambda^*}), \lambda^*)$. Then there exists $t_\lambda^{n,+}(v_{\lambda^*})\in(t_n(v_{\lambda^*}),1=t_{\lambda^*}^{n,+}(v_{\lambda^*}))$ such that $t_\lambda^{n,+}(v_{\lambda^*})v_{\lambda^*}\in \mathcal{N}_\lambda^+$. Using the fact that the map $t\mapsto I_{\lambda^*}(tv_{\lambda^*})$ is decreasing in $(t_n(v_{\lambda^*}),1)$ and $t_\lambda^{n,+}(v_{\lambda^*})\in(t_n(v_{\lambda^*}),1)$, we deduce		
        \begin{align*}
            0<c_{\mathcal{N}^+_{\lambda}}\leq& I_{\lambda}(t_\lambda^{n,+}(v_{\lambda^*})v_{\lambda^*})=I_{\lambda^*+(\lambda-\lambda^*)}(t_\lambda^{n,+}(v_{\lambda^*})v_{\lambda^*})\nonumber\\ 
            =& I_{\lambda^*}(t_\lambda^{n,+}(v_{\lambda^*})v_{\lambda^*})+\frac{\lambda^*-\lambda}{p+1}\|t_\lambda^{n,+}(v_{\lambda^*})v_{\lambda^*}\|_{p+1}^{p+1}\nonumber\\
            \leq & I_{\lambda^*}(v_{\lambda^*})+\frac{\lambda^*-\lambda}{p+1}\|v_{\lambda^*}\|_{p+1}^{p+1}\nonumber\\
            =& c_{\mathcal{N}^+_{\lambda^*}}+\frac{\lambda^*-\lambda}{p+1}\|v_{\lambda^*}\|_{p+1}^{p+1}
            =0+\frac{\lambda^*-\lambda}{p+1}\|v_{\lambda^*}\|_{p+1}^{p+1}.\nonumber
        \end{align*}
        Consequently,
        \begin{eqnarray}\label{c-cont}
            0<c_{\mathcal{N}^+_{\lambda}}=c_{\mathcal{N}^+_{\lambda}}-c_{\mathcal{N}^+_{\lambda^*}}\leq  \frac{\lambda^*-\lambda}{p+1}\|v_{\lambda^*}\|_{p+1}^{p+1},
        \end{eqnarray}
        which proves that the function $\lambda\mapsto c_{\mathcal{N}^+_{\lambda}}$ is left-continuous at $\lambda=\lambda^*$. 
\end{proof}
    
	\begin{prop}\label{min-real-N+}
		Suppose $0<\beta,p < 1$. Then there exists $\epsilon_0=\epsilon_0(\lambda^*)>0$ and  $v_\lambda\in \mathcal{N}^+_{\lambda}$ such that $c_{\mathcal{N}^+_{\lambda}}=I_\lambda(v_\lambda)$ for all $\lambda>\lambda^*-\epsilon_0$.		
	\end{prop}
	\begin{proof}
		Due to Proposition \ref{min-real-c+}, there exists $v_\lambda\in \mathcal{N}^+_{\lambda}\cup \mathcal{N}^0_{\lambda}$ such that $c_{\mathcal{N}^+_{\lambda}}=I_\lambda(v_\lambda)$. We claim that $v_\lambda\in \mathcal{N}^+_{\lambda}$ for $\lambda$ sufficiently close to $\lambda^*$. The proof follows by contradiction, assuming that $v_\lambda\in \mathcal{N}_\lambda^0$. If $\lambda\geq\lambda^*$, then from Proposition \ref{min-real-c+} we have that $c_{\mathcal{N}^+_{\lambda}}\leq 0$. On the other hand, if $v_\lambda\in\mathcal N_\lambda^0$, then Proposition~\ref{N0-dist-0} yields $I_\lambda(v_\lambda)\geq\sigma_\lambda>0$. Since $I_\lambda(v_\lambda)=c_{\mathcal N_\lambda^+}$, we obtain a contradiction. Furthermore, for $\lambda<\lambda^*$, by Lemma \ref{c-lamb-cont},
there exists $\epsilon=\epsilon(\lambda^*)$ such that if
$\lambda>\lambda^*-\epsilon$, we have that
$c_{\mathcal{N}^+_{\lambda}}
=|c_{\mathcal{N}^+_{\lambda}}-c_{\mathcal{N}^+_{\lambda^*}}|
<\sigma_{\lambda^*}$. Since \eqref{sigm-lambda} implies that
$\sigma_\lambda\geq\sigma_{\lambda^*}$ whenever
$\lambda\leq\lambda^*$, we deduce		$$c_{\mathcal{N}^+_{\lambda}}<\sigma_{\lambda^*}\leq I_\lambda(v_\lambda)=c_{\mathcal{N}^+_{\lambda}}.$$
		This is a contradiction, which finishes the proof. 		
	\end{proof}

 	\section{Analysis of the perturbed problem via Rayleigh quotients}\label{trunc}
	Inspired by \cite{Montenegro}, we introduce the following perturbed problem:
		\begin{equation}\label{P-pertubed}
			\left\{\begin{array}{ll}\tag{$P_{\lambda,\eps}$}
				-\Delta u  +g_\epsilon (u) =  \lambda |u|^{p-1}u, & \mbox{in }   \Omega,  \\
				u=0,&\mbox{on }\, \partial\Omega,
			\end{array}
			\right.
		\end{equation}
		where 
        \begin{equation}\label{pertubation_salvadora}
	g_\epsilon(u)=\displaystyle\frac{|u|^{q-1}u}{(|u|+\epsilon)^{q+\beta}},
	\quad 0<\beta<1.
\end{equation}
Although $g_\epsilon$ is not differentiable at the origin, the map
$t\mapsto g_\epsilon(t)t$ is of class $C^1$. We set
\[
\Gamma_\epsilon(t):=
\begin{cases}
g_\epsilon'(t)t, & t\neq0,\\
0, & t=0.
\end{cases}
\]
Throughout the work, we consider the following assumption:
\begin{enumerate}
    \item[$(Q)$]
    $0<q\leq
    \min\left\{
    \dfrac{p(\beta+2)}{p+2\beta+2},
    \dfrac{1-\beta}{p+\beta+2}
    \right\}<p<1$.
\end{enumerate}
Note that $g_\epsilon$ is continuous and odd, and $g_\epsilon(t)t\geq0$ for every $t\in\mathbb{R}$. Recall that a function $u\in H_0^1(\Omega)$ is said to be a weak solution of
Problem~\eqref{P-pertubed} whenever
\[
\int_\Omega \nabla u\nabla\varphi\,\mathrm{d}x = \int_\Omega\left[-g_\eps(u)+\lambda |u|^{p-1}u\right] \varphi\,\mathrm{d}x, \quad \forall\,\varphi\in H_0^1(\Omega).
\]
     The energy functional $I_{\lambda,\epsilon}: H_0^1(\Omega)\to \mathbb{R}$ associated with Problem \eqref{P-pertubed} is defined as
         $$I_{\lambda,\eps}(u)=\frac{1}{2}\|u\|^2+\int_{\Omega}G_{\epsilon}(u)\,\mathrm{d}x -\frac{\lambda}{p+1}\|u\|^{p+1}_{p+1},$$
        where $G_\epsilon(u):=\int_0^u g_\epsilon(s)\,\mathrm{d}s\geq 0$. One may check that $I_{\lambda,\epsilon}$ is a $C^1$-functional. 

Condition $(Q)$ is introduced in order to guarantee the strict
monotonicity of the auxiliary map $f_\eps$ defined in
\eqref{functionf} below, and consequently the uniqueness of the
critical point $t_{n,\eps}(u)$ of the fibering map associated with
$R_{n,\eps}$. This condition is of the same type as the restriction considered in \cite{Montenegro} and is always feasible for 
$0<\beta,p<1$.

\begin{rmk}
Since $g_\epsilon$ is odd, $G_\epsilon$ is even. Consequently,
$I_{\lambda,\epsilon}(|u|)=I_{\lambda,\epsilon}(u)$ for every
$u\in H_0^1(\Omega)$. Moreover, the perturbed Nehari constraint is
invariant under the map $u\mapsto |u|$. Hence, minimizing sequences
and constrained minimizers may be replaced by their absolute values
without changing the corresponding energy levels.

Notice, however, that the evenness of $I_{\lambda,\epsilon}$ alone
implies only that $-u$ is a critical point whenever $u$ is a critical
point; in general, it does not imply that $|u|$ is also a critical
point. In the present work, nonnegativity is instead obtained through
the variational constructions themselves. For the constrained
minimization problem, this follows from the invariance under the
modulus, whereas, for the Mountain Pass solution, the minimax level
may be computed by restricting the admissible paths to the
nonnegative cone.
\end{rmk}
		
We next derive lower and upper estimates for the perturbed nonlinear term. These estimates will play important role in the proof of our results.

\begin{lem}\label{G,g}
The following estimates hold:
\begin{equation} \label{est:g-prime}
-\beta\, g_\varepsilon(t)t
\le \Gamma_\varepsilon(t)t
\le q\, g_\varepsilon(t)t. \quad \text{and} \quad (1-\beta) G_\eps(t) \leq g_\eps(t)t \leq (q+1)G_\eps(t), \quad \forall t \in\mathbb{R}.
\end{equation}
Furthermore, the following growth condition is satisfied:
\begin{equation} \label{est:G-growth}
G_\eps(1) \min\{|t|^{q+1}, |t|^{1-\beta}\} \leq G_\eps(t) \leq \frac{1}{1-\beta} |t|^{1-\beta}, \quad \forall t \in\mathbb{R}.
\end{equation}
\end{lem}
\begin{proof} 

Note that
$
\Gamma_\varepsilon(t)=\theta(t)g_\varepsilon(t),
$
where
$
\theta(t):=\frac{q\varepsilon-\beta |t|}{|t|+\varepsilon}.
$
It is easy to check that $-\beta\leq \theta(t)\leq q$. Hence, for each \(t\ge 0\), we see that \(g_\varepsilon(t)\ge 0\) and
\begin{equation}\label{positive}
-\beta\, g_\varepsilon(t)\le g_\varepsilon'(t)t\le q\, g_\varepsilon(t), \quad \forall t\geq0.
\end{equation} 
Moreover, for each \(t\le 0\), we obtain \(g_\varepsilon(t)\le 0\) and
\[
q\, g_\varepsilon(t)\le g_\varepsilon'(t)t\le -\beta\, g_\varepsilon(t), \quad \forall t\leq0.
\]
Finally, since \(g_\varepsilon(t)t\ge 0\) for all \(t\in\mathbb{R}\), we infer that
\[
-\beta\, g_\varepsilon(t)t\le g_\varepsilon'(t)t^{2}\le q\, g_\varepsilon(t)t, \quad \forall t\in\mathbb{R}.
\]

Now, we consider the second set of estimates in \eqref{est:g-prime}. It follows from \eqref{positive} that
\[
(1-\beta)g_\varepsilon(t)\le g_\varepsilon(t)+g_\varepsilon'(t)t\le (q+1)g_\varepsilon(t), \quad \forall \, t>0,
\]
which implies that
\[
(1-\beta)g_\varepsilon(t)\le \bigl(tg_\varepsilon(t)\bigr)'\le (q+1)g_\varepsilon(t),
\qquad \forall \, t>0.
\]
Hence,
\[
(1-\beta)G_\varepsilon(t)\le t\,g_\varepsilon(t)\le (q+1)G_\varepsilon(t),
\qquad \forall \, t>0.
\]
Since \(G_\varepsilon\) and \(g_\varepsilon(t)t\) are even functions, we conclude that
\[
(1-\beta)G_\varepsilon(t)\le g_\varepsilon(t)t\le (q+1)G_\varepsilon(t),
\qquad \forall \,t\in\mathbb{R}.
\]

Now, we prove the growth condition \eqref{est:G-growth}. We observe that
$$\frac{d}{dt}\left[\frac{G_{\epsilon}(t)}{t^{q+1}}\right] = \frac{g_{\epsilon}(t)t^{q+1}-(q+1) t^{q}G_{\epsilon}(t)}{t^{2q+2}} = \frac{t^q}{t^{2q+2}}\left[g_{\epsilon}(t)t-(q+1)G_{\epsilon}(t)\right] \leq 0, \quad \forall t>0.$$
Consequently, the map $t \mapsto \frac{G_{\epsilon}(t)}{t^{q+1}}$ is non-increasing. 

Now, we consider two cases based on the value of $t$.

\noindent \textbf{Case 1:} $0 < t \leq 1$. From the monotonicity established above, it follows that
$$
\frac{G_{\epsilon}(t)}{t^{q+1}} \geq \frac{G_{\epsilon}(1)}{1^{q+1}} \implies G_{\epsilon}(t) \geq G_{\epsilon}(1)t^{q+1} = G_{\epsilon}(1)\min\{t^{1-\beta}, t^{q+1}\}, \quad \forall~ 0 < t \leq 1.
$$

\vspace{0,2cm}
\noindent \textbf{Case 2:} $t > 1$.
Taking into account that $\frac{g_{\epsilon}(s)}{G_{\epsilon}(s)} \geq \frac{1-\beta}{s}$, we integrate this inequality over the interval $[1, t]$ to obtain
$$
\int_{1}^{t}\frac{g_{\epsilon}(s)}{G_{\epsilon}(s)}\,\mathrm{d}s \geq \int_{1}^{t}\frac{1-\beta}{s}\,\mathrm{d}s \implies \ln\left(\frac{G_{\epsilon}(t)}{G_{\epsilon}(1)}\right) \geq \ln\left(t^{1-\beta}\right).
$$
This yields $G_{\epsilon}(t) \geq G_{\epsilon}(1)t^{1-\beta}$. In particular, for each $t > 1$, we have $t^{1-\beta} = \min\{t^{1-\beta}, t^{q+1}\}$. Thus,
$$
G_{\epsilon}(t) \geq G_{\epsilon}(1)t^{1-\beta} = G_{\epsilon}(1)\min\{t^{1-\beta}, t^{q+1}\}, \quad \forall t > 1.$$
Hence, in both cases we know that 
$$G_{\epsilon}(t) \geq G_{\epsilon}(1)\min\{t^{1-\beta}, t^{q+1}\}, \quad \forall t > 0.
$$
Furthermore, by using the fact that \(G_\varepsilon\) is even, we infer that 
\[
G_\varepsilon(1) \min\{|t|^{q+1},|t|^{1-\beta}\}\le G_\varepsilon(t)
\le \frac{1}{1-\beta}|t|^{1-\beta},
\qquad \forall t\in\mathbb R,
\]
which finishes the proof.
\end{proof}

\begin{rmk}

Importantly, \eqref{est:G-growth} is a uniform estimate that is preserved as $\epsilon \to 0$, since
$$\displaystyle G_{\epsilon}(1)\geq \frac{1}{2^{q+\beta}}\int_{\epsilon}^1s^{q}=\frac{1}{2^{q+\beta}}\frac{1-\epsilon^{q+1}}{q+1}\geq\frac{q}{2^{\beta+q}(q+1)}.$$ 
\end{rmk}

\begin{lem}\label{G-eps-norm1-beta} 
    Given $0<\eps<1$, we have
    $$\int_\Omega G_\eps(u)\,\mathrm{d}x\leq \frac{1}{1-\beta}\|u\|^{1-\beta}_{1-\beta}, \quad \forall \, u\in H_0^1(\Omega).$$
\end{lem}
\begin{proof} The proof follows by using  \eqref{est:G-growth}. The details are omitted. \end{proof}


Now, we consider the Nehari manifold for Problem \eqref{P-pertubed} as follows:
		\begin{equation}
				\mathcal{N}_{\lambda,\epsilon}:=\{ u \in H_0^1(\Omega)\setminus \{0\} : I_{\lambda,\epsilon}'(u)u=0 \}
				=\left\{ u\in  H_0^1(\Omega)\setminus \{0\}:\|u\|^2+\int_\Omega g_\epsilon(u)u \,\mathrm{d}x=\lambda\|u\|^{p+1}_{p+1}\right\}.
			\label{nehari-eps}
		\end{equation}
Note that $u \in \mathcal{N}_{\lambda, \epsilon}$ if and only if
		\begin{equation}\label{a0i-eps}
			\lambda = \dfrac{\displaystyle\|u\|^2 +\int_\Omega g_\eps(u)u\,\mathrm{d}x}{\|u\|^{p+1}_{p+1}}.
		\end{equation}
Arguing as before, we introduce the functionals $R_{n,\epsilon},R_{e,\eps}: H_0^1(\Omega)\setminus\{0\} \to \mathbb{R}$ defined as
	\begin{eqnarray}\label{Rs}
		R_{n,\epsilon}(u):=\dfrac{\displaystyle\|u\|^2 +\int_\Omega g_\eps(u)u\,\mathrm{d}x}{\|u\|^{p+1}_{p+1}}\qquad\mbox{and}\qquad R_{e,\eps}(u)= \dfrac{\frac{1}{2}\|u\|^2 +\displaystyle\int_\Omega G_\eps(u)\,\mathrm{d}x}{\frac{1}{p+1}\|u\|^{p+1}_{p+1}}.
	\end{eqnarray}
	Note that, for $t>0$,
\begin{align*}
\|u\|_{p+1}^{p+1}\frac{d}{dt}R_{n,\epsilon}(tu)
=\,(1-p)\|u\|^2t^{-p}
-t^{q-p-1}\int_\Omega
\frac{|u|^{q+1}\big[(p+\beta)t|u|+(p-q)\epsilon\big]}
{(t|u|+\epsilon)^{q+\beta+1}}\,\mathrm{d}x.
\end{align*}
The last identity show us that $\frac{d}{dt}R_{n,\epsilon}(tu)=0$ is equivalent to 
\begin{equation}\label{eq-crit-R}
    (1-p)\|u\|^2
    =
    \int_\Omega
    \frac{
        (p+\beta)t|u|+(p-q)\eps
    }{
        (t|u|)^{1-q}
        (t|u|+\eps)^{q+\beta+1}
    }
    |u|^2\,\mathrm{d}x.
\end{equation}

		We claim that the function 
        \begin{equation}\label{functionf}
         \displaystyle f_\eps(s):=\frac{[(p+\beta)s+(p-q)\epsilon]}{s^{1-q}(s+\epsilon)^{q+\beta+1}}=\frac{(p+\beta)s^q}{(s+\epsilon)^{q+\beta+1}}+\frac{(p-q)\epsilon}{s^{1-q}(s+\epsilon)^{q+\beta+1}}
        \end{equation}
        is decreasing for $s>0$. Indeed, taking into account that $q\leq p\frac{(\beta+2)}{p+2\beta+2}$, we deduce that 
		\begin{equation*}
			(\beta+p)q- (p-q)(\beta+2)\leq 0.
		\end{equation*}		
		 Therefore, 
		 $$f_\eps'(s)=\frac{-(p+\beta)(\beta+1)s^2+[(\beta+p)q-(p-q)(\beta+2)]\epsilon s-(p-q)(1-q)\epsilon^2}{s^{2-q}(s+\epsilon)^{q+\beta+2}}<0, \quad \forall \, s>0,$$		
	which proves the claim. Moreover, observe that
	\begin{equation}\label{limits}
    \lim_{s\to0^+}f_\eps(s)=+\infty
    \qquad\mbox{and}\qquad
    \lim_{s\to\infty}f_\eps(s)=0.
    \end{equation}
    Therefore, \eqref{eq-crit-R} admits a unique solution $t=t_{n,\epsilon}(u)$ which is a global mínimum of $t\mapsto R_{n,\epsilon}(tu)$.  Define
		\begin{eqnarray}\label{gamma-ep}
			\lambda_{n,\epsilon}(u):=R_{n,\eps}(t_{n,\epsilon}(u)u)>0,\quad \forall \, u\in H_0^1(\Omega)\setminus\{0\}.
		\end{eqnarray}
\begin{lem}\label{N0-ep}
Suppose that $(Q)$ holds, and let $\lambda_{n,\epsilon} : H_0^1(\Omega) \setminus \{0\} \to \mathbb{R}$ be defined in \eqref{gamma-ep}. Then the following assertions hold:
\begin{itemize}
\item[(i)] $\lambda_{n,\epsilon}$ is $0$--homogeneous, i.e., $\lambda_{n,\epsilon}(tu)=\lambda_{n,\epsilon}(u)$ for every $t>0$ and $u\in H_0^1(\Omega)\setminus\{0\}$;
\item[(ii)] The function $t_{n,\eps}:H_0^1(\Omega)\setminus\{0\}\to(0,\infty)$ is continuous, and  $\lambda_{n,\epsilon}\in C^1(H_0^1(\Omega)\setminus\{0\},\mathbb{R})$. Furthermore, $\lambda_{n,\epsilon}:H_0^1(\Omega)\setminus\{0\} \to\mathbb{R}$ is weakly lower semicontinuous;
\item[(iii)] The function $\lambda_{n,\epsilon}$ is unbounded from above, that is, there exists a sequence $(w_k) \subset H_0^1(\Omega) \setminus \{0\}$ such that $\lambda_{n,\eps}(w_k) \to \infty$ as $k \to \infty$;
\item[(iv)] There exists $u:=u_{\eps,*}\in H_0^1(\Omega)\setminus\{0\}$ such that $\displaystyle\lambda_{\epsilon,*}=\lambda_{n,\epsilon}(u)=\inf_{v\in H_0^1(\Omega)\setminus\{0\}}\lambda_{n,\epsilon}(v)>0$, that is, the number $\lambda_{\epsilon,*}$ is attained by some function $u \in H_0^1(\Omega) \setminus \{0\}$;
\item[(v)] Let $u_{\varepsilon,*}\in H_0^1(\Omega)\setminus\{0\}$ be a minimizer of $\lambda_{n,\varepsilon}$, and set
$w_{\varepsilon,*}:=
    t_{n,\varepsilon}(u_{\varepsilon,*})
    u_{\varepsilon,*}.$
Then
$w_{\varepsilon,*}\in\mathcal N_{\lambda_{\varepsilon,*},\varepsilon}^{0},$
where $\lambda_{\varepsilon,*}
    =
    \lambda_{n,\varepsilon}(u_{\varepsilon,*})
    =
    R_{n,\varepsilon}(w_{\varepsilon,*}).$
Moreover, $w_{\varepsilon,*}$ satisfies
\begin{equation}\label{eq:3.36}
    2\int_\Omega
    \nabla w_{\varepsilon,*}\nabla\varphi\,dx
    +
    \int_\Omega
    \bigl[
        \Gamma_\varepsilon(w_{\varepsilon,*})
        +
        g_\varepsilon(w_{\varepsilon,*})
    \bigr]\varphi\,dx
    =
    (p+1)\lambda_{\varepsilon,*}
    \int_\Omega
    |w_{\varepsilon,*}|^{p-1}
    w_{\varepsilon,*}\varphi\,dx,
\end{equation}
for every $\varphi\in H_0^1(\Omega)$.
Equivalently, it satisfies the Dirichlet problem
\[
\left\{
\begin{aligned}
    -2\Delta w_{\varepsilon,*}
    &=
    -\bigl[
        \Gamma_\varepsilon(w_{\varepsilon,*})
        +
        g_\varepsilon(w_{\varepsilon,*})
    \bigr]
    +(p+1)\lambda_{\varepsilon,*}
    |w_{\varepsilon,*}|^{p-1}
    w_{\varepsilon,*}
    &&\text{in }\Omega,
    \\
    w_{\varepsilon,*}&=0
    &&\text{on }\partial\Omega.
    \end{aligned}
    \right.
    \]
	\end{itemize}	
\end{lem}
\begin{proof}
		For (i), since $R_{n,\eps}'(tu)tu=0$ has a unique solution, we have that $R_{n,\eps}(t_{n,\epsilon}(u)u)=R_{n,\eps}\left(\frac{t_{n,\epsilon}(u)}{s}su\right)$. Thus $t_{n,\epsilon}(su)=\frac{t_{n,\epsilon}(u)}{s}$. Consequently,  
		$$\lambda_{n,\eps}(su)=R_{n,\eps}(t_{n,\epsilon}(su)su)=R_{n,\eps}\left(\frac{t_{n,\epsilon}(u)}{s}su\right)=R_{n,\eps}(t_{n,\epsilon}(u)u)=\lambda_{n,\eps}(u).$$
	
To prove (ii), we define the function $\mathcal{F}:(0,\infty) \times (H_0^1(\Omega) \setminus \{0\}) \to \mathbb{R}$ by $\mathcal{F}(t,u) = R'_{n,\eps}(tu)u$. Note that $\mathcal{F}$ is of class $C^1$ on $(0,\infty)\times(H_0^1(\Omega)\setminus\{0\})$, and
$$
\mathcal{F}(t_{n,\eps}(u),u) = R'_{n,\eps}(t_{n,\eps}(u)u)u = \frac{d}{dt}R_{n,\eps}(tu) \Big|_{t=t_{n,\epsilon}(u)} = 0.
$$
Furthermore, by using that $t_{n,\eps}(u)u$ is the global minimum of $R_{n,\eps}$, we see that 
$$
\frac{\partial}{\partial t}\mathcal{F}(t, u) \Big|_{t=t_{n,\epsilon}(u)} = \frac{d^2}{dt^2}R_{n,\eps}(tu) \Big|_{t=t_{n,\epsilon}(u)} > 0.
$$
In light of the Implicit Function Theorem, for each $u\in H_0^1(\Omega) \setminus \{0\}$, there exist $\delta > 0$ and a continuously differentiable function $\tau: B_\delta(u) \to (0,\infty)$, such that 
\[
 \mathcal{F}(\tau(v),v) = 0 \quad \mbox{and} \quad \frac{\partial}{\partial t}\mathcal{F}(t, u) \Big|_{t=\tau(u)}>0, \quad \forall \, v \in B_\delta(u).
\]
Now, for each $v \in B_\delta(u)$ and taking into account that $t_{n,\epsilon}(v)$ is the unique solution to the equation $R_{n,\eps}'(tv)v = 0$, we have $\tau(v) = t_{n,\epsilon}(v)$ in $B_\delta(u)$. Consequently, the functional $t_{n,\epsilon}$ is differentiable in $B_\delta(u)$. Since $u$ was chosen arbitrarily, we conclude that $t_{n,\epsilon}$ is differentiable in $H_0^1(\Omega) \setminus \{0\}$. It follows that $\lambda_{n,\eps}(u):=R_{n,\eps}(t_{n,\epsilon}(u)u)$ is differentiable in $H_0^1(\Omega) \setminus \{0\}$.

 Now, we prove that $\lambda_{n,\epsilon}: H_0^1(\Omega)\setminus\{0\} \to \mathbb{R}$ is weakly lower semicontinuous. It is already known that for each sequence $(u_k) \in H^1_0(\Omega)$ such that $u_k \rightharpoonup u \neq 0$ in $H^1_0(\Omega)$, we have $R_{n,\epsilon}(u) \leq \liminf_{k\to\infty} R_{n,\epsilon}(u_k)$. Suppose that $t_k := t_{n,\epsilon}(u_k) \to t_0 \neq 0$. Consequently, $t_k u_k \rightharpoonup t_0 u$ and
\begin{equation}
    \lambda_{n,\epsilon}(u) = R_{n,\epsilon}(t_{n,\epsilon}(u)u) \leq R_{n,\epsilon}(t_0 u) \leq \liminf_{k\to\infty} R_{n,\epsilon}(t_k u_k) = \liminf_{k\to\infty} \lambda_{n,\epsilon}(u_k).
\end{equation}

It remains to consider the cases $t_k \to 0$ and $t_k \to \infty$. Assuming that $t_k \to 0$, we obtain that $t_k u_k \to 0$ in $H_0^1(\Omega)$. Now, using the fact that  
\[
\displaystyle \int_{\Omega} \frac{|u_k|^{q+1}}{(t_k|u_k| + \eps)^{q+\beta}} \, dx \to \frac{1}{\eps^{q+\beta}}\int_{\Omega} |u|^{q+1} \, dx \neq 0,
\]
we conclude 
\begin{equation*}
\|u_k\|_{p+1}^{p+1} R_{n,\epsilon}(t_k u_k) \geq t_k^{q-p} \int_{\Omega} \frac{|u_k|^{q+1}}{(t_k|u_k| + \eps)^{q+\beta}} \, dx \to \infty.
\end{equation*}
Now, assume that $t_k \to \infty$. In particular, we obtain 
\begin{equation*}
\|u_k\|_{p+1}^{p+1} R_{n,\epsilon}(t_k u_k) \geq t_k^{1-p} \|u_k\|^2 \to \infty,
\end{equation*}
which finishes the proof of item (ii).
		
Now, we prove item (iii). Let $x_0\in \Omega$ and $\rho_0>0$ be such that $B_\rho(x_0)\subset \Omega$ for all $0<\rho<\rho_0$. Choose a test function $\phi\in C_c^\infty(B_1(0))$, with $\phi\not\equiv 0$, such that $\operatorname{supp}(\phi)\subset B_1(0)$ and $||\phi||_{\infty}=1$. We define
\begin{equation}
    w_\rho(x):=\phi\!\left(\frac{x-x_0}{\rho}\right).
\end{equation}
It is not hard to see that $w_\rho\in C_c^\infty(\Omega)\subset H_0^1(\Omega)$ and $w_\rho\not\equiv 0$. We claim that
\begin{equation}
    \lambda_{n,\varepsilon}(w_\rho)\to\infty, \quad \text{as } \rho\downarrow 0,
\end{equation} 
uniformly in $\epsilon\in (0,1)$. To see this, note that 
 \[
  g_\varepsilon(s)s\geq \frac{|s|^{q+1}}{(|s|+1)^{q+\beta}}, \quad \forall\; 0<\epsilon<1.
\]
As a consequence, for every $t>0$, we have
\begin{equation}
    R_{n,\varepsilon}(tw_\rho)\geq R_{n,1}(tw_\rho)
    =  \frac{t^2\|w_\rho\|^2+\displaystyle\int_\Omega \frac{t^{q+1}|w_\rho|^{q+1}}{(t|w_\rho|+1)^{q+\beta}}\,\mathrm{d}x}{t^{p+1}\|w_\rho\|_{p+1}^{p+1}}.
\end{equation}
Since $t_{n,\varepsilon}(w_\rho)$ minimizes the map $t\mapsto R_{n,\varepsilon}(tw_\rho)$, it is enough to estimate the term $\inf_{t>0} R_{n,1}(tw_\rho)$. For simplicity, we set the following constants:
\begin{equation}
    A:=\|\nabla \phi\|_2^2, \quad B:=\|\phi\|_{p+1}^{p+1}, \quad D:=\|\phi\|_{q+1}^{q+1}.
\end{equation}
By a change of variables, we observe that
\begin{equation}
    \|w_\rho\|^2=\rho^{N-2}A \quad \text{and} \quad \|w_\rho\|_{p+1}^{p+1} = \rho^N B.
\end{equation}
It also follows that
\begin{equation}
    \int_\Omega \frac{t^{q+1}|w_\rho|^{q+1}}{(t|w_\rho|+1)^{q+\beta}}\,\mathrm{d}x = \rho^N \int_{B_1} \frac{t^{q+1}|\phi(y)|^{q+1}}{(t|\phi(y)|+1)^{q+\beta}}\,\mathrm{d}y.
\end{equation}
{Thus, a direct computation yields}
\begin{equation}
    R_{n,1}(tw_\rho) = \frac{A}{B}\rho^{-2}t^{1-p} +\frac{1}{B} \int_{B_1} \frac{t^{q-p}|\phi|^{q+1}}{(t|\phi|+1)^{q+\beta}}\,\mathrm{d}y.
\end{equation}
We divide the argument into two cases.

\medskip
\noindent \textbf{Case 1: $0<t\le 1$.} \noindent In this case, we have $t|\phi(y)|+1\le 2$ for all $y\in B_1(0)$. Hence,
\begin{equation}
    \frac{1}{B} \int_{B_1} \frac{t^{q-p}|\phi|^{q+1}}{(t|\phi|+1)^{q+\beta}}\,\mathrm{d}y \ge \frac{t^{q-p}}{B\;2^{q+\beta}} \int_{B_1} |\phi|^{q+1}\,\mathrm{d}y = C_1 t^{q-p},
\end{equation}
where $C_1:=D[B\,2^{q+\beta}]^{-1}>0$. Consequently,
\begin{equation}
    R_{n,1}(tw_\rho) \ge \frac{A}{B}\rho^{-2}t^{1-p}+C_1 t^{q-p}.
\end{equation}
Since $0<q<p<1$, we have $1-p>0$ and $q-p<0$. Now, we consider the auxiliary function $f_\rho(t):=\frac{A}{B}\rho^{-2}t^{1-p}+C_1 t^{q-p}$. A straightforward minimization shows that
\begin{equation}
    \min_{t>0} f_\rho(t) = C_2 \rho^{-\frac{2(p-q)}{1-q}},
\end{equation}
where $C_2>0$ is a constant {independent of $\rho$}. Therefore,
\begin{equation}
    \inf_{0<t\leq 1} R_{n,1}(tw_\rho) \ge C_2 \rho^{-\frac{2(p-q)}{1-q}}.
\end{equation}

\medskip
\noindent \textbf{Case 2: $t\ge 1$.} Now, we observe that
\begin{equation}
    R_{n,1}(tw_\rho) \geq \frac{A}{B}\rho^{-2}t^{1-p} \ge \frac{A}{B}\rho^{-2}.
\end{equation}
Thus, there exists $C_3>0$, {uniform with respect to $\rho$}, such that
\begin{equation}
    \inf_{t\geq 1} R_{n,1}(tw_\rho)\ge C_3 \rho^{-2}.
\end{equation}

From now on, using the previous estimates, we obtain
\begin{equation}
    \lambda_{n,\epsilon}(w_\rho) = \inf_{t>0} R_{n,\varepsilon}(tw_\rho) \geq \inf_{t>0} R_{n,1}(tw_\rho)   \geq \min\left\{ C_2 \rho^{-\frac{2(p-q)}{1-q}}, \, C_3 \rho^{-2} \right\},
\end{equation}
which implies that $\lambda_{n,\varepsilon}(w_\rho)\to\infty$ as $\rho\downarrow 0$ uniformly for any $0<\epsilon<1.$ This finishes the proof of item (iii).

\medskip
\noindent \textbf{Proof of (iv).}
We divide the argument into three steps.

\medskip
\noindent \textbf{Step 1.} Boundedness of the minimizing sequence.
We mention that $\lambda_{n,\eps}(u) > 0$ for all $u \in H_0^1(\Omega) \setminus \{0\}$. Let $(v_k)$ be a minimizing sequence for $\lambda_{n,\eps}$. We claim that the sequence $(\tilde{v}_k) := (t_{n,\eps}(v_k)v_k)$ is bounded. Indeed, taking into account that $R_{n,\epsilon}(\tilde{v}_k) \leq \lambda_{\epsilon,*} + 1$ for $k\in\mathbb{N}$ sufficiently large, we have
		\begin{equation}
			\Vert \tilde{v}_k \Vert^{1-p} \leq S_{p+1}^{p+1}(\lambda_{\epsilon,*} + 1), \quad k \gg 1.
		\end{equation}
This ends the proof of the claim.

\medskip
\noindent \textbf{Step 2.} Nontriviality of the weak limit.
By Step 1, there exists $v\in H_0^1(\Omega)$ such that, up to a subsequence, $\tilde{v}_k\rightharpoonup v$ in $H_0^1(\Omega)$. We shall prove that $v \neq 0$. The proof proceeds by contradiction. Assume that $v = 0.$ Thus we may suppose that
\(
\tilde v_k\to0
\text{ in }L^{p+1}(\Omega)
\)
as $k \to \infty$. Since the sequence $(\lambda_k):= (R_{n,\epsilon}(\tilde{v}_k))$ is bounded and
\[
\|\tilde v_k\|^2
+\int_\Omega g_\epsilon(\tilde v_k)\tilde v_k\,\mathrm{d}x
=
\lambda_k\|\tilde v_k\|_{p+1}^{p+1},
\]
with $g_\epsilon(t)t\geq0$, we obtain
\(
\|\tilde v_k\|^2
\leq
\lambda_k\|\tilde v_k\|_{p+1}^{p+1}
\longrightarrow0.
\)
Hence, $\tilde v_k\to0$ strongly in $H_0^1(\Omega)$. It follows from $R'_{n,\epsilon}(\tilde{v}_k)\tilde{v}_k = 0$ and \eqref{eq-crit-R} that
\begin{equation}\label{tilde-v-k}
	(1-p)\|\tilde{v}_k\|^2
	=
	\int_{\Omega}
	\frac{(p+\beta)|\tilde{v}_k|+(p-q)\epsilon}
	{|\tilde{v}_k|^{1-q}(|\tilde{v}_k|+\epsilon)^{q+\beta+1}}
	|\tilde{v}_k|^2\,\mathrm{d}x.
\end{equation}   
 Define
\(
a_k:=\|\tilde v_k\|
\mbox{ and }
w_k:=\frac{\tilde v_k}{a_k}.
\)
In particular, $a_k\to0$ and $\|w_k\|=1$. Since
$R_{n,\epsilon}'(\tilde v_k)\tilde v_k=0$, we have
\[
(1-p)a_k^2
=
\int_\Omega
g_\epsilon(a_kw_k)a_kw_k
\left[
(p-q)
+
(q+\beta)
\frac{a_k|w_k|}{a_k|w_k|+\epsilon}
\right]
\,\mathrm{d}x.
\]
Consequently,
\begin{equation}\label{g-ak}
\int_\Omega
g_\epsilon(a_kw_k)a_kw_k\,\mathrm{d}x
\leq C a_k^2.
\end{equation}
Now, we consider the set
\(
E_k:=\{x\in\Omega:a_k|w_k|\leq\epsilon\}.
\)
On $E_k$, we have
\[
g_\epsilon(a_kw_k)a_kw_k
\geq C_\epsilon a_k^{q+1}|w_k|^{q+1}.
\]
Hence, by using \eqref{g-ak}, we see that 
\[
\int_{E_k}|w_k|^{q+1}\,\mathrm{d}x
\leq C_\epsilon a_k^{1-q}.
\]
On the other hand, using $\|w_k\|=1$ and the Sobolev embedding,
\[
\int_{\Omega\setminus E_k}|w_k|^{q+1}\,\mathrm{d}x
\leq
C_\epsilon a_k^{2^*-(q+1)}
\leq
C_\epsilon a_k^{1-q}.
\]
Therefore,
\begin{equation}\label{wk-q}
\|w_k\|_{q+1}^{q+1}
\leq C_\epsilon a_k^{1-q}.
\end{equation}
Now, by the interpolation inequality using $L^{q+1}(\Omega)$ and $L^{2^*}(\Omega)$, we deduce that 
\[
\|w_k\|_{p+1}^{p+1}
\leq
C
\left(
\|w_k\|_{q+1}^{q+1}
\right)^\gamma,
\qquad
\gamma:=
\frac{2^*-(p+1)}{2^*-(q+1)}.
\]
In view of \eqref{wk-q} we obtain that
\[
\|w_k\|_{p+1}^{p+1}
\leq
C_\epsilon a_k^{\gamma(1-q)}.
\]
In light of the fact that $(\lambda_k)$ is bounded and
\(
a_k^2
\leq
\lambda_k a_k^{p+1}\|w_k\|_{p+1}^{p+1},
\)
we obtain
\[
a_k^{1-p}
\leq
C\|w_k\|_{p+1}^{p+1}
\leq
C_\epsilon a_k^{\gamma(1-q)}.
\]
However,
\(
\gamma(1-q)-(1-p)
=
\frac{(2^*-2)(p-q)}
{2^*-(q+1)}
>0,
\)
which is impossible as $a_k\to0$. Therefore, $v\neq0$.

\medskip
\noindent \textbf{Step 3.} Attainment of $\lambda_{\eps,*}$.
Since $u \mapsto \lambda_{n,\epsilon}(u)$ is a weakly lower semicontinuous functional, we deduce that
\[
\lambda_{\eps,*}
\leq \lambda_{n,\eps}(v)
\leq \liminf_{k\to\infty}\lambda_{n,\eps}(\tilde v_k)
= \liminf_{k\to\infty}\lambda_{n,\eps}(v_k)
= \lambda_{\eps,*}.
\]
Therefore,
\(
\lambda_{n,\eps}(v)=\lambda_{\eps,*}>0,
\)
where we used the fact that $v\neq0$. This finishes the proof of item (iv).

\medskip
\noindent \textbf{Proof of (v).}
Since $u_{\varepsilon,*}$ minimizes $\lambda_{n,\varepsilon}$, we have
$\lambda_{n,\varepsilon}'(u_{\varepsilon,*})\varphi=0$ for every $\varphi\in H_0^1(\Omega)$. Let
$w_{\varepsilon,*}=
    t_{n,\varepsilon}(u_{\varepsilon,*})
    u_{\varepsilon,*}.$
By the definition of $t_{n,\varepsilon}$, we have
$R_{n,\varepsilon}'(w_{\varepsilon,*})
    w_{\varepsilon,*}=0.$
Differentiating
$\lambda_{n,\varepsilon}(u)
    =
    R_{n,\varepsilon}
    \bigl(t_{n,\varepsilon}(u)u\bigr)$
at $u=u_{\varepsilon,*}$, we conclude that
$R_{n,\varepsilon}'(w_{\varepsilon,*})\varphi=0$ for every $\varphi\in H_0^1(\Omega)$. Furthermore, we observe that
\[
    R_{n,\varepsilon}(w)
    =
    \frac{
        \|w\|^2
        +
        \displaystyle\int_\Omega g_\varepsilon(w)w\,dx
    }{
        \|w\|_{p+1}^{p+1}
    }.
\]
A simple differentiation gives
\[
\begin{aligned}
0&=R_{n,\varepsilon}'(w_{\varepsilon,*})\varphi
    =
    \frac{1}{\|w_{\varepsilon,*}\|_{p+1}^{p+1}}
    \Bigg[
        2\int_\Omega
        \nabla w_{\varepsilon,*}\nabla\varphi\,dx
        +
        \int_\Omega
        \bigl(
            \Gamma_\varepsilon(w_{\varepsilon,*})
            +
            g_\varepsilon(w_{\varepsilon,*})
        \bigr)\varphi\,dx
        -(p+1)\lambda_{\varepsilon,*}
        \int_\Omega
        |w_{\varepsilon,*}|^{p-1}
        w_{\varepsilon,*}\varphi\,dx
    \Bigg],
\end{aligned}
\]
which implies \eqref{eq:3.36}. Finally,
$R_{n,\varepsilon}'(w_{\varepsilon,*})
    w_{\varepsilon,*}=0
$
is equivalent to
$w_{\varepsilon,*}\in
    \mathcal N_{\lambda_{\varepsilon,*},\varepsilon}^{0}.$
This completes the proof.		
	\end{proof}

Now, we define
\[
\lambda_{e,\eps}(u) :=\min_{t>0} R_{e,\eps}(t u), \quad \forall \,u \in H_0^1(\Omega) \setminus \{0\}.
\]
The result below is a version of Lemma \ref{N0-ep} to $\lambda_{e,\eps}(u)$.

\begin{lem}\label{gamma-e-ep}
	Assume that $(Q)$ is satisfied. Then the following statements hold:
	\begin{itemize}
		\item[(i)] For every $u\in H_0^1(\Omega)\setminus\{0\}$, there exists
		a unique $t_{e,\eps}(u)>0$ such that
		\[
		\lambda_{e,\eps}(u)
		=
		R_{e,\eps}(t_{e,\eps}(u)u)
		=
		\min_{t>0}R_{e,\eps}(tu).
		\]
		Moreover, $\lambda_{e,\eps}$ is $0$--homogeneous,
		$\lambda_{e,\eps}\in
		C^1(H_0^1(\Omega)\setminus\{0\},\mathbb R)$, and
		$\lambda_{e,\eps}$ is weakly lower semicontinuous and bounded
		from below;
		\item[(ii)] $\lambda_{e,\eps}(u)>\lambda_{n,\eps}(u)$ for every $u\in H_0^1(\Omega)\setminus\{0\}$;
		\item[(iii)] There exists
		$u=u_{e,\eps}\in H_0^1(\Omega)\setminus\{0\}$ such that
		\[
		\lambda_{\eps}^*
		=
		\lambda_{e,\eps}(u)
		=
		\inf_{v\in H_0^1(\Omega)\setminus\{0\}}
		R_{e,\eps}(v)>0.
		\]
		In particular,
		\(
		0<\lambda_{\eps,*}<\lambda_{\eps}^*;
		\)
		\item[(iv)] Let $u$ be a critical point of
		$\lambda_{e,\eps}$ and set
		\(
		\lambda:=\lambda_{e,\eps}(u),
		~
		w:=t_{e,\eps}(u)u.
		\)
		Then
		\(
		w\in\mathcal N_{\lambda,\eps}^+
		\)
		and $w$ is a weak solution of Problem~\eqref{P-pertubed}.
		In particular, for the minimizer $u=u_{e,\eps}$ obtained in
		(iii), Problem~\eqref{P-pertubed} possesses a weak solution in
		$\mathcal N_{\lambda_{\eps}^*,\eps}^+$.
	\end{itemize}
\end{lem}
\begin{proof}
The proof follows the same arguments as in Lemma~\ref{N0-ep}. We only
point out the main differences.
For $s>0$, set
\[
\Psi_\eps(s):=(p+1)G_\eps(s)-g_\eps(s)s,
\qquad
K_\eps(s):=\frac{\Psi_\eps(s)}{s^2}.
\]
A direct computation gives
\[
\frac{d}{dt}R_{e,\eps}(tu)=0
\quad\Longleftrightarrow\quad
\frac{1-p}{2}\|u\|^2
=
\int_\Omega K_\eps(t|u|)|u|^2\,\mathrm{d}x.
\]
Moreover, $(Q)$ says that 
\(
K_\eps'(s)<0,
~s>0,
\)
while
\(
K_\eps(s)\to+\infty
~\text{as }s\to0^+,
~
K_\eps(s)\to0
~\text{as }s\to+\infty.
\)
Therefore, there exists a unique $t_{e,\eps}(u)>0$ such that
\[
\lambda_{e,\eps}(u)
=
R_{e,\eps}(t_{e,\eps}(u)u)
=
\min_{t>0}R_{e,\eps}(tu).
\]
The $0$--homogeneity, the $C^1$--regularity, the weak lower
semicontinuity, and the attainment of $\lambda_{\eps}^*$ follow
exactly as in Lemma~\ref{N0-ep}.

Now, we use the identity
\begin{equation}\label{Rn-Re-eps}
R_{n,\eps}(tu)-R_{e,\eps}(tu)
=
\frac{t}{p+1}
\frac{d}{dt}R_{e,\eps}(tu),
\qquad t>0,
\end{equation}
which holds for every $u\in H_0^1(\Omega)\setminus\{0\}$ and every $t>0$.
Arguing as in the proof of the corresponding result for the
unperturbed quotients, we obtain
\(
t_{n,\eps}(u)<t_{e,\eps}(u),
\)
and consequently
\(
\lambda_{n,\eps}(u)
<
\lambda_{e,\eps}(u),
~
u\in H_0^1(\Omega)\setminus\{0\}.
\)
In particular, we see that 
\(
0<\lambda_{\eps,*}<\lambda_{\eps}^*.
\)

Finally, we assume that $u$ is a critical point of $\lambda_{e,\eps}$ and
\(
w:=t_{e,\eps}(u)u,
~
\lambda:=\lambda_{e,\eps}(u).
\)
Hence, by using the same argument employed in Lemma~\ref{N0-ep}, we deduce that 
\(
R_{e,\eps}'(w)=0.
\)
Together with \eqref{Rn-Re-eps}, this implies that
\(
w\in\mathcal N_{\lambda,\eps}^{+}
\)
and that $w$ is a weak solution of Problem~\eqref{P-pertubed}.
In particular, using the fact that $u=u_{e,\eps}$ is a minimizer, we have
$\lambda=\lambda_{\eps}^*$. This ends the proof. 
\end{proof}

Now, we establish some properties of the functionals associated with the  problems \eqref{P-pertubed} and \eqref{P}.

\begin{lem}\label{dep-cont-eps}
	Assume that $(Q)$ is satisfied. If $u \in H_0^1(\Omega) \setminus \{0\}$, then the maps 
    \begin{equation} \label{all-maps-dec}
        \epsilon \mapsto t_{n,\epsilon}(u), t_{e,\epsilon}(u), R_{e,\epsilon}(u), R_{n,\epsilon}(u), \lambda_{n,\epsilon}(u), \lambda_{e,\epsilon}(u),\lambda_{\eps,*},\lambda_{\eps}^*
    \end{equation}
    are strictly decreasing in $\epsilon>0$. In particular, for all $\eps>0$ the following properties hold:
    \begin{enumerate}
        \item[(i)] $t_{n,\epsilon}(u) < t_n(u)$ and $t_{e,\epsilon}(u) < t_e(u)$;
        \item[(ii)] $R_{n,\epsilon}(u) < R_{n}(u)$ and $R_{e,\epsilon}(u) < R_{e}(u)$;
        \item[(iii)] $\lambda_{n,\epsilon}(u) < \lambda_{n}(u)$ and $\lambda_{e,\epsilon}(u) < \lambda_{e}(u)$;
        \item[(iv)] $0 < \lambda_{\epsilon,*} < \lambda_{*}$ and $0 < \lambda_{\epsilon}^* < \lambda^{*}$.
    \end{enumerate}
    Furthermore, the maps defined by \eqref{all-maps-dec} are right-continuous at $\epsilon = 0$.
\end{lem}
\begin{proof}
We divide the proof into two parts.

\medskip
Part 1. Monotonicity in $\eps$.
\noindent Consider $f_\epsilon$ defined by \eqref{functionf}. Since 
	\begin{equation}
		\frac{\partial f_\epsilon(t)}{\partial \epsilon} = -\frac{(q+\beta)t^{q-1} \left[ (p+\beta+1)t + (p-q)\epsilon \right]}{(t+\epsilon)^{q+\beta+2}} < 0,
	\end{equation}
	it follows that $\eps\mapsto f_\epsilon(t)$ is decreasing for all $\epsilon \geq 0$ and $t>0$. Consequently,  by using \eqref{eq-crit-R} with $0\leq\eps_1<\eps_2$, we see that  $t_{n,\eps_1}(u) > t_{n,\epsilon_2}(u)$. 
	
	It is straightforward to show that $R_{n,\epsilon_1}(u) > R_{n,\eps_2}(u)$ for $0\leq \eps_1<\eps_2$. In view of the fact that $t_{n,\epsilon_i}(u)$ is the unique critical point (global minimum) of the map $t \mapsto R_{n,\epsilon_i}(tu)$, $i=1,2$, we have
	\begin{equation}\label{lamb-ep}
		\lambda_{n,\epsilon_2}(u) := R_{n,\epsilon_2}(t_{n,\epsilon_2}(u)u) < R_{n,\epsilon_2}(t_{n,\epsilon_1}(u)u)< R_{n,\epsilon_1}(t_{n,\epsilon_1}(u)u) =: \lambda_{n,\epsilon_1}(u).
	\end{equation}
	In view of Lemma \ref{N0-ep}, there exists $u_{\eps_1,*} \in
H_0^1(\Omega)\setminus\{0\}$ such that $\lambda_{\eps_1,*} = \lambda_{n,\eps_1}(u_{\eps_1,*})$. Now, choosing $u = u_{\eps_1,*}$ in \eqref{lamb-ep}, we deduce that 
	\begin{equation}
		0 < \lambda_{\epsilon_2,*} \leq \lambda_{n,\epsilon_2}(u_{\eps_1,*}) < \lambda_{n,\eps_1}(u_{\eps_1,*}) = \lambda_{\eps_1,*}, \quad \forall \, 0\leq\eps_1<\eps_2.
	\end{equation}

\medskip
\noindent Part 2. Right-continuity at $\eps=0$.
\noindent At this stage, we prove that the map $\epsilon \mapsto t_{n,\epsilon}(u)$
is of class $C^1$ in $(0,1)$. Indeed, for
$u\in H_0^1(\Omega)\setminus\{0\}$ define the function
$h:(0,\infty)\times(0,1)\to\mathbb{R}$ by
\begin{equation}
	h(t,\epsilon)
	:=
	\int_{\Omega} f_\eps(t|u|)|u|^2 \,\mathrm{d}x
	-(1-p)\|u\|^2.
\end{equation}
Moreover, the partial derivative with respect to $t$ is given by
\begin{equation}
	\frac{\partial h}{\partial t}(t,\epsilon)
	=
	\int_{\Omega}f_\eps'(t|u|)|u|^3 \,\mathrm{d}x.
\end{equation}
Recall also that
\(
q\leq\frac{p(\beta+2)}{p+2\beta+2}.
\)
Therefore, 
\(
\frac{\partial h}{\partial t}(t,\epsilon)<0,
~ t>0.
\)
In light of the Implicit Function Theorem
\cite[Theorem~4.2.1]{DrabekMilota}, the map
$\epsilon\mapsto t_{n,\epsilon}(u)$ is of class $C^1$ in $(0,1)$.
Furthermore, taking into account that $\epsilon\mapsto t_{n,\epsilon}(u)$ is decreasing
and $t_{n,\epsilon}(u)<t_n(u)$, there exists $\bar t\leq t_n(u)$ such
that
\(
t_{n,\epsilon}(u)\to\bar t
~\mbox{as }\epsilon\to0^+.
\)
Passing to the limit in \eqref{eq-crit-R}, we obtain
\[
(1-p)\|u\|^2
=
\int_{\Omega}f_0(\bar t|u|)|u|^2\,\mathrm{d}x.
\]
By the uniqueness of the solution of \eqref{eq-crit-R} for
$\epsilon=0$, we conclude that $\bar t=t_n(u)$. Consequently,
\(
t_{n,\epsilon}(u)\to t_n(u)
~\mbox{as }\epsilon\to0^+.
\)
It is straightforward to show that the maps
$\epsilon\mapsto R_{n,\epsilon}(u)$ and
$\epsilon\mapsto\lambda_{n,\epsilon}(u)$ are also continuous at
$\epsilon=0$.

Now, we prove that $\displaystyle \lim_{\eps\to 0^+}\lambda_{\eps,*}=\lambda_*$. In view of Lemma~\ref{N0-ep}~(iv), there exists $u_{\eps,*} \neq 0$ such that $\lambda_{\eps,*} = \lambda_{n,\eps}(u_{\eps,*})$ and $1 = t_{n,\eps}(u_{\eps,*}) < t_n(u_{\eps,*})$. Since $t_{n,\eps}(u_{\eps,*}) = 1$, we have $u_{\eps,*} \in \mathcal{N}_{\lambda_{\eps,*},\eps}^0$ and \(
\|u_{\eps,*}\|^{1-p}
\leq
\lambda_{\eps,*}S_{p+1}^{p+1}
\leq
\lambda_*S_{p+1}^{p+1}
\). Consequently, $u_{\eps,*} \rightharpoonup u_*$ weakly in $H_0^1(\Omega)$. Furthermore, by using the fact that $0<\lambda_{\eps,*}\leq\lambda_*$, the proof of
Proposition~\ref{N-way-origin-eps} (ii), with $\lambda_*$ in place of
$\lambda$, yields
$\|u_{\eps,*}\|_{p+1}\geq C_{\lambda_*}>0$. The last assertion together with the compact Sobolev embedding ensures that $u_*\not\equiv0$. Up to a subsequence, $u_{\eps,*} \to u_*$ a.e. in $\Omega$ with $|u_{\eps,*}| \leq h_r \in L^r(\Omega)$. Furthermore, by the monotonicity proved above and the fact that $\lambda_{\eps,*}<\lambda_*$ for every $\eps>0$, there exists $\tilde{\lambda}\leq\lambda_*$ such that
$\lambda_{\eps,*}\to\tilde{\lambda}$ as $\eps\to0^+$. Hence, by using the Dominated Convergence Theorem and Fatou's Lemma, we infer that
    \begin{equation}
        \int_{\Omega} |\nabla u_*|^2\,\mathrm{d}x
        \leq \liminf_{\eps \to 0^+} \int_{\Omega} |\nabla u_{\eps,*}|^2 \,\mathrm{d}x
        = \int_{\Omega} (\tilde\lambda |u_*|^{p+1} - |u_*|^{1-\beta}) \,\mathrm{d}x.
    \end{equation}
    This assertion implies that $R_n(u_*) \leq \tilde \lambda\leq \lambda_*\leq R_n(u_*)$. Therefore, $\lambda_* =\tilde\lambda= R_n(u_*)$, and $u_{\eps,*} \to u_*$ strongly in $H_0^1(\Omega)$. 
    As a consequence, we see that 
\begin{align}
    0 < \lambda_* - \lambda_{\eps,*}
    &\leq \lambda_{n}(u_{\eps,*})
    -\lambda_{n,\eps}(u_{\eps,*}) \nonumber\\
    &= R_n(t_n(u_{\eps,*})u_{\eps,*})
    -R_{n,\eps}(t_{n,\eps}(u_{\eps,*})u_{\eps,*}) \nonumber\\
    &\leq
    R_n(t_{n,\eps}(u_{\eps,*})u_{\eps,*})
    -R_{n,\eps}(t_{n,\eps}(u_{\eps,*})u_{\eps,*}) \nonumber\\
    &=\displaystyle
    \frac{1}{
    \|u_{\eps,*}\|_{p+1}^{p+1}
    }
    \int_{\Omega}
    \left(
    |u_{\eps,*}|^{1-\beta}
    -g_{\eps}(u_{\eps,*})u_{\eps,*}
    \right)\,\mathrm{d}x
        \to0,
    \qquad\mbox{as }\eps\to0^+.
\end{align}
The proof of the corresponding properties for $t_{e,\eps}(u)$,
$R_{e,\eps}(u)$, $\lambda_{e,\eps}(u)$, and $\lambda_{\eps}^*$ follows
by the same arguments. In fact,  we just apply Lemma~\ref{gamma-e-ep} and the identity
\[
\frac{1-p}{2}\|w\|^2
=
\int_\Omega
\big[(p+1)G_\eps(w)-g_\eps(w)w\big]\,\mathrm{d}x,
\qquad
w=t_{e,\eps}(u)u.
\]
In particular, the same compactness argument applied to the minimizers
of $\lambda_{e,\eps}$ yields $\lambda_{\eps}^*\to\lambda^*$ as
$\eps\to0^+$.
\end{proof}

\begin{prop}\label{compar-ep}
Assume that $0<\beta<1$ and that condition $(Q)$ holds. Then, for every
$u\in H_0^1(\Omega)\setminus\{0\}$, the following alternatives hold:
\begin{enumerate}
    \item[(i)] Assume that $
    \lambda>\lambda_{n,\eps}(u)=R_{n,\eps}\bigl(t_{n,\eps}(u)u\bigr)$. Then
    the fibering map
    $
        \phi_{\eps,u}(t)=I_{\lambda,\eps}(tu),
    $
    for $t>0$, has exactly two distinct critical points
    $
        0<t_{\lambda,\eps}^-(u)<t_{n,\eps}(u)<t_{\lambda,\eps}^+(u).
    $
    Moreover, the following properties hold:
    \begin{enumerate}
        \item[(i.1)] The point $t_{\lambda,\eps}^+(u)$ is a local minimum of
        the fibering map $\phi_{\eps,u}$, and
        $
            t_{\lambda,\eps}^+(u)u\in \mathcal{N}_{\lambda,\eps}^+.
        $
        On the other hand, $t_{\lambda,\eps}^-(u)$ is a local maximum of
        $\phi_{\eps,u}$, and
        $
            t_{\lambda,\eps}^-(u)u\in \mathcal{N}_{\lambda,\eps}^-.
        $

        \item[(i.2)] Consider the set
        $
            U_{\lambda,\eps}:=
            \left\{
                u\in H_0^1(\Omega)\setminus\{0\}:
                \lambda>\lambda_{n,\eps}(u)
            \right\}.
        $
        Then the maps
        $
            u\mapsto t_{\lambda,\eps}^+(u)$ and 
            $u\mapsto t_{\lambda,\eps}^-(u)
        $
        belong to $C^1\bigl(U_{\lambda,\eps},\mathbb{R}\bigr)$.
    \end{enumerate}

    \item[(ii)] Suppose that
    $
        \lambda=\lambda_{n,\eps}(u).
    $
    Then the fibering map $\phi_{\eps,u}$ has exactly one critical point,
    namely $t_{n,\eps}(u)>0$. This critical point is degenerate and satisfies
    $
        t_{n,\eps}(u)u\in \mathcal{N}_{\lambda,\eps}^0.
    $
    Furthermore, the map $t\mapsto \phi_{\eps,u}(t)$ is nondecreasing on $(0,\infty)$.

    \item[(iii)] Assume that
    $
        \lambda<\lambda_{n,\eps}(u).
    $
    Then the fibering map $\phi_{\eps,u}$ has no critical points. Moreover,
    the map $t\mapsto \phi_{\eps,u}(t)$ is increasing on $(0,\infty)$.
\end{enumerate}

\noindent In addition, for each $0\leq\eps_1\leq\eps_2$, we have
$U_{\lambda,\eps_1}\subset U_{\lambda,\eps_2}$. Moreover, for every
$u\in U_{\lambda,\eps_1}$,
\[
t_{\lambda,\eps_2}^-(u)
\leq t_{\lambda,\eps_1}^-(u)
<t_{\lambda,\eps_1}^+(u)
\leq t_{\lambda,\eps_2}^+(u).
\]
\end{prop}
\begin{proof}
	Initially, we point out that
	\[
	\lim_{t\to0^+}R_{n,\eps}(tu)
	=
	\lim_{t\to+\infty}R_{n,\eps}(tu)
	=
	+\infty.
	\]
	Moreover, the map $t\mapsto R_{n,\eps}(tu)$ has a unique critical
	point $t_{n,\eps}(u)$, which is a global minimum. Notice also that
	\[
	R_{n,\eps}(tu)-\lambda
	=
	\frac{\phi_{\eps,u}'(t)}
	{t^p\|u\|_{p+1}^{p+1}},
	\qquad t>0.
	\]
	Consequently, the critical points of $\phi_{\eps,u}$ are precisely
	the solutions of $R_{n,\eps}(tu)=\lambda$. Thus, the assertions in
	$(i)$--$(iii)$ follow from the fibering geometry of
	$t\mapsto R_{n,\eps}(tu)$. Furthermore, by using that  $R_{n,\eps}\in
	C^1(H_0^1(\Omega)\setminus\{0\},\mathbb{R})$ and \( \frac{d}{dt}R_{n,\eps}(tu)\neq0 \)
	at $t=t_{\lambda,\eps}^{\pm}(u)$, the Implicit Function Theorem
	implies that the maps
	$u\mapsto t_{\lambda,\eps}^{\pm}(u)$ belong to
	$C^1(U_{\lambda,\eps},\mathbb{R})$.
Furthermore, the last assertions follow from the monotonicity of
	$\eps\mapsto R_{n,\eps}(tu)$. This ends the proof.
\end{proof}
\begin{prop}\label{problema-eps}
		Suppose that $0 < \beta < 1$ and that condition $(Q)$ holds.
		\begin{itemize}
			\item[(i)] If $\lambda > \lambda_{\eps,*}$, then $\mathcal{N}_{\lambda,\eps}^+$ and $\mathcal{N}_{\lambda,\eps}^-$ are not empty;
			\item[(ii)] If $\lambda = \lambda_{\eps,*}$, then $\mathcal{N}_{\lambda_{\eps,*},\eps}=\mathcal{N}_{\lambda_{\eps,*},\eps}^0$;
			\item[(iii)] If $\lambda \geq \lambda_{\eps,*}$, then $\mathcal{N}_{\lambda,\eps}^0$ is not empty;
			\item[(iv)] If $\lambda<\lambda_{\eps,*}$ then $\mathcal{N}_{\lambda,\eps}$ is empty.  
		\end{itemize}	
\end{prop}
\begin{proof}
    The proof follows the same ideas employed in the proof of Proposition \ref{problema}. The details are omitted.     
\end{proof}

\begin{prop}\label{N-way-origin-eps}
	    Suppose that $0 < \beta < 1$ and that condition $(Q)$ holds. Then:
 \begin{itemize}
     \item[(i)] $\mathcal{N}_{\lambda,\varepsilon}$ is bounded independently of $\varepsilon$;
     \item[(ii)] Assume that $\lambda \geq \lambda_{\eps,*}$. Then for $\eps>0$ sufficiently small there exists $\delta_\lambda>0$, independent of $\eps$, such that \( \|u\|,\|u\|_{p+1}\geq\delta_\lambda, ~u\in\mathcal N_{\lambda,\eps}. \)
     \item[(iii)] Suppose that $u\neq0$ and $I_{\lambda,\eps}(u)=0$. Hence, for each $\eps>0$ sufficiently small, \( \|u\|,\|u\|_{p+1}\geq\tilde\delta_\lambda, \) where $\tilde\delta_\lambda>0$ is independent of $\eps$.
 \end{itemize}        
\end{prop}
\begin{proof}
Let us assume that $u\in\mathcal{N}_{\lambda,\varepsilon}$. Then
	 \[
	  \Vert u\Vert^{2}+\int_\Omega g_{\varepsilon}(u)u\,\mathrm{d}x=\lambda\Vert u\Vert_{p+1}^{p+1},
	 \]
	which implies that 
	 \(
	  \Vert u\Vert^{1-p}\leq \lambda S_{p+1}^{p+1}.
	 \)
This proves item (i). 

Now, we prove item (ii). Fix $r>2$ such that
$H_0^1(\Omega)\hookrightarrow L^r(\Omega)$. Choose $\eta>0$ such that
\(
\lambda 2^{q+\beta}\eta^{p+\beta}\leq\frac12,
\)
where $0<\epsilon\leq\eta$. On the set $\{|u|\leq\eta\}$, we see that 
\[
\lambda |u|^{p+1}
\leq
\frac12
\frac{|u|^{q+1}}{(|u|+\epsilon)^{q+\beta}}
=
\frac12 g_\epsilon(u)u.
\]
On the other hand, on $\{|u|>\eta\}$,
\(
|u|^{p+1}
\leq
\eta^{p+1-r}|u|^r.
\)
Consequently, by the Sobolev embedding,
\[
\begin{aligned}
I_{\lambda,\epsilon}'(u)u
=
\|u\|^2
+\int_\Omega g_\epsilon(u)u\,\mathrm{d}x
-\lambda\|u\|_{p+1}^{p+1}
\geq
\|u\|^2
-
C\lambda\eta^{p+1-r}\|u\|^r.
\end{aligned}
\]
Hence, there exists $\delta_\lambda>0$, independent of $\epsilon$, such that
\(
I_{\lambda,\epsilon}'(u)u>0,
~
0<\|u\|<\delta_\lambda,
\)
for every $0<\epsilon\leq\eta$. Therefore,
\(
\|u\|\geq\delta_\lambda,
~
u\in\mathcal{N}_{\lambda,\epsilon}.
\)
In view of the fact that $u\in\mathcal{N}_{\lambda,\epsilon}$, we obtain
\[
\lambda\|u\|_{p+1}^{p+1}
=
\|u\|^2+\int_\Omega g_\epsilon(u)u\,\mathrm{d}x
\geq\|u\|^2.
\]
Hence,
\(
\|u\|_{p+1}
\geq
\left(\frac{\delta_\lambda^2}{\lambda}\right)^{\frac1{p+1}}.
\)
Note that for $0 < \lambda \leq \Lambda$, the constants above can be chosen independently of $\lambda$ by using $\Lambda$ instead of $\lambda$ in the previous estimates.

Now, the proof of item (iii) follows from the same argument used in item (ii).
Indeed, by using the fact that \(g_\eps(t)t\leq(q+1)G_\eps(t)\)
and $q<p$, the same estimates yield
\[
I_{\lambda,\eps}(u)
\geq
\frac12\|u\|^2-C\|u\|^r,
\qquad r>2.
\]
Hence, for each $u\neq0$ and $I_{\lambda,\eps}(u)=0$, we infer that
$\|u\|\geq\tilde\delta_\lambda$. Moreover,
\[
\frac{\lambda}{p+1}\|u\|_{p+1}^{p+1}
=
\frac12\|u\|^2+\int_\Omega G_\eps(u)\,\mathrm{d}x,
\]
which also gives
$\|u\|_{p+1}\geq\tilde\delta_\lambda$, after decreasing
$\tilde\delta_\lambda$ if necessary.
\end{proof}

\begin{prop}\label{N0-dist-0-eps}
Suppose that $0<\beta<1$ and that condition $(Q)$ holds.
For every fixed $\lambda>0$, there exists $\sigma_\lambda>0$,
independent of $\eps$, such that
\[
    I_{\lambda,\eps}(u)\geq\sigma_\lambda,
\]
for every $0<\eps\ll1$ satisfying \( \lambda\geq\lambda_{\eps,*} \) and every $u\in\mathcal N_{\lambda,\eps}^0$.
\end{prop}

\begin{proof}
Fix $u \in \mathcal{N}_{\lambda,\eps}^0$. Then $u$ satisfies
\begin{equation}
\begin{cases}
    \displaystyle
    \|u\|^2
    +
    \int_\Omega g_\eps(u)u \,\mathrm{d}x
    =
    \lambda \|u\|_{p+1}^{p+1},
    \\[2mm]
    \displaystyle
    \|u\|^2
    +
    \int_\Omega
    \left[
        q-(q+\beta)\frac{|u|}{|u|+\eps}
    \right]
    g_\eps(u)u \,\mathrm{d}x
    =
    p\lambda \|u\|_{p+1}^{p+1}.
\end{cases}
\end{equation}
As a consequence, we obtain
\begin{equation}\label{n-n0}
    (1-p)\|u\|^2
    =
    \int_{\Omega}
    g_{\eps}(u)u
    \left[
        (p-q)
        +
        (q+\beta)\frac{|u|}{|u|+\eps}
    \right]
    \,\mathrm{d}x.
\end{equation}
Since $0\leq |u|/(|u|+\eps)\leq 1,$ it follows from \eqref{n-n0} that
\begin{equation}\label{j1}
    (1-p)\|u\|^2
    \leq
    (p+\beta)
    \int_{\Omega}g_{\eps}(u)u\,\mathrm{d}x.
\end{equation}

On the other hand, by using the fact that 
$u\in\mathcal{N}_{\lambda,\eps}$, we have
\begin{align}
    I_{\lambda,\eps}(u)
    &=
    \left(
        \frac{1}{2}
        -
        \frac{1}{p+1}
    \right)\|u\|^2
    +
    \int_\Omega
    \left[
        G_\eps(u)
        -
        \frac{1}{p+1}g_\eps(u)u
    \right]
    \,\mathrm{d}x.
    \label{energy-N0-1}
\end{align}
In light of \eqref{n-n0} and \eqref{energy-N0-1}, we obtain
\begin{align}
    I_{\lambda,\eps}(u)
    &=
    \int_\Omega
    \Bigg[
        G_\eps(u)
        -
        \frac{
            p-q+2
            +
            (q+\beta)\dfrac{|u|}{|u|+\eps}
        }{2(p+1)}
        g_\eps(u)u
    \Bigg]
    \,\mathrm{d}x.
    \label{energy-N0-exact}
\end{align}
Since both $G_\eps(s)$ and $g_\eps(s)s$ are even functions of $s$, it is enough to consider $s\geq0$. For $s>0$, we deduce
\begin{align}
&\frac{d}{ds}
\left[
G_\eps(s)
-
\frac{
p-q+2
+
(q+\beta)\dfrac{s}{s+\eps}
}{2(p+1)}
g_\eps(s)s
\right]
=
\frac{g_\eps(s)}{2(p+1)}
\Bigg[
(1-q)(p-q)
\left(
\frac{\eps}{s+\eps}
\right)^2
\nonumber\\
&\qquad\qquad
+
\big[
p(\beta+2)-q(p+2\beta+2)
\big]
\frac{s\eps}{(s+\eps)^2}
+
(1+\beta)(p+\beta)
\left(
\frac{s}{s+\eps}
\right)^2
\Bigg].
\label{derivative-energy-N0}
\end{align}
In view of assumption $(Q)$, we have $q<p$ and $p(\beta+2)-q(p+2\beta+2)\geq0$. Hence,
$
(1-q)(p-q)>0
$ and $(1+\beta)(p+\beta)>0.$
Using the inequality
\[
\left(
\frac{\eps}{s+\eps}
\right)^2
+
\left(
\frac{s}{s+\eps}
\right)^2
\geq\frac12,
\]
we infer from \eqref{derivative-energy-N0} that
\begin{align}
&\frac{d}{ds}
\left[
G_\eps(s)
-
\frac{
p-q+2
+
(q+\beta)\dfrac{s}{s+\eps}
}{2(p+1)}
g_\eps(s)s
\right]
\geq
\frac{
\min\{
(1-q)(p-q),(1+\beta)(p+\beta)
\}
}{4(p+1)}
g_\eps(s).
\label{derivative-energy-N0-lower}
\end{align}
Integrating
\eqref{derivative-energy-N0-lower} over $(0,s)$ and using Lemma \ref{G,g}, we obtain
\begin{align}
&G_\eps(s)
-
\frac{
p-q+2
+
(q+\beta)\dfrac{s}{s+\eps}
}{2(p+1)}
g_\eps(s)s
\geq
\frac{
\min\{
(1-q)(p-q),(1+\beta)(p+\beta)
\}
}{
4(p+1)(q+1)
}
g_\eps(s)s.
\label{energy-integrand-final}
\end{align}
Hence, from \eqref{j1}, \eqref{energy-N0-exact} and \eqref{energy-integrand-final}, we obtain
\begin{align}
    I_{\lambda,\eps}(u)
    &\geq
    \frac{
        (1-p)
        \min\left\{
            (1-q)(p-q),(1+\beta)(p+\beta)
        \right\}
    }{
        4(p+1)(q+1)(p+\beta)
    }
    \|u\|^2.
    \label{I-N0-positive}
\end{align}
It follows from Proposition~\ref{N-way-origin-eps} (ii) that
\[
    \|u\|^2\geq\frac{\delta_\lambda^{2}}{2},
    \qquad 0<\eps\leq1,
\]
where $\delta_\lambda>0$ is independent of $\eps$. Therefore,
\[
    I_{\lambda,\eps}(u)
    \geq
    \frac{
        (1-p)
        \min\left\{
            (1-q)(p-q),(1+\beta)(p+\beta)
        \right\}
    }{
        8(p+1)(q+1)(p+\beta)
    }
    \delta_\lambda^2
    :=\sigma_\lambda>0,
\]
for all $u\in\mathcal N_{\lambda,\eps}^0.$ Moreover, by the uniform choice in Proposition~\ref{N-way-origin-eps}, for each $0<\lambda\leq\Lambda$, $\sigma_\lambda$ can be replaced by
a constant $\sigma_\Lambda>0$ independent of both $\lambda$ and $\eps$.
\end{proof}

\begin{prop} \label{fechada-eps}	
			Suppose that $0 < \beta < 1$ and that condition $(Q)$ holds. Assume also that $\lambda\geq\lambda_{\eps,*}$. Then the Nehari manifolds $\mathcal{N}_{\lambda,\eps}$ and $\mathcal{N}_{\lambda,\eps}^0$ are closed. 
		\end{prop}
		\begin{proof}
			 The functionals
				$F_1, F_2:H_0^1(\Omega)\to \mathbb{R}$ defined as 
				$$F_1(u)=\|u\|^2+\int_\Omega g_\epsilon(u)u\,\mathrm{d}x-\lambda\|u\|_{p+1}^{p+1}
                \qquad\mbox{and}\qquad
                F_2(u)=\|u\|^2+\int_\Omega \Gamma_\epsilon(u)u\,\mathrm{d}x-p\lambda\|u\|_{p+1}^{p+1}$$
				are continuous and $\mathcal{N}_{\lambda,\eps}\cup \{0\}=F_{1}^{-1}(\{0\})$ and $\mathcal{N}_{\lambda,\eps}^0\cup \{0\}=F_{1}^{-1}(\{0\})\cap F_{2}^{-1}(\{0\})$. Since $F_1$ and $F_2$ are continuous, the sets $F_{1}^{-1}(\{0\})$ and $F_{1}^{-1}(\{0\})\cap F_{2}^{-1}(\{0\})$ are closed. In view of Proposition \ref{N-way-origin-eps}, the sets $\mathcal{N}_{\lambda,\eps}$ and $\mathcal{N}_{\lambda,\eps}^0$ are bounded away from zero. As a result, we conclude that $\mathcal{N}_{\lambda,\eps}$ and $\mathcal{N}_{\lambda,\eps}^0$ are closed.		
        \end{proof}

\begin{prop}\label{min-eps}
Assume that $0<\beta<1$ and that $(Q)$ holds. Let
$\lambda>\lambda_{\eps,*}$ and $0<\eps\ll1$. Then there exists $v_{\lambda,\eps}
    \in
    \mathcal N_{\lambda,\eps}^+
    \cup
    \mathcal N_{\lambda,\eps}^0$
such that $I_{\lambda,\eps}(v_{\lambda,\eps})
    =
    c_{\mathcal N_{\lambda,\eps}^+}.$
Moreover, any minimizing sequence for
$c_{\mathcal N_{\lambda,\eps}^+}$ admits a subsequence converging
strongly in $H_0^1(\Omega)$ to $v_{\lambda,\eps}$.
\end{prop}  

\begin{proof}
Let $(u_k)\subset\mathcal N_{\lambda,\eps}^+$ be a minimizing
sequence. By Proposition~\ref{N-way-origin-eps}, $(u_k)$ is bounded
and bounded away from zero. Hence, up to a subsequence,
$u_k\rightharpoonup u$ weakly in $H_0^1(\Omega)$ and
$u_k\to u$ strongly in $L^{p+1}(\Omega)$. Moreover, for fixed $\eps>0$,
\(
g_\eps(u_k)u_k\to g_\eps(u)u
~\text{and}~
G_\eps(u_k)\to G_\eps(u)
~\text{in }L^1(\Omega).
\)
Indeed, by using the fact that $0<1-\beta<p+1$, the sequence
$(u_k)$ is uniformly bounded in $L^{1-\beta}(\Omega)$. Therefore, together with
the strong convergence in $L^{p+1}(\Omega)$ and the pointwise
convergence, we obtain that $L^1$-convergence by using the 
Dominated Convergence Theorem. Furthermore, Proposition~\ref{N-way-origin-eps} and the strong
convergence in $L^{p+1}(\Omega)$ imply that $u\neq0$.

Passing to the limit in the nonlinear terms and using the weak lower
semicontinuity of the norm gives $I_{\lambda,\eps}'(u)u\leq0$. Assume that the
inequality is strict. Then $R_{n,\eps}(u)<\lambda$ and 
$t_{\lambda,\eps}^+(u)>1$ is well defined. Since the fibering map is
decreasing on $[1,t_{\lambda,\eps}^+(u)]$, we have
\[
c_{\mathcal N_{\lambda,\eps}^+}
    \leq
    I_{\lambda,\eps}
    \bigl(t_{\lambda,\eps}^+(u)u\bigr)
    <
    I_{\lambda,\eps}(u)
    \leq
    \liminf_{k\to\infty}
    I_{\lambda,\eps}(u_k)
    =
    c_{\mathcal N_{\lambda,\eps}^+},
\]
which is not possible. Therefore, $I_{\lambda,\eps}'(u)u=0$. Using the Nehari identities for $u_k$ and $u$, together with the
strong convergence of the nonlinear terms, we obtain
$\|u_k\|\to\|u\|$. Hence $u_k\to u$ strongly in $H_0^1(\Omega)$.
Moreover, by the strong convergence obtained above and the continuity of
\[
u\;\;\mapsto\;\;
\|u\|^2+\int_\Omega\Gamma_\eps(u)u\,\mathrm{d}x
-p\lambda\|u\|_{p+1}^{p+1},
\]
we obtain
\(
u\in\mathcal N_{\lambda,\eps}^+
\cup\mathcal N_{\lambda,\eps}^0.
\)
As a result, setting $v_{\lambda,\eps}:=u$, we have
\(
I_{\lambda,\eps}(v_{\lambda,\eps})
=
c_{\mathcal N_{\lambda,\eps}^+}.
\)
\end{proof}
We aim to prove that the minimizers obtained in Proposition \ref{min-eps} solve Problem \eqref{P-pertubed}. For this purpose, we state the following lemma.

\begin{lem}\label{min-sol}
Suppose that $0 < \beta < 1$ and that condition $(Q)$ holds. Let $u$ be a local minimizer of $I_{\lambda,\epsilon}$ on $\mathcal{N}_{\lambda,\epsilon}^+$ or $\mathcal{N}_{\lambda,\epsilon}^-$. Then $u$ is a critical point of $I_{\lambda,\epsilon}$ in $H_0^1(\Omega)$, i.e., $I_{\lambda,\epsilon}'(u)\psi = 0$ for all $\psi \in H_0^1(\Omega)$.
\end{lem}
\begin{proof}
By the Lagrange Multiplier Theorem, there exists $\mu\in\mathbb{R}$ such that
\[
I_{\lambda,\eps}'(u)=\mu\mathcal{J}_{\lambda,\eps}'(u),
\]
where $\mathcal{J}_{\lambda,\eps}(u)=I_{\lambda,\eps}'(u)u$. Once
$\mathcal{J}_{\lambda,\eps}'(u)u\neq0$ and 
$u\in\mathcal{N}_{\lambda,\eps}^{+}\cup\mathcal{N}_{\lambda,\eps}^{-}$,
we obtain $\mu=0$. Hence, $I_{\lambda,\eps}'(u)=0$.
\end{proof}

\subsection{First solution: Minimization on \texorpdfstring{$\mathcal{N}_{\lambda, \eps}^+$}{mathcal{N}\_{lambda, eps}\^+}} In this section, we show that Problem \eqref{P-pertubed} admits a solution in $\mathcal{N}_{\lambda, \epsilon}^+$. For this purpose, we ensure that the map $\lambda \mapsto c_{\mathcal{N}^+_{\lambda, \epsilon}}$ is left-continuous at $\lambda = \lambda_{\eps}^*$.
\begin{lem}\label{c-lamb-cont-eps}
    Suppose that $0 < \beta < 1$ and that condition $(Q)$ holds. Then the function $\lambda \mapsto c_{\mathcal{N}^+_{\lambda, \epsilon}}$ is left-continuous at $\lambda = \lambda_{\eps}^*$.
\end{lem}
\begin{proof}
   The proof is analogous to the proof of Lemma \ref{c-lamb-cont}.
\end{proof}

\begin{prop}\label{min-real-c+-eps}
	Suppose that $0 < \beta < 1$ and that condition $(Q)$ holds. Then there exists $v_{\lambda,\eps}\in \mathcal{N}^+_{\lambda,\eps}\cup \mathcal{N}^0_{\lambda,\eps}$ such that $I_{\lambda,\eps}(v_{\lambda,\eps})=\inf\limits_{\mathcal{N}^+_{\lambda,\eps}}I_{\lambda,\eps}=:c_{\mathcal{N}^+_{\lambda,\eps}}$. Moreover, 
		\begin{itemize}
			\item[(i)] For each $\lambda\in (\lambda_{\eps,*}, \lambda_{\eps}^*)$ we obtain that $c_{\mathcal{N}^+_{\lambda,\eps}}>0$;
			\item[(ii)] For $\lambda=\lambda_{\eps}^*$ it follows that $c_{\mathcal{N}^+_{\lambda_{\eps}^*,\eps}}=0$;
			\item[(iii)] For each $\lambda>\lambda_{\eps}^*$ we obtain also that $c_{\mathcal{N}^+_{\lambda,\eps}}<0$.
		\end{itemize}		
\end{prop}
\begin{proof}
   The proof is similar to the proof of Proposition \ref{min-real-c+}.
\end{proof}

\begin{prop}\label{eps-0}
    Suppose that $0 < \beta < 1$ and that condition $(Q)$ holds. There exists $\delta_0=\delta_0(\lambda^*)>0$ such that for all $\lambda\geq\lambda^*-\delta_0$, we have that $c_{\mathcal{N}_{\lambda,\epsilon}^+}< c_{\mathcal{N}_{\lambda}^+}$ for all $\eps>0$.
\end{prop}
\begin{proof}
    In view of Proposition \ref{min-real-N+}, there exist $\delta_0=\delta_0(\lambda^*)>0$ and $v_\lambda\in \mathcal{N}_\lambda^+$ such that $c_{\mathcal{N}_{\lambda}^+}=I_\lambda(v_\lambda)$ for all $\lambda>\lambda^*-\delta_0$. Due to Lemma \ref{dep-cont-eps}, $R_{n,\eps}(u)<R_{n}(u)$ for $u\in H_0^1(\Omega)\setminus \{0\}$. This implies that $t_{\lambda,\eps}^+(v_\lambda)>1$. Consequently,
    $$c_{\mathcal{N}_{\lambda,\epsilon}^+}\leq I_{\lambda,\epsilon}(t_{\lambda,\eps}^+(v_\lambda)v_\lambda)< I_{\lambda,\epsilon}(v_\lambda)<I_{\lambda}(v_\lambda)=c_{\mathcal{N}_{\lambda}^+},$$
    which finishes the proof.
\end{proof}

In \cite{Montenegro}, the authors obtained solutions in parameter regimes
where the corresponding energies have distinct signs. The following
proposition is crucial for obtaining distinct solutions whose energies
may have the same sign.

\begin{prop}\label{min-real-N-eps}
		Assume $0 < \beta < 1$ and that $(Q)$ is satisfied.         Then there exists $\delta = \delta(\lambda^*) > 0$ such that for every $\lambda > \lambda^* - \delta$ there exists $v_{\lambda,\epsilon} \in \mathcal{N}^+_{\lambda,\epsilon}$ satisfying $c_{\mathcal{N}^+_{\lambda,\epsilon}} = I_{\lambda,\epsilon}(v_{\lambda,\epsilon})$ and $I_{\lambda,\epsilon}'(v_{\lambda,\epsilon}) = 0$, for $0 < \epsilon \ll 1$.
\end{prop}
\begin{proof}
    We divide the proof into two cases: $\lambda \geq \lambda_{\eps}^*$ and $\lambda^* - \delta < \lambda < \lambda_{\eps}^*$, for some $\delta = \delta(\lambda^*)$.
    \vspace{0,2cm}
    
    \nd {\bf Case 1: $\lambda\geq \lambda_{\eps}^*$.} In view of Propositions \ref{min-eps} and \ref{min-real-c+-eps}, for each $\lambda \geq \lambda_{\eps}^*$, there exists $v_{\lambda,\epsilon} \in \overline{\mathcal{N}_{\lambda,\epsilon}^+} \subset \mathcal{N}_{\lambda,\epsilon}^+ \cup \mathcal{N}_{\lambda,\epsilon}^0$ such that $c_{\mathcal{N}_{\lambda,\epsilon}^+} = I_{\lambda,\epsilon}(v_{\lambda,\epsilon}) \leq 0$. In particular, if $\lambda=\lambda_{\eps}^*$, then $c_{\mathcal{N}_{\lambda_{\eps}^*,\epsilon}^+} = I_{\lambda_{\eps}^*,\epsilon}(v_{\lambda_{\eps}^*,\epsilon})=0$; see Proposition \ref{min-real-c+-eps}. On the other hand, by Proposition \ref{N0-dist-0-eps}, we have $I_{\lambda,\epsilon}(v)\geq \sigma_{\lambda}>0$  for every $v \in \mathcal{N}_{\lambda,\epsilon}^0$, which implies that $v_{\lambda,\epsilon} \in \mathcal{N}_{\lambda,\epsilon}^+$.
   \vspace{0,2cm}

\nd {\bf Case 2: $\lambda^* - \delta < \lambda < \lambda_{\eps}^*$.} By Proposition \ref{eps-0}, we obtain $\delta_1>0$, uniform with respect to $\epsilon>0$, such that $c_{\mathcal{N}_{\lambda,\epsilon}^+} < c_{\mathcal{N}_{\lambda}^+}$ for all $\lambda > \lambda^* - \delta_1$. Furthermore, by Lemma \ref{dep-cont-eps}, there exists $\epsilon_1 > 0$ such that for each $0<\eps\leq \eps_1$ we have $\lambda_{\eps}^* > \lambda^* - \delta_1$. Thus, assuming that $\lambda^* - \delta_1 < \lambda < \lambda_{\eps}^*$, it follows from Propositions \ref{min-real-c+}, \ref{min-real-c+-eps} and \ref{eps-0} that
\begin{equation}\label{des-N0-C-ep}
    0 < I_{\lambda,\epsilon}(v_{\lambda,\epsilon}) = c_{\mathcal{N}_{\lambda,\epsilon}^+} = c_{\mathcal{N}_{\lambda,\epsilon}^+} - c_{\mathcal{N}_{\lambda_{\eps}^*,\epsilon}^+} < c_{\mathcal{N}_{\lambda}^+}= c_{\mathcal{N}_{\lambda}^+} - c_{\mathcal{N}_{\lambda^*}^+}, \quad \mbox{for } \lambda^* - \delta_1 < \lambda < \lambda_{\eps}^*.
\end{equation}
Furthermore, by Propositions~\ref{min-real-c+} and
\ref{N0-dist-0-eps}, together with Lemma~\ref{c-lamb-cont}, there
exists $\delta_2=\delta_2(\lambda^*)>0$ such that, for every $\lambda\in(\lambda^*-\delta_2,\lambda^*)$, Proposition~\ref{N0-dist-0-eps}, with
$\Lambda=\lambda^*$, yields
\begin{equation}\label{des-N0-C}
    c_{\mathcal{N}_{\lambda}^+} - c_{\mathcal{N}_{\lambda^*}^+} = c_{\mathcal{N}_{\lambda}^+} < {\frac{1}{2}\sigma_{\lambda^*} \leq I_{\lambda,\epsilon}(v), \quad \forall v \in \mathcal{N}_{\lambda,\epsilon}^0}.
\end{equation}
Arguing as before, in view of Lemma \ref{c-lamb-cont-eps}, we obtain $\epsilon_2 > 0$ such that $\lambda_{\eps}^* > \lambda^* - \delta_2$ for all $0 < \epsilon \leq \epsilon_2$. It follows from \eqref{des-N0-C-ep} and \eqref{des-N0-C} that
\[
0 < I_{\lambda,\epsilon}(v_{\lambda,\epsilon})
= c_{\mathcal{N}_{\lambda,\epsilon}^+}
< c_{\mathcal{N}_{\lambda}^+}
< \frac{1}{2}\sigma_{\lambda^*}
\leq I_{\lambda,\epsilon}(v), \quad
\forall v \in \mathcal{N}_{\lambda,\epsilon}^0, \quad
\lambda \in (\lambda^* - \delta, \lambda^*],
\]
where $\delta=\min\{\delta_1,\delta_2,(\lambda^*-\lambda_*)/2\}$ and $0<\eps\leq \min\{\eps_1,\eps_2\}$. This allows us to conclude that $v_{\lambda,\epsilon} \notin \mathcal{N}_{\lambda,\epsilon}^0$ for all $\lambda > \lambda^* - \delta$. Therefore, $v_{\lambda,\epsilon} \in \mathcal{N}_{\lambda,\epsilon}^+$. Finally, by Lemma~\ref{min-sol}, $I_{\lambda,\epsilon}'(v_{\lambda,\epsilon})=0$.
\end{proof}

\begin{cor}\label{C-monotone-eps}
    Assume $0 < \beta < 1$ and that $(Q)$ is satisfied. There exists $\delta = \delta(\lambda^*) > 0$ such that, for every $\lambda > \lambda^* - \delta$, the map $\epsilon \mapsto c_{\mathcal{N}_{\lambda,\epsilon}^+}$ is strictly decreasing for $\epsilon \geq 0$ sufficiently small.
\end{cor}
\begin{proof}
    In view of Proposition \ref{min-real-N-eps} (or Proposition \ref{min-real-N+} if $\eps=0$), for $0 \leq \epsilon \ll 1$, there exists $v_{\lambda,\epsilon} \in \mathcal{N}_{\lambda,\epsilon}^+$ such that $c_{\mathcal{N}_{\lambda,\epsilon}^+} = I_{\lambda,\epsilon}(v_{\lambda,\epsilon})$ for all $\lambda >\lambda^*-\delta$. Given $\eps'<\eps''$, it follows from Lemma \ref{dep-cont-eps} that 
    $\lambda_{n,\eps''}(v_{\lambda,\eps'})<\lambda_{n,\eps'}(v_{\lambda,\eps'})$. Hence, it follows from Proposition \ref{compar-ep} that $1=t_{\lambda,\eps'}^+(v_{\lambda,\eps'})<t_{\lambda,\eps''}^+(v_{\lambda,\eps'})$, when $\lambda_{n,\eps'}(v_{\lambda,\eps'})<\lambda$. Therefore, 
    $$c_{\mathcal{N}_{\lambda,\epsilon''}^+}\leq I_{\lambda,\eps''}(t_{\lambda,\eps''}^+(v_{\lambda,\eps'})v_{\lambda,\eps'})<I_{\lambda,\eps''}(t_{\lambda,\eps'}^+(v_{\lambda,\eps'})v_{\lambda,\eps'})<I_{\lambda,\eps'}(t_{\lambda,\eps'}^+(v_{\lambda,\eps'})v_{\lambda,\eps'})=I_{\lambda,\eps'}(v_{\lambda,\eps'})=c_{\mathcal{N}_{\lambda,\epsilon'}^+}.$$
    This completes the proof.
\end{proof}

\subsection{Second solution: A mountain pass solution}\label{MPCap} In this section, we solve Problem \eqref{P-pertubed} by applying the Mountain Pass Theorem. Initially, we observe that there exists a weak solution in $\mathcal{N}_{\lambda,\eps}^+$. The main objective is to obtain an additional solution. In contrast to the methods in \cite{Montenegro}, we use the results from Proposition \ref{min-real-N-eps} taking into account the framework established in \cite{Silva}. 

\begin{lem}\label{GMP1}
Given $\lambda>0$, there exists $\rho_0>0$, independent of $\eps$, such that, for every $0<\rho\leq\rho_0$ and $0<\eps<1$,
\[
I_{\lambda,\eps}(u)\geq A_\rho:=\frac{\rho^2}{4}>0,\qquad u\in\partial B_\rho(0).
\]
\end{lem}
\begin{proof}
The result follows from \cite[Lemma 2.2]{Montenegro}.
\end{proof}
\begin{lem}\label{GMP2}
    Assume $0 < \beta < 1$ and that $(Q)$ is satisfied. There exists
    $\delta_1=\delta_1(\lambda^*)>0$ such that if
    $\lambda > \lambda^* - \delta_1$, then there is a function
    $w\in H_0^1(\Omega)\setminus\{0\}$ satisfying
    \[
    \|w\|>\rho
    \qquad\mbox{and}\qquad
    I_{\lambda,\epsilon}(w)<A,
    \]
    for $0<\eps\ll 1$, where $\rho$ and $A=A_\rho$ are given by
    Lemma~\ref{GMP1}.
\end{lem}

\begin{proof}
Let $v_{\lambda^*}$ be given by Proposition~\ref{min-real-c+}. As a consequence, we know  that
\(
I_{\lambda^*}(v_{\lambda^*})=0.
\)
Choose
\(
0<\rho<\min\{\rho_0,\|v_{\lambda^*}\|\},
\)
where $\rho_0$ is given by Lemma~\ref{GMP1}. Consider 
$A=A_\rho$. Since $I_{\lambda,\epsilon}(v_{\lambda^*})
<I_{\lambda}(v_{\lambda^*})$, we have
\[
I_{\lambda,\epsilon}(v_{\lambda^*})
<
I_{\lambda^*}(v_{\lambda^*})
+
(\lambda^*-\lambda)
\frac{\|v_{\lambda^*}\|_{p+1}^{p+1}}{p+1}
=
(\lambda^*-\lambda)
\frac{\|v_{\lambda^*}\|_{p+1}^{p+1}}{p+1}.
\]
Therefore, assuming that 
\[
\lambda^*-\lambda
<
\delta_1
:=
\frac{(p+1)A}
{2\|v_{\lambda^*}\|_{p+1}^{p+1}},
\]
we obtain that 
\(
I_{\lambda,\epsilon}(v_{\lambda^*})<\frac{A}{2}<A.
\)
As a consequence, taking $w=v_{\lambda^*}$, we have
\(
\|w\|>\rho
~\mbox{and}~
I_{\lambda,\epsilon}(w)<A.
\)
This completes the proof.
\end{proof}
A sequence $(u_{n}) \subset H^{1}_{0}(\Omega)$ is called a Palais-Smale sequence whenever 
\begin{equation}\label{PS-seq}
    I_{\lambda,\eps}(u_n)\leq C\qquad\mbox{and}\qquad \|I_{\lambda,\eps}'(u_n)\|=o_n(1).
\end{equation}
We say that $I_{\lambda,\eps}$ satisfies the \textit{Palais-Smale condition} if every Palais-Smale sequence admits a convergent subsequence.

\begin{lem}\label{PS}
    Assume $0 < \beta < 1$ and that $(Q)$ is satisfied. Then the functional $I_{\lambda,\eps}$ satisfies the Palais-Smale condition.
\end{lem}
\begin{proof}
     Let $(w_k)$ be a sequence satisfying \eqref{PS-seq}. Hence, there exists $C> 0$ such that
    \[
C\geq {I_{\lambda,\eps}(w_k)
\geq
\frac12\|w_k\|^2
-
\frac{\lambda}{p+1}
S_{p+1}^{p+1}\|w_k\|^{p+1}.}
\]
The last assertion implies that $(w_k)$ is bounded in $H^1_0(\Omega)$. Consequently, up to a subsequence, there exists $w\in H_0^1(\Omega)$
such that $w_k\rightharpoonup w$. By the compact Sobolev embedding,
$w_k\to w$ strongly in $L^1(\Omega)$ and in $L^{p+1}(\Omega)$. From
$I_{\lambda,\eps}'(w_k)(w_k-w)=o_k(1)$, we obtain
\begin{align*}
    |\langle w_k,w_k-w\rangle|
    \leq{}&
    \int_\Omega |g_\eps(w_k)||w_k-w|\,\mathrm{d}x
    +\lambda\int_\Omega |w_k|^{p}|w_k-w|\,\mathrm{d}x
    +o_k(1)\\
    \leq{}&
    \|g_\eps\|_\infty\|w_k-w\|_1
    +\lambda\|w_k\|_{p+1}^{p}\|w_k-w\|_{p+1}
    +o_k(1)
    =o_k(1),
\end{align*} 
where we used that $g_\eps$ is bounded. In fact, we have used that  $\|g_\eps\|_\infty \leq \eps^{-(q+\beta)}$ together with the strong convergence $w_k\to w$ in $L^1(\Omega)$ and the same convergence in $L^{p+1}(\Omega)$. Since $\langle w,w_k-w\rangle\to 0$, we conclude that $w_k\to w$ in $H_0^1(\Omega)$ as $k\rightarrow\infty$. This completes the proof. 
\end{proof}
In light of Lemmas~\ref{GMP1}, \ref{GMP2}, and \ref{PS} with the Mountain Pass Theorem, we obtain the following result.
\begin{prop}\label{MPS}
Let $0<\eps\ll1$ and
$\lambda\in(\lambda^*-\delta_1,+\infty)$. Then there exists
$w\in H_0^1(\Omega)$, with $w\geq0$, such that
$\|w\|>\rho$ and $I_{\lambda,\eps}(w)<A$. Setting
\[
\Gamma:=\{\gamma\in C([0,1],H_0^1(\Omega)):
\gamma(0)=0,\ \gamma(1)=w\},
\]
there exists a nonnegative critical point
$u_{\lambda,\eps}\in H_0^1(\Omega)\setminus\{0\}$ satisfying
\[
I_{\lambda,\eps}(u_{\lambda,\eps})
=c_{\lambda,\eps}^m
:=\inf_{\gamma\in\Gamma}\max_{r\in[0,1]}
I_{\lambda,\eps}(\gamma(r))\geq A,
\]
where $\rho$ and $A$ are given by Lemma~\ref{GMP1}.
\end{prop}
\begin{proof}
By Lemma~\ref{GMP2} and the identity
$I_{\lambda,\eps}(|u|)=I_{\lambda,\eps}(u)$, the function $w$ may
be chosen nonnegative. Moreover, for every $\gamma\in\Gamma$, the
path $|\gamma|\in\Gamma$ and
\[
I_{\lambda,\eps}(|\gamma(r)|)
=I_{\lambda,\eps}(\gamma(r)),\qquad r\in[0,1].
\]
Hence, $c_{\lambda,\eps}^m$ may be computed by restricting the
minimax class to nonnegative paths. Let $(\gamma_n)\subset\Gamma$ be a minimizing sequence of such
paths.  Now, by using the standard deformation argument for the Mountain Pass Theorem,
there exists a Palais--Smale sequence $(u_n)$ at level
$c_{\lambda,\eps}^m$ such that
\[
\operatorname{dist}
\bigl(u_n,\gamma_n([0,1])\bigr)\to0.
\]
Since $\gamma_n([0,1])$ consists of nonnegative functions, it follows
that
\[
\operatorname{dist}
\bigl(u_n,\{u\in H_0^1(\Omega):u\geq0\}\bigr)\to0.
\]
In light of Lemma~\ref{PS}, up to a subsequence,
$u_n\to u_{\lambda,\eps}$ in $H_0^1(\Omega)$. Since the positive
cone is closed, $u_{\lambda,\eps}\geq0$. Therefore
$I_{\lambda,\eps}'(u_{\lambda,\eps})=0$ and
$I_{\lambda,\eps}(u_{\lambda,\eps})=c_{\lambda,\eps}^m\geq A$.
\end{proof}

	\section{Passage to the limit: From \texorpdfstring{$\eqref{P-pertubed}$}{eqref{P-pertubed}} to \texorpdfstring{$\eqref{P}$}{eqref{P}}}\label{Pass-lim}

    In this section, we prove that Problem \eqref{P} has at least two solutions. For this purpose, we consider the solutions $v_{\lambda,\epsilon}$ and $u_{\lambda,\epsilon}$ obtained in Propositions \ref{min-real-N-eps} and \ref{MPS}, respectively.

    \begin{prop}\label{lim-sol}
        There exists $\delta_2=\delta_2(\lambda^*)>0$ such that, for every $\lambda\in(\lambda^*-\delta_2,+\infty)$ and every sequence $\eps_n\to0^+$, up to a subsequence, $u_{\lambda,\eps_n}\rightharpoonup u_\lambda$ and $v_{\lambda,\eps_n}\rightharpoonup v_\lambda$ weakly in $H_0^1(\Omega)$, where $u_\lambda,v_\lambda$ are distinct nontrivial weak solutions.
    \end{prop}
 \begin{proof}
Proceeding as in \cite[Proposition 2.5]{Montenegro}, we have, up to a subsequence, that $z_{\lambda,\eps_n}\rightharpoonup z_\lambda$ weakly in $H_0^1(\Omega)$ and $z_{\lambda,\eps_n}\to z_\lambda$ strongly in $L^r(\Omega)$ for $1\leq r<2^*$, where
$z_{\lambda,\eps_n}=u_{\lambda,\eps_n}$ or
$z_{\lambda,\eps_n}=v_{\lambda,\eps_n}$. Indeed, the sequences $(z_{\lambda,\eps_n})$ are bounded in $H_0^1(\Omega)$ uniformly with respect to $\eps_n$ (by Proposition~\ref{N-way-origin-eps}). Therefore, the weak convergence follows from the Banach--Alaoglu Theorem. Moreover, the strong convergence in $L^r(\Omega)$ follows from the compact Sobolev embedding. As a consequence,
\[
\int_\Omega g_{\eps_n}(z_{\lambda,\eps_n})z_{\lambda,\eps_n}\,dx
\to\int_\Omega z_\lambda^{1-\beta}\,dx,
\qquad
\int_\Omega G_{\eps_n}(z_{\lambda,\eps_n})\,dx
\to\frac{1}{1-\beta}\int_\Omega z_\lambda^{1-\beta}\,dx.
\]
Recall also that $z_{\lambda,\eps_n}\in\mathcal N_{\lambda,\eps_n}$. Therefore, 
\[
I_{\lambda,\eps_n}(z_{\lambda,\eps_n})
=\int_\Omega\left[G_{\eps_n}(z_{\lambda,\eps_n})
-\frac12g_{\eps_n}(z_{\lambda,\eps_n})z_{\lambda,\eps_n}\right]dx
+\frac{\lambda(p-1)}{2(p+1)}
\int_\Omega z_{\lambda,\eps_n}^{p+1}\,dx.
\]

Let $\rho>0$ and $A:=\rho^2/4$ be given by
Lemma~\ref{GMP1} with $\lambda=\lambda^*$. Notice that, for
$\lambda\leq\lambda^*$,
\[
I_{\lambda,\eps}(u)
=
I_{\lambda^*,\eps}(u)
+
\frac{\lambda^*-\lambda}{p+1}\|u\|_{p+1}^{p+1}
\geq
I_{\lambda^*,\eps}(u).
\]
Hence, the same constant $A>0$ can be used for the Mountain Pass
geometry for each $\lambda\leq\lambda^*$ sufficiently close to
$\lambda^*$.

Let $\delta>0$ be given by Proposition~\ref{min-real-N-eps} and
$\delta_1>0$ by Proposition~\ref{MPS}. In light of Lemma~\ref{c-lamb-cont},
there exists $\delta_3>0$ such that
\[
c_{\mathcal N_\lambda^+}<A,
\qquad
\lambda\in(\lambda^*-\delta_3,\lambda^*).
\]
Set
\(
\delta_2
:=
\min\left\{
\delta,\delta_1,\delta_3,
\frac{\lambda^*-\lambda_*}{2}
\right\}.
\)
Therefore, Proposition~\ref{eps-0} yields
\[
I_{\lambda,\eps_n}(v_{\lambda,\eps_n})
=
c_{\mathcal N_{\lambda,\eps_n}^+}
<
c_{\mathcal N_\lambda^+}
<A,
\]
whereas Proposition~\ref{MPS} gives
$
I_{\lambda,\eps_n}(u_{\lambda,\eps_n})\geq A.
$
Passing to the limit in the preceding energy identity, we obtain
\[
\frac{1+\beta}{2(1-\beta)}
\int_\Omega |v_\lambda|^{1-\beta}\,\mathrm{d}x
+
\frac{\lambda(p-1)}{2(p+1)}
\int_\Omega |v_\lambda|^{p+1}\,\mathrm{d}x
<A,
\]
while
\[
\frac{1+\beta}{2(1-\beta)}
\int_\Omega |u_\lambda|^{1-\beta}\,\mathrm{d}x
+
\frac{\lambda(p-1)}{2(p+1)}
\int_\Omega |u_\lambda|^{p+1}\,\mathrm{d}x
\geq A.
\]

Now, for each $\lambda\geq\lambda^*$ with $0<\eps\ll1$,
$\lambda\geq\lambda_{\eps,*}$ and
$I_{\lambda,\eps}(v_{\lambda,\eps})\leq0$, we infer that
$I_{\lambda,\eps}(u_{\lambda,\eps})\geq A>0$. Passing to the limit
as above again yields $u_\lambda\neq v_\lambda$.
Moreover, Proposition~\ref{N-way-origin-eps} yields, uniformly in
$\eps_n$,
\(
\|u_{\lambda,\eps_n}\|_{p+1}, \|v_{\lambda,\eps_n}\|_{p+1}\geq\delta_\lambda>0.
\)
This completes the proof. 
\end{proof}

\begin{prop}\label{C-cont-eps}
    Assume $0 < \beta < 1$ and that $(Q)$ is satisfied. There exists $\delta = \delta(\lambda^*) > 0$ such that, for every $\lambda > \lambda^* - \delta$, the map $\epsilon \mapsto c_{\mathcal{N}_{\lambda,\epsilon}^+}$ is right-continuous at $\epsilon = 0$.
\end{prop}
\begin{proof}
Now, by using Propositions~\ref{eps-0} and~\ref{min-real-N-eps}, for $0<\eps\ll1$ there exists
$v_{\lambda,\eps}\in\mathcal N_{\lambda,\eps}^+$ such that
$I_{\lambda,\eps}(v_{\lambda,\eps})=c_{\mathcal N_{\lambda,\eps}^+}$
and
$c_{\mathcal N_{\lambda,\eps}^+}<c_{\mathcal N_\lambda^+}$.
Moreover, by Proposition~\ref{compar-ep},
$t_\lambda^+(v_{\lambda,\eps})<t_{\lambda,\eps}^+(v_{\lambda,\eps})=1$.
Hence, by the fibering properties,
\[
c_{\mathcal N_{\lambda,\eps}^+}
<c_{\mathcal N_\lambda^+}
\leq I_\lambda(t_\lambda^+(v_{\lambda,\eps})v_{\lambda,\eps})
<I_\lambda(v_{\lambda,\eps}),
\]
and consequently
\begin{equation}\label{est-c-eps=0}
0<c_{\mathcal N_\lambda^+}-c_{\mathcal N_{\lambda,\eps}^+}
<
\int_\Omega\left[
\frac{|v_{\lambda,\eps}|^{1-\beta}}{1-\beta}
-G_\eps(v_{\lambda,\eps})
\right]\,dx.
\end{equation}

Let $\eps_n\to0^+$ be arbitrary and set
$v_n:=v_{\lambda,\eps_n}$. By Proposition~\ref{lim-sol}, up to a subsequence,
$v_n\to v_\lambda$ in $L^1(\Omega)$ and
\[
\int_\Omega G_{\eps_n}(v_n)\,dx
\longrightarrow
\frac{1}{1-\beta}\int_\Omega|v_\lambda|^{1-\beta}\,dx.
\]
Since $0<1-\beta<1$,
\[
\int_\Omega\bigl||v_n|^{1-\beta}-|v_\lambda|^{1-\beta}\bigr|\,dx
\leq
|\Omega|^\beta
\|v_n-v_\lambda\|_{L^1(\Omega)}^{1-\beta}
\longrightarrow0.
\]
Therefore, the right-hand side of \eqref{est-c-eps=0} converges to zero, which implies that
\(
c_{\mathcal N_{\lambda,\eps_n}^+}
\longrightarrow
c_{\mathcal N_\lambda^+}.
\)

Now, by Corollary~\ref{C-monotone-eps}, the map
$\eps\mapsto c_{\mathcal N_{\lambda,\eps}^+}$ is monotone for $\eps>0$
sufficiently small. Thus the limit at $\eps=0$ of the expression on the right-hand side is finite.
On the other hand, the sequence $(\eps_n)$ considered above was
arbitrary, and we have shown that it possesses a subsequence along
which
$
    c_{\mathcal N_{\lambda,\eps_n}^+}
    \rightarrow
    c_{\mathcal N_\lambda^+}.
$
Therefore,
\[
\lim_{\eps\to0^+}c_{\mathcal N_{\lambda,\eps}^+}
=
c_{\mathcal N_\lambda^+}.
\]
This finishes the proof. 
\end{proof}

\begin{lem}\label{lem:localized-test}
 Let $u\in H_0^1(\Omega)$ be a weak solution of problem~\eqref{P}.
Assume that $\varphi\in H_0^1(\Omega)\cap L^\infty(\Omega)$ and that, for some $\delta>0$, $\varphi=0$ a.e. on $\{|u|\leq\delta\}$. Then $\varphi$ is admissible as a test function in the weak formulation of~\eqref{P}. More precisely,
\[
    \int_\Omega \nabla u\nabla\varphi\,dx
    +
    \int_{\{u\neq0\}}
    \frac{u}{|u|^{\beta+1}}\varphi\,dx
    =
    \lambda
    \int_\Omega |u|^{p-1}u\varphi\,dx.
\]
\end{lem}

\begin{proof}
We divide the argument into two steps.

\medskip
\noindent
\textbf{Step 1.} Extension to bounded $H_0^1$-functions with compact support.
Let $\psi\in H_0^1(\Omega)\cap L^\infty(\Omega)$ have compact
support in $\Omega$. Since
\[
\frac{u}{|u|^{\beta+1}}\chi_{\{u\neq0\}}\in L_{\rm loc}^1(\Omega)
\quad\text{and}\quad
|u|^p\in L^1(\Omega),
\]
a standard density argument extends the weak formulation from
$C_c^\infty(\Omega)$ to $\psi$. Indeed, approximating $\psi$ in
$H_0^1(\Omega)$ by uniformly bounded smooth functions whose supports
are contained in a fixed compact subset of $\Omega$, the gradient
term converges by the $H_0^1$ convergence, whereas the remaining
terms converge by the Dominated Convergence Theorem, using
$\frac{u}{|u|^{\beta+1}}\chi_{\{u\neq0\}}\in L^1_{\rm loc}(\Omega)$ and
$|u|^p\in L^1(\Omega)$ together with the uniform boundedness of the
approximating sequence. Consequently,
\[
\int_\Omega\nabla u\nabla\psi\,dx
+
\int_{\{u\neq0\}}
\frac{u}{|u|^{\beta+1}}\psi\,dx
=
\lambda\int_\Omega |u|^{p-1}u\psi\,dx.
\]

\medskip
\noindent
\textbf{Step 2.} Removal of the compact-support assumption. Let $\varphi\in H_0^1(\Omega)\cap L^\infty(\Omega)$ satisfy
$\varphi=0$ a.e. on $\{|u|\leq\delta\}$. Let
$\vartheta_\delta:\mathbb R\to[0,1]$ be a Lipschitz function such that
\[
\vartheta_\delta(s)=0 \quad\text{if }|s|\leq\frac{\delta}{2},
\qquad
\vartheta_\delta(s)=1 \quad\text{if }|s|\geq\delta.
\]
Choose $(v_j)\subset C_c^\infty(\Omega)$ such that
$v_j\to\varphi$ in $H_0^1(\Omega)$ and a.e. in $\Omega$, with
$|v_j|\leq C\|\varphi\|_\infty$. Now, we consider the set
\[
\psi_j:=\vartheta_\delta(u)v_j.
\]
Then $\psi_j\in H_0^1(\Omega)\cap L^\infty(\Omega)$ has compact
support, and hence it is an admissible test function by Step~1.

Since $\vartheta_\delta(u)\varphi=\varphi$, the Sobolev product and
chain rules, together with dominated convergence, yield
\[
\psi_j\to\varphi
\quad\text{in }H_0^1(\Omega).
\]
Moreover, $\psi_j=0$ on $\{|u|\leq\delta/2\}$, so that
\[
\left|
\frac{u}{|u|^{\beta+1}}\psi_j
\right|
\leq
C\|\varphi\|_\infty
\left(\frac{2}{\delta}\right)^\beta
\quad\text{a.e. in }\Omega.
\]
Hence, by the Dominated Convergence Theorem, we may pass to the limit
in the singular term. The reaction term is treated similarly since
$|u|^p\in L^1(\Omega)$ and the sequence $(\psi_j)$ is uniformly
bounded. Letting $j\to\infty$ in the weak formulation
for $\psi_j$, we obtain
\[
\int_\Omega \nabla u\nabla\varphi\,dx
+
\int_{\{u\neq0\}}
\frac{u}{|u|^{\beta+1}}\varphi\,dx
=
\lambda
\int_\Omega |u|^{p-1}u\varphi\,dx.
\]
This completes the proof.
\end{proof}

\begin{lem}\label{lem:nehari-identity-solution}
Let $u\in H_0^1(\Omega)$ be a nontrivial nonnegative weak solution of
Problem~\eqref{P}. Then
\begin{equation}\label{eq:nehari-identity-solution}
    \|u\|^2
    +
    \|u\|_{1-\beta}^{1-\beta}
    =
    \lambda
    \|u\|_{p+1}^{p+1}.
\end{equation}
In particular,
\(
    u\in\mathcal N_\lambda.
\)
\end{lem}

\begin{proof}
For $k\in\mathbb N$,  we consider 
\[
    \varphi_k:=\left(u-\frac1k\right)^+.
\]
Since $\varphi_k=0$ a.e. on $\{u\leq1/k\}$, the singular term is
bounded on the support of $\varphi_k$. Hence, by
Lemma~\ref{lem:localized-test}, $\varphi_k$ is an admissible test
function in the weak formulation of Problem~\eqref{P}. Therefore,
\[
\int_{\{u>1/k\}}|\nabla u|^2\,dx
+
\int_{\{u>1/k\}}
\frac{u-1/k}{u^\beta}\,dx
=
\lambda
\int_{\{u>1/k\}}
u^p\left(u-\frac1k\right)\,dx.
\]
Letting $k\to\infty$, we have
\[
\int_{\{u>1/k\}}|\nabla u|^2\,dx
\longrightarrow
\|u\|^2,
\]
since $\nabla u=0$ a.e. on $\{u=0\}$. Moreover, by the Monotone
Convergence Theorem applied to the sequences
$\frac{u-1/k}{u^\beta}\chi_{\{u>1/k\}}$ and
$u^p(u-1/k)\chi_{\{u>1/k\}}$ (which are nonnegative and increase
a.e. to $u^{1-\beta}$ and $u^{p+1}$, respectively), we obtain
\[
\int_{\{u>1/k\}}
\frac{u-1/k}{u^\beta}\,dx
\longrightarrow
\int_\Omega u^{1-\beta}\,dx
\qquad\mbox{and}\qquad
\int_{\{u>1/k\}}
u^p\left(u-\frac1k\right)\,dx
\longrightarrow
\int_\Omega u^{p+1}\,dx.
\]
As a consequence, we see that 
\[
\|u\|^2
+
\int_\Omega u^{1-\beta}\,dx
=
\lambda
\int_\Omega u^{p+1}\,dx.
\]
Hence
$u\in\mathcal N_\lambda$.
\end{proof}

Now, we turn to our main results. The proof strategy relies on the interplay between the techniques developed in \cite[Theorem 1.1]{Montenegro} and the preliminary results obtained in the previous sections.

\subsection{Proof of Theorem \ref{theorem1}} 
Choose $q>0$ satisfying $(Q)$, which is always possible since
$0<\beta,p<1$. In view of Lemma~\ref{gamma}-(iii), we have
$0<\lambda_*<\lambda^*<+\infty$. Let $\epsilon_1:=\min\{\delta_1,\delta_2\}$ and
$\lambda\in(\lambda^*-\epsilon_1,+\infty)$. By
Proposition~\ref{MPS}, for $0<\eps\ll1$ there exists a nonnegative
critical point
$u_{\lambda,\eps}\in H_0^1(\Omega)\setminus\{0\}$ such that
$
I_{\lambda,\eps}(u_{\lambda,\eps})\geq A.
$
Moreover, Proposition~\ref{lim-sol} yields, up to a subsequence,
$
u_{\lambda,\eps_n}\rightharpoonup u_\lambda
$ weakly in $H_0^1(\Omega)$, with $u_\lambda\neq0$. Hence,
$u_{\lambda,\eps_n}\to u_\lambda$ strongly in $L^r(\Omega)$ for every
$1\leq r<2^*$ and a.e. in $\Omega$. Since
$u_{\lambda,\eps_n}\geq0$, it follows that $u_\lambda\geq0$ a.e. in
$\Omega$. Furthermore,
\[
g_{\eps_n}(u_{\lambda,\eps_n})
=
\frac{u_{\lambda,\eps_n}^{q}}
{(u_{\lambda,\eps_n}+\eps_n)^{q+\beta}}.
\]
Therefore, proceeding as in \cite[Theorem 1.1]{Montenegro}, the local
estimates and the truncation argument used there allow us to pass to
the limit in the singular term. Consequently, $u_\lambda$ is a
nontrivial weak solution of Problem~\eqref{P}.

Now, we prove that the convergence is strong in $H_0^1(\Omega)$. Since
$u_\lambda$ is a nontrivial nonnegative weak solution of
Problem~\eqref{P}, Lemma~\ref{lem:nehari-identity-solution} yields
\[
\|u_\lambda\|^2
+
\int_\Omega u_\lambda^{1-\beta}\,dx
=
\lambda
\int_\Omega u_\lambda^{p+1}\,dx.
\]
On the other hand, taking into account that $u_{\lambda,\eps_n}$ is a critical point of
$I_{\lambda,\eps_n}$, we have
\[
\|u_{\lambda,\eps_n}\|^2
+\int_\Omega
g_{\eps_n}(u_{\lambda,\eps_n})u_{\lambda,\eps_n}\,\mathrm{d}x
=
\lambda\int_\Omega u_{\lambda,\eps_n}^{p+1}\,\mathrm{d}x.
\]
Arguing as in the proof of Proposition~\ref{lim-sol}, we obtain
\[
\int_\Omega
g_{\eps_n}(u_{\lambda,\eps_n})u_{\lambda,\eps_n}\,\mathrm{d}x
\to
\int_\Omega u_\lambda^{1-\beta}\,\mathrm{d}x.
\]
Moreover, by using the fact that $u_{\lambda,\eps_n}\to u_\lambda$ in $L^{p+1}(\Omega)$, we deduce that 
\[
\lim_{n\to\infty}\|u_{\lambda,\eps_n}\|^2
=
\lambda\int_\Omega u_\lambda^{p+1}\,\mathrm{d}x
-\int_\Omega u_\lambda^{1-\beta}\,\mathrm{d}x
=
\|u_\lambda\|^2.
\]
Therefore, $u_{\lambda,\eps_n}\to u_\lambda$ strongly in $H_0^1(\Omega)$.

By Proposition~\ref{lim-sol} and the strong convergence
obtained above, we conclude that
$
I_{\lambda,\eps_n}(u_{\lambda,\eps_n})
\to I_\lambda(u_\lambda).
$
It follows from $I_{\lambda,\eps_n}(u_{\lambda,\eps_n})\geq A$ that
$I_\lambda(u_\lambda)\geq A$. Hence, Problem~\eqref{P} admits a
nontrivial weak solution $u_\lambda$ for every
$\lambda\in(\lambda^*-\epsilon_1,+\infty)$. This finishes the proof of Theorem \ref{theorem1}.

\subsection{Proof of Theorem \ref{theorem2}} 
Let us consider $q>0$ satisfying $(Q)$. Choose $\epsilon_2>0$ sufficiently small so that
Propositions~\ref{min-real-N-eps}, \ref{lim-sol} and
\ref{C-cont-eps} can be applied for 
$\lambda\in(\lambda^*-\epsilon_2,+\infty)$. In view of Proposition~\ref{min-real-N-eps}, for
$0<\eps\ll1$ there exists
$v_{\lambda,\eps}\in\mathcal N_{\lambda,\eps}^+$ such that
$I_{\lambda,\eps}(v_{\lambda,\eps})
    =
    c_{\mathcal N_{\lambda,\eps}^+}.
$
Since $I_{\lambda,\eps}$ and
$\mathcal N_{\lambda,\eps}^+$ are invariant under the map
$u\mapsto |u|$, the function $|v_{\lambda,\eps}|$ is again a
minimizer of $I_{\lambda,\eps}$ on
$\mathcal N_{\lambda,\eps}^+$. Therefore, replacing
$v_{\lambda,\eps}$ by $|v_{\lambda,\eps}|$, we
assume that
$v_{\lambda,\eps}\geq0$ in $\Omega$.
Now, by using Lemma~\ref{min-sol},  $v_{\lambda,\eps}$ is a minimizer which give us
a critical point of $I_{\lambda,\eps}$. Hence
$I_{\lambda,\eps}'(v_{\lambda,\eps})=0.$ Now, by taking $\eps_n\to0^+$ and taking into account Proposition~\ref{lim-sol}, we infer that
$
v_{\lambda,\eps_n}\rightharpoonup v_\lambda
$ weakly in $H_0^1(\Omega)$,
where $v_\lambda\neq0$. Arguing similarly to the proof of Theorem~\ref{theorem1}, with
$v_{\lambda,\eps_n}$ in place of $u_{\lambda,\eps_n}$, we conclude
that $v_\lambda$ is a nontrivial weak solution of
Problem~\eqref{P}  and 
$
v_{\lambda,\eps_n}\to v_\lambda
$ strongly in $H_0^1(\Omega)$.
Consequently, Proposition~\ref{lim-sol} yields
$
I_{\lambda,\eps_n}(v_{\lambda,\eps_n})
\to I_\lambda(v_\lambda).
$
On the other hand, Proposition~\ref{C-cont-eps} gives 
\(
c_{\mathcal N_{\lambda,\eps_n}^+}
\to c_{\mathcal N_\lambda^+}
\)
as $n\to \infty$. Therefore,
$
I_\lambda(v_\lambda)=c_{\mathcal N_\lambda^+}.
$

Now, we prove that $v_\lambda\in\mathcal N_\lambda^+$, i.e., item $(i)$ of Theorem \ref{theorem2}. The Nehari identity obtained in the same
truncation argument gives $v_\lambda\in\mathcal N_\lambda$. Moreover,
taking into account that $v_{\lambda,\eps_n}\in\mathcal N_{\lambda,\eps_n}^+$, we have that
\[
0<\phi_{\eps_n,v_{\lambda,\eps_n}}''(1)=
(1-p)\|v_{\lambda,\eps_n}\|^2
-\int_\Omega g_{\eps_n}(v_{\lambda,\eps_n})
v_{\lambda,\eps_n}
\left[
p+\beta-(q+\beta)
\frac{\eps_n}{|v_{\lambda,\eps_n}|+\eps_n}
\right]\,\mathrm{d}x.
\]
We observe that
\[
0\leq
\frac{\eps\,g_\eps(s)s}{|s|+\eps}
=
\eps^{1-\beta}
\frac{(|s|/\eps)^{q+1}}
{(1+|s|/\eps)^{q+\beta+1}}
\leq C\eps^{1-\beta},
\qquad s\in\mathbb R.
\]
Hence, by Proposition~\ref{lim-sol} and the strong convergence obtained
above, we mention that 
\[
\phi_{\eps_n,v_{\lambda,\eps_n}}''(1)
\to
(1-p)\|v_\lambda\|^2
-(p+\beta)\int_\Omega |v_\lambda|^{1-\beta}\,\mathrm{d}x
=
\phi_{v_\lambda}''(1).
\]
Therefore, $\phi_{v_\lambda}''(1)\geq 0$.

Now, we show that the previous inequality is strict. Recall from Proposition~\ref{N0-dist-0} that, for every
$\mu\geq\lambda_*$, there exists $\sigma_\mu>0$ such that $I_\mu(u)\geq\sigma_\mu$ for all $u\in\mathcal N_\mu^0$. Moreover, by the explicit expression obtained in
\eqref{sigm-lambda}, the map $\mu\mapsto\sigma_\mu$ is positive and continuous. Hence, after reducing $\epsilon_2>0$, we assume that $\lambda^*-\epsilon_2>\lambda_*$ and 
\[
    \sigma_0
    :=
    \inf_{\mu\in[\lambda^*-\epsilon_2,\lambda^*]}
    \sigma_\mu
    >0.
\]
In view of Lemma \ref{c-lamb-cont}, the map $\lambda\mapsto c_{\mathcal N_\lambda^+}$ is left-continuous at $\lambda=\lambda^*$. Moreover, it follows from Proposition \ref{min-real-c+} $(ii)$ that $c_{\mathcal N_{\lambda^*}^+}=0$. Thus, decreasing $\epsilon_2$ once more if necessary, we may assume
that
\[
    0<c_{\mathcal N_\lambda^+}
    <\sigma_0,
    \quad
    \text{for every }
    \lambda\in(\lambda^*-\epsilon_2,\lambda^*).
\]
Consequently, for each
$\lambda\in(\lambda^*-\epsilon_2,\lambda^*)$, we infer that 
\[
    c_{\mathcal N_\lambda^+}
    <
    \sigma_0
    \leq
    \sigma_\lambda
    \leq
    \inf_{u\in\mathcal N_\lambda^0}I_\lambda(u).
\]
For $\lambda\geq\lambda^*$, Proposition \ref{min-real-c+} $(ii)$ and $(iii)$ give $c_{\mathcal N_\lambda^+}\leq0$. Therefore, by using Proposition~\ref{N0-dist-0}, we also see that
\[
    \inf_{u\in\mathcal N_\lambda^0}I_\lambda(u)
    \geq\sigma_\lambda>0.
\]
Therefore, by using the cases considered above, we conclude that
\begin{equation}\label{energy-separation-N0}
    c_{\mathcal N_\lambda^+}
    <
    \inf_{u\in\mathcal N_\lambda^0}I_\lambda(u),
    \qquad
    \lambda\in(\lambda^*-\epsilon_2,+\infty).
\end{equation}

Now we return to prove that $\phi_{v_\lambda}''(1)>0$. Suppose, by contradiction, that $\phi_{v_\lambda}''(1)=0$. Since $v_\lambda$ is a nontrivial weak solution of \eqref{P}, it follows from Lemma~\ref{lem:nehari-identity-solution} that $v_\lambda\in\mathcal N_\lambda$. Consequently, $v_\lambda\in\mathcal N_\lambda^0$. Thus, by \eqref{energy-separation-N0},
\[
    I_\lambda(v_\lambda)
    \geq
    \inf_{u\in\mathcal N_\lambda^0}I_\lambda(u)
    >
    c_{\mathcal N_\lambda^+},
\]
which contradicts the fact that $I_\lambda(v_\lambda) = c_{\mathcal N_\lambda^+}$. Therefore $\phi_{v_\lambda}''(1)>0$
and $v_\lambda\in\mathcal N_\lambda^+$. This finishes the proof of item $(i)$.

Now, by reducing $\epsilon_2>0$, we assume that
$\lambda^*-\epsilon_2>\lambda_*$. Since
$
I_\lambda(v_\lambda)=c_{\mathcal N_\lambda^+},
$
the remaining assertions follow directly from
Proposition~\ref{min-real-c+}. The proof is now complete.

\subsection{Proof of Corollary \ref{coro}} 
By decreasing $\epsilon_2>0$, we assume that
$\epsilon_2\leq\epsilon_1$, where $\epsilon_1$ is given by
Theorem~\ref{theorem1}. Let
$\lambda\in(\lambda^*-\epsilon_2,+\infty)$. Then,
Theorem~\ref{theorem1} provides a nontrivial weak solution
$u_\lambda$ of Problem~\eqref{P}, whereas
Theorem~\ref{theorem2} provides another nontrivial weak solution
$v_\lambda$. Moreover, Proposition~\ref{lim-sol} yields
$u_\lambda\neq v_\lambda$. Therefore, Problem~\eqref{P} admits at
least two distinct nontrivial weak solutions for every
$\lambda\in(\lambda^*-\epsilon_2,+\infty)$.

\subsection{Proof of Theorems \ref{theorem4a} and \ref{theorem4}} 
By Theorem~\ref{theorem1}, there exists $\lambda_0<\lambda^*$ such
that Problem~\eqref{P} admits a nontrivial nonnegative weak solution.
Hence, by the definition of $\bar\lambda$,
\(
\bar\lambda\leq\lambda_0<\lambda^*.
\)
In order to prove that $\lambda_*<\bar\lambda$ and that the infimum
defining $\bar\lambda$ is attained, we first establish some auxiliary
results.
\begin{lem}
         Assume $0 < \beta < 1$ and that $(Q)$ is satisfied. Then the minimizer $u_{\eps,*}$ of the map $u \mapsto \lambda_{n,\eps}(u)$ satisfies $u_{\eps,*} \to u_*$ as $\eps\rightarrow0^{+}$, where $u_*$ is a minimizer for the functional $\lambda_n$.  Furthermore, the function $v:=t_n(u_*)u_*$ belongs to $\mathcal{N}_{\lambda_*}^0$ and satisfies
\begin{align}
2\int_\Omega \nabla v \nabla\varphi\,\mathrm{d}x
&+
(1-\beta)
\int_\Omega
\frac{v}{|v|^{\beta+1}}
\chi_{\{|v|>0\}}\varphi\,\mathrm{d}x =
\lambda_*(p+1)
\int_\Omega |v|^{p-1}v\varphi\,\mathrm{d}x,
\label{eq-lambda-n-local}
\end{align}
for all $\varphi\in H_0^1(\Omega)\cap L^\infty(\Omega)$ such that $\varphi=0$ a.e. in $\{|v|\leq\delta\}$, for some $\delta>0$.
\end{lem}
\begin{proof}
Since $u_{\eps,*}$ is a minimizer of $\lambda_{n,\eps}$ and it is $0$-homogeneous, we also have that $tu_{\eps,*}$ is a minimizer of $\lambda_{n,\eps}$ for all $t>0$. In particular, we can choose $t_{n,\eps}(u_{\eps,*})=1$. Moreover, up to a subsequence, $u_{\eps,*}\to u_*$ strongly in $H_0^1(\Omega)$ as $\eps\rightarrow0^{+}$, 
where $u_*$ is a minimizer of $\lambda_n$, $\lambda_{\eps,*}\to\lambda_*$ as $\eps\rightarrow0^{+}$ and $R_n(u_*)=\lambda_*$. Thus,
\[
\lambda_*
\leq
\lambda_n(u_*)
=
R_n(t_n(u_*)u_*)
\leq
R_n(u_*)
=
\lambda_*.
\]
Consequently, $\lambda_n(u_*)=R_n(u_*)=\lambda_*.$ Notice also that $t_n(u_*)$ is the unique critical point of the map $t\mapsto R_n(tu_*),~ t>0$. Therefore, $t_n(u_*)=1$ which implies that 
$
v:=t_n(u_*)u_*=u_*,
$ satisfies $d/d t[R_n(v)] = 0$. In view of Proposition \ref{der-Rn} (iii), we infer also that $v\in\mathcal N_{\lambda_*}^0$. Furthermore, for every
$w\in H_0^1(\Omega)\setminus\{0\}$, we have
\[
R_n(v)
=
\lambda_*
\leq
\lambda_n(w)
=
R_n(t_n(w)w)
\leq
R_n(w),
\]
which implies that $v$ is a minimizer of $R_n$.

Let $\delta>0$ and
$\varphi\in H_0^1(\Omega)\cap L^\infty(\Omega)$ satisfy
$\varphi=0$ a.e. on $\{|v|\leq\delta\}$. As was done in the proof of 
Lemma~\ref{lem:localized-test}, this condition controls the variation
on the singular set. More precisely, for $|t|$ sufficiently
small,
\[
|v+t\varphi|\geq\frac{\delta}{2}
\quad\text{a.e. on }\{\varphi\neq0\}.
\]
Hence, the map $t\mapsto R_n(v+t\varphi)$ is differentiable at
$t=0$. Since $v$ is a minimizer of $R_n$, it
follows that
\[
\left.
\frac{d}{dt}
R_n(v+t\varphi)
\right|_{t=0}
=0.
\]
Since $R_n(v)=\lambda_*$, we obtain
\eqref{eq-lambda-n-local}, which finishes the proof.
\end{proof}

\begin{prop}\label{non-exis-lam*}
Suppose $0<\beta,p < 1$ and $\lambda=\lambda_*$. Then Problem \eqref{P} has no nontrivial weak solution in $\mathcal{N}_{\lambda_*}^0$.
\end{prop}
\begin{proof}
Here, we use the ideas of \cite[Corollary 2.1]{AlvesSantosSilva}.
Assume, by contradiction, that there exists
$w\in\mathcal{N}_{\lambda_*}^0$ such that $w$ is a weak solution of
\eqref{P}. Since $R_n(w)=\lambda_*,$ we have
\[
R_n(w)
=
\lambda_*
\leq
\lambda_n(u)
=
R_n(t_n(u)u)
\leq
R_n(u),
\]
for every $u\in H_0^1(\Omega)\setminus\{0\}$. As a consequence, $w$ is a global minimizer of $R_n$.

Let $\delta>0$ and
$\varphi\in H_0^1(\Omega)\cap L^\infty(\Omega)$ satisfy
$\varphi=0$ a.e. on $\{|w|\leq\delta\}$. Since $w$ is a global
minimizer of $R_n$, the same argument can be used in order to show that the
derive \eqref{eq-lambda-n-local} yields
\[
2\int_\Omega \nabla w \nabla\varphi\,\mathrm{d}x
+
(1-\beta)\int_\Omega
\frac{w}{|w|^{\beta+1}}\varphi\,\mathrm{d}x
=
\lambda_*(p+1)
\int_\Omega |w|^{p-1}w\varphi\,\mathrm{d}x.
\]
On the other hand, Lemma~\ref{lem:localized-test} gives
\[
\int_\Omega \nabla w \nabla\varphi\,\mathrm{d}x
+
\int_\Omega
\frac{w}{|w|^{\beta+1}}\varphi\,\mathrm{d}x
=
\lambda_*
\int_\Omega |w|^{p-1}w\varphi\,\mathrm{d}x.
\]
As a consequence, we obtain
\begin{equation}\label{identity-w}
(1+\beta)
\int_\Omega
\frac{w}{|w|^{\beta+1}}\varphi\,\mathrm{d}x
=
\lambda_*(1-p)
\int_\Omega |w|^{p-1}w\varphi\,\mathrm{d}x.
\end{equation}
For $\delta>0$, choose a Lipschitz function
$\eta_\delta:\mathbb R\to[0,1]$ such that
\[
\eta_\delta(s)=0 \quad\text{if }|s|\leq\delta,
\qquad
\eta_\delta(s)>0 \quad\text{if }|s|>\delta,
\qquad
\eta_\delta(s)=1 \quad\text{if }|s|\geq2\delta.
\]
For every $\psi\in C_c^\infty(\Omega)$, the function
$\varphi=\eta_\delta(w)\psi$ belongs to
$H_0^1(\Omega)\cap L^\infty(\Omega)$ and vanishes on
$\{|w|\leq\delta\}$. Hence, taking this function in
\eqref{identity-w}, we obtain
\[
\int_\Omega
\left[
(1+\beta)\frac{w}{|w|^{\beta+1}}
-\lambda_*(1-p)|w|^{p-1}w
\right]
\eta_\delta(w)\psi\,\mathrm{d}x=0.
\]
Since $\psi\in C_c^\infty(\Omega)$ is arbitrary and
$\eta_\delta(w)>0$ on $\{|w|>\delta\}$, the Du Bois--Reymond
Lemma yields
\[
(1+\beta)\frac{w}{|w|^{\beta+1}}
=
\lambda_*(1-p)|w|^{p-1}w
\quad\text{a.e. on }\{|w|>\delta\}.
\]
Since $\delta>0$ is arbitrary,
\[
(1+\beta)\frac{w}{|w|^{\beta+1}}
=
\lambda_*(1-p)|w|^{p-1}w
\quad\text{a.e. on }\{w\neq0\}.
\]
Consequently,
\[
|w|
=
\left(
\frac{1+\beta}{\lambda_*(1-p)}
\right)^{\frac{1}{p+\beta}}=:c_*,
\quad\text{a.e. in }\{w\neq0\}.
\]
The last assertion implies also that $|w|$ is a constant function of $H_0^1(\Omega)$. That is $w\equiv 0$, which is an absurd. Therefore, Problem~\eqref{P} has no solution in
$\mathcal{N}_{\lambda_*}^0$.
\end{proof}



To prove that the extremal parameter $\bar\lambda$ is attained, we first
establish a compactness result for the regularized problems. Another feature here is to ensure a subsolution principle that implies the monotonicity on an appropriate interval for the parameter $\lambda  > 0$. 

\begin{lem}\label{lem:singular-compactness}
Suppose $0<\beta,p<1$, let $0<q<p$, and fix $\lambda>0$. Let
$\eps_n\to0^+$ and let $(w_n)\subset H_0^1(\Omega)$ be a sequence of
nonnegative weak solutions of
\[
-\Delta w_n+g_{\eps_n}(w_n)=\lambda w_n^p
\quad\text{in }\Omega,
\qquad
w_n=0
\quad\text{on }\partial\Omega.
\]
Then, up to a subsequence, there exists a nonnegative weak solution
$w\in H_0^1(\Omega)$ of Problem~\eqref{P} such that
\(
w_n\rightharpoonup w
\quad\text{in }H_0^1(\Omega),
~
w_n\to w
\quad\text{in }L^r(\Omega)
\)
for every $1\leq r<2^*$, and $w_n\to w$ a.e. in $\Omega$.
Moreover, the convergence is locally uniform in $\Omega$.
Furthermore, assuming that there exists
$u_0\in H_0^1(\Omega)\setminus\{0\}$, $u_0\geq0$, such that
\(
w_n\geq u_0
~\text{a.e. in }\Omega
\)
for every $n$, we obtain that $w\not\equiv0$.
\end{lem}

\begin{proof}
Now, using $w_n$ as a testing function together with the fact that
$g_{\eps_n}(t)t\geq0$, we obtain
\[
\|w_n\|^2
\leq
\lambda\|w_n\|_{p+1}^{p+1}
\leq
\lambda S_{p+1}^{p+1}\|w_n\|^{p+1}.
\]
Hence $(w_n)$ is bounded in $H_0^1(\Omega)$. Therefore, up to a
subsequence, by using the Banach--Alaoglu Theorem and the compact Sobolev embedding,
\(
w_n\rightharpoonup w
~\text{in }H_0^1(\Omega),
~
w_n\to w
\quad\text{in }L^r(\Omega),
\)
for every $1\leq r<2^*$, and $w_n\to w$ a.e. in $\Omega$.

Recall also that $w_n\geq0$. Hence, we can apply the same process on the regularization considered by \cite{Montenegro}.
Consequently, the uniform $L^\infty$ estimate, the local gradient
estimate, and the truncation argument developed in
\cite[Lemma~2.4, Lemma~3.1 and the proof of Theorem~1.1]{Montenegro}
can be applied for $(w_n)$. It follows that $w_n\to w$ locally uniformly in $\Omega$ and that $w$
is a nonnegative weak solution of Problem~\eqref{P}.

Finally, assuming that $w_n\geq u_0$ a.e. in $\Omega$ for every $n$, the
a.e. convergence implies that
\(
w\geq u_0
\quad\text{a.e. in }\Omega.
\)
Since $u_0\not\equiv0$, we conclude that $w\not\equiv0$.
\end{proof}

\begin{prop}\label{prop:existence-monotone}
Assume that $0<\beta,p<1$ and let $\lambda>0$. Suppose that
$u_0\in H_0^1(\Omega)\setminus\{0\}$ is nonnegative,
\(
\frac{1}{u_0^\beta}\chi_{\{u_0>0\}}
\in L^1_{\rm loc}(\Omega),
\)
and satisfies
\begin{equation}\label{subsolution-P}
\int_\Omega\nabla u_0 \nabla\varphi\,dx
+
\int_{\{u_0>0\}}
\frac{\varphi}{u_0^\beta}\,dx
\leq
\lambda\int_\Omega u_0^p\varphi\,dx\;\;\forall \varphi\in C_c^\infty(\Omega),\;\varphi\geq 0.
\end{equation}
Then
Problem~\eqref{P} admits a nontrivial nonnegative weak solution
$w_\lambda$ satisfying $
w_\lambda\geq u_0
\quad\text{a.e. in }\Omega.$\\
In particular, assuming that Problem~$(P_{\lambda_0})$ admits a nontrivial
nonnegative weak solution for some $\lambda_0>0$, we obtain that
Problem~$(P_\lambda)$ admits a nontrivial nonnegative weak solution
for every $\lambda>\lambda_0$.
\end{prop}

\begin{proof}
We divide the proof into two parts.

\medskip
\noindent \textbf{Part 1. Subsolution and minimization.}
For each $\eps>0$, define
$\widetilde f_{\lambda,\eps}:\Omega\times\mathbb R\to\mathbb R$ by
$$
\widetilde f_{\lambda,\eps}(x,s)
=\left\{
\begin{array}{lr}
\lambda u_0(x)^p-g_\eps(u_0(x)),
& s\leq u_0(x),\\
\lambda s^p-g_\eps(s),
& s>u_0(x).
\end{array}\right.\;\;\text{and}\;\;\;\;
\widetilde F_{\lambda,\eps}(x,s)
:=
\int_0^s\widetilde f_{\lambda,\eps}(x,t)\,dt.
$$
Consider
\[
\widetilde I_{\lambda,\eps}(u)
=
\frac12\|u\|^2
-
\int_\Omega
\widetilde F_{\lambda,\eps}(x,u)\,dx,
\qquad
u\in H_0^1(\Omega).
\]
Since $0<q<p<1$, the functional $\widetilde I_{\lambda,\eps}$ is
coercive and weakly lower semicontinuous. Hence, by the direct
method of the calculus of variations, there exists
$w_{\lambda,\eps}\in H_0^1(\Omega)$ such that
\[
\widetilde I_{\lambda,\eps}(w_{\lambda,\eps})
=
\inf_{H_0^1(\Omega)}\widetilde I_{\lambda,\eps}.
\]
In particular,
\[
-\Delta w_{\lambda,\eps}
=
\widetilde f_{\lambda,\eps}(x,w_{\lambda,\eps})
\quad\text{in }\Omega.
\]

We claim that
\(
w_{\lambda,\eps}\geq u_0
\quad\text{a.e. in }\Omega.
\)
Set
$$
h_\eps:=u_0-w_{\lambda,\eps}
$$
and define
\[
F
:=
\lambda u_0^p
-
\frac{1}{u_0^\beta}\chi_{\{u_0>0\}}
\in L^1_{\rm loc}(\Omega).
\]
By \eqref{subsolution-P},
\[
\left\langle
F-(-\Delta u_0),\varphi
\right\rangle
:=
\int_\Omega F\varphi\,dx
-
\int_\Omega\nabla u_0\nabla\varphi\,dx
\geq0,
\qquad
0\leq\varphi\in C_c^\infty(\Omega).
\]
Hence
\(
\mu:=F-(-\Delta u_0)
\)
is a nonnegative Radon measure in $\Omega$, and therefore
\[
-\Delta u_0=F-\mu
\quad\;\in\mathcal D'(\Omega).
\]
Since
$-\Delta w_{\lambda,\eps}
=
\widetilde f_{\lambda,\eps}
(x,w_{\lambda,\eps})
\in L^1_{\rm loc}(\Omega)
$,
it follows that $-\Delta h_\eps=-\Delta(u_0-w_{\lambda,\epsilon})$ is a Radon measure. Moreover,
$h_\eps\in H_0^1(\Omega)$, and thus
\(
-\Delta h_\eps
\in
\mathcal M_{loc} (\Omega)\cap H^{-1}(\Omega).
\)
In particular, $-\Delta h_\eps$ is diffuse with respect to the
$W^{1,2}$-capacity. Therefore, by the measure version of Kato's
inequality \cite[Theorem~1.1 and Remark~1]{BrezisPonce},
\[
-\Delta h_\eps^+
\leq
\chi_{\{h_\eps>0\}}(-\Delta h_\eps)
\quad\text{in }\mathcal D'(\Omega).
\]

On the set $\{h_\eps>0\}$, we have
$w_{\lambda,\eps}<u_0$, and hence
\[
-\Delta w_{\lambda,\eps}
=
\lambda u_0^p-g_\eps(u_0).
\]
Consequently, on the set $\{h_{\epsilon}>0\}$
$$-\Delta h_\eps
=\left(
g_\eps(u_0)-\frac{1}{u_0^\beta}\right)-\mu\;\;
\leq\;\;0,$$
where we used $
0\leq g_\eps(s)\leq s^{-\beta}$ and $\mu\geq0$. Hence
\(
-\Delta h_\eps^+\leq0
~\text{in }\mathcal D'(\Omega).
\)
Since $h_\eps^+\in H_0^1(\Omega)$, the weak maximum principle gives
\(
h_\eps^+=0.
\)
Therefore,
\(
w_{\lambda,\eps}\geq u_0\geq 0
~\text{a.e. in }\Omega.
\)

Thus the truncation is inactive at $w_{\lambda,\eps}$, and
\[
-\Delta w_{\lambda,\eps}
+g_\eps(w_{\lambda,\eps})
=
\lambda w_{\lambda,\eps}^p
\quad\text{in }\Omega.
\]

Using $w_{\lambda,\eps}$ as a test function and taking into account that
$g_\eps(t)t\geq0$, we obtain
\[
\|w_{\lambda,\eps}\|^2
\leq
\lambda S_{p+1}^{p+1}
\|w_{\lambda,\eps}\|^{p+1}.
\]
Hence $(w_{\lambda,\eps})$ is bounded in $H_0^1(\Omega)$ uniformly
with respect to $\eps>0$.

\medskip

\noindent \textbf{Part 2. Passage to the limit and monotonicity.}
Let $\eps_n\to0^+$. By
Lemma~\ref{lem:singular-compactness}, up to a subsequence, there
exists a nonnegative weak solution $w_\lambda$ of $(P_\lambda)$
such that
\(
w_{\lambda,\eps_n}\to w_\lambda
~\text{a.e. in }\Omega.
\)
Since
\(
w_{\lambda,\eps_n}\geq u_0
\quad\text{a.e. in }\Omega,
\)
we obtain
\(
w_\lambda\geq u_0
~\text{a.e. in }\Omega.
\)
The last assertion ensures that $w_\lambda\not\equiv0$.

Now, we suppose that $u_0$ is a nontrivial nonnegative weak
solution of $(P_{\lambda_0})$ and that $\lambda>\lambda_0$.
Then, for every nonnegative $\varphi\in C_c^\infty(\Omega)$,
\[
\begin{aligned}
\int_\Omega\nabla u_0 \nabla\varphi\,dx
+
\int_{\{u_0>0\}}\frac{\varphi}{u_0^\beta}\,dx
&=
\lambda_0\int_\Omega u_0^p\varphi\,dx
\leq
\lambda\int_\Omega u_0^p\varphi\,dx.
\end{aligned}
\]
Hence $u_0$ satisfies \eqref{subsolution-P}. The last
argument yields a nontrivial nonnegative weak solution of
\eqref{P}.
\end{proof}
Now, we are in a position to prove that the extremal parameter
$\bar\lambda$ is attained.

\medskip
\noindent{\bf Proof of Theorem \ref{theorem4a}.}
We divide the proof into two parts.

\medskip
\noindent \textbf{Part 1.} Existence of a nontrivial subsolution and attainment of $\bar\lambda$.
Let \(
\mathcal S
\mbox{~and~}
\bar\lambda:=\inf\mathcal S
\) be defined as \eqref{lambda-bar}. By the definition of $\bar\lambda$, there exists a sequence
$(\lambda_k)\subset\mathcal S$ such that
\(
\lambda_k\to\bar\lambda.
\)
For each $k$, choose a nontrivial nonnegative weak solution $u_k$
of $(P_{\lambda_k})$. In view of
Lemma~\ref{lem:nehari-identity-solution},
\(
u_k\in\mathcal N_{\lambda_k}.
\)
Therefore, Proposition~\ref{problema}-(iv) gives
\(
\lambda_k\geq\lambda_*
~\text{for every }k.
\)
In particular, we see that
\(
\bar\lambda\geq\lambda_*>0.
\)

Notice also that $(\lambda_k)$ is bounded. Then, there exists $\Lambda>0$ such that
$\lambda_k\leq\Lambda$ for every $k$. Once again, using the definition of the Nehari set, we see that 
\[
\|u_k\|^2
\leq
\lambda_k\|u_k\|_{p+1}^{p+1}
\leq
\Lambda S_{p+1}^{p+1}\|u_k\|^{p+1}.
\]
Thus $(u_k)$ is bounded in $H_0^1(\Omega)$. Up to a subsequence,
we know that \(
u_k\rightharpoonup\bar u
~\text{in }H_0^1(\Omega),
~
u_k\to\bar u
~\text{in }L^{p+1}(\Omega),
\)
and $u_k\to\bar u$ a.e. in $\Omega$. Furthermore, by \eqref{dis-ori} and $\lambda_k\leq\Lambda$, there
exists $d_\Lambda>0$, independent of $k$, such that
\(
\|u_k\|_{p+1}^{p+1}\geq d_\Lambda>0.
\)
Consequently, $\bar u\not\equiv0$ since
\[
\|\bar u\|_{p+1}^{p+1}
=
\lim_{k\to\infty}\|u_k\|_{p+1}^{p+1}
\geq d_\Lambda>0.
\]

We claim that $\bar u$ satisfies \eqref{subsolution-P} with
$\lambda=\bar\lambda$. Let
$\varphi\in C_c^\infty(\Omega)$, $\varphi\geq0$. Recall also that $u_k$ is a
weak solution of $(P_{\lambda_k})$,
\begin{equation}\label{LTD}
\int_\Omega\nabla u_k \nabla\varphi\,dx
+
\int_{\{u_k>0\}}\frac{\varphi}{u_k^\beta}\,dx
=
\lambda_k\int_\Omega u_k^p\varphi\,dx.
\end{equation}
We emphasize that \eqref{LTD} implies uniform boundedness for the integral $\int_{\{u_k>0\}}\frac{\varphi}{u_k^\beta}\,dx,$ once \begin{equation}\label{LTD2}\int_{\{u_k>0\}}\frac{\varphi}{u_k^\beta}\,dx\leq \Lambda ||\varphi||_{\infty}||u_k||_p^p+C||u_k||.\end{equation}
For each $\delta>0$, we consider the truncated integrals
\[
I_k^\delta
:=
\int_{\{u_k>\delta\}}\frac{\varphi}{u_k^\beta}\,dx.
\]
Since $u_k\geq0$ and $\{u_k>\delta\}\subset\{u_k>0\}$, we have
$I_k^\delta\leq\int_{\{u_k>0\}}\varphi u_k^{-\beta}\,dx$. On the set
$\{\bar u>\delta\}$, the strong convergence
$u_k\to\bar u$ in $L^{p+1}(\Omega)$ and the a.e.\ convergence
$u_k\to\bar u$ imply that $u_k^{-\beta}\varphi\to\bar u^{-\beta}\varphi$ a.e in $\Omega$, with
$\bigl|u_k^{-\beta}\varphi\bigr|\chi_{\{u_k>\delta\}}\leq \delta^{-\beta}\|\varphi\|_\infty$.
As a consequence, from Fatou's lemma,

$$
\int_{\{\bar u>\delta\}}\frac{\varphi}{\bar u^\beta}\,dx
\leq 
\liminf_{k\to\infty} \int_{\{u_k>\delta\}}
\frac{\varphi}{u_k^\beta}\,dx
\leq
\liminf_{k\to\infty}
\int_{\{u_k>0\}}
\frac{\varphi}{u_k^\beta}\,dx.
$$

From the uniform boundedness \eqref{LTD2} we may letting $\delta\downarrow0$ and using the Monotone Convergence
Theorem on the left-hand side (since
$\bar u^{-\beta}\varphi\chi_{\{\bar u>\delta\}}$ increases to
$\bar u^{-\beta}\varphi\chi_{\{\bar u>0\}}$ as $\delta\downarrow0$),
we obtain
\[
\int_{\{\bar u>0\}}
\frac{\varphi}{\bar u^\beta}\,dx
\leq
\liminf_{k\to\infty}
\int_{\{u_k>0\}}
\frac{\varphi}{u_k^\beta}\,dx.
\]
On the other hand, the weak convergence in $H_0^1(\Omega)$, the
strong convergence in $L^{p+1}(\Omega)$, imply that 
\[
\lambda_k\int_\Omega u_k^p\varphi\,dx
-
\int_\Omega\nabla u_k \nabla\varphi\,dx
\longrightarrow
\bar\lambda\int_\Omega\bar u^p\varphi\,dx
-
\int_\Omega\nabla\bar u \nabla\varphi\,dx.
\]
Therefore,
\[
\int_\Omega\nabla\bar u \nabla\varphi\,dx
+
\int_{\{\bar u>0\}}
\frac{\varphi}{\bar u^\beta}\,dx
\leq
\bar\lambda\int_\Omega\bar u^p\varphi\,dx.
\]
In particular,
\(
\frac{1}{\bar u^\beta}\chi_{\{\bar u>0\}}
\in L^1_{\rm loc}(\Omega).
\)
Hence, $\bar u$ is a nontrivial nonnegative weak subsolution of
$(P_{\bar\lambda})$. As a consequence, Proposition~\ref{prop:existence-monotone}
yields a nontrivial nonnegative weak solution
$u_{\bar\lambda}$ of $(P_{\bar\lambda})$. Thus 
\(
\bar\lambda\in\mathcal S.
\)

\medskip
\noindent \textbf{Part 2.} Location of $\bar\lambda$.
It remains to show that $\bar\lambda>\lambda_*$. Suppose, by
contradiction, that
\(
\bar\lambda=\lambda_*.
\)
Since $u_{\bar\lambda}$ is a nontrivial nonnegative weak solution
of $(P_{\lambda_*})$, Lemma~\ref{lem:nehari-identity-solution}
gives
\(
u_{\bar\lambda}\in\mathcal N_{\lambda_*}.
\)
In view of Proposition~\ref{problema}-(ii), we mention that
\(
\mathcal N_{\lambda_*}
=
\mathcal N_{\lambda_*}^0.
\)
Therefore, 
\(
u_{\bar\lambda}\in\mathcal N_{\lambda_*}^0.
\)
This is a contradiction with Proposition~\ref{non-exis-lam*}. As a consequence, 
\(
\lambda_*<\bar\lambda.
\)

Now,  we consider $\epsilon_1>0$  which is given by Theorem~\ref{theorem1} and
choose
\(
\lambda_0\in(\lambda^*-\epsilon_1,\lambda^*).
\)
In virtue of Theorem~\ref{theorem1}, Problem~$(P_{\lambda_0})$ admits a
nontrivial nonnegative weak solution. Thus $\lambda_0\in\mathcal S$
and
\(
\bar\lambda\leq\lambda_0<\lambda^*.
\)
Consequently,
\(
\lambda_*<\bar\lambda<\lambda^*,
\)
which completes the proof.

\noindent{\bf Proof of Theorem \ref{theorem4}.}
In light of Theorem~\ref{theorem4a},
$\bar\lambda\in\mathcal{S}$. Hence Problem $(P_{\;\bar\lambda})$ admits a
nontrivial nonnegative weak solution.
Let us consider $\lambda>\bar\lambda$. In virtue of 
Proposition~\ref{prop:existence-monotone} with
$\lambda_0=\bar\lambda$, we conclude that
Problem~\eqref{P} admits a nontrivial nonnegative weak solution
for every $\lambda>\bar\lambda$. Conversely, for each $0<\lambda<\bar\lambda$, we obtain that
$\lambda\notin\mathcal S$. Therefore, we infer that 
\(
\mathcal S=[\bar\lambda,+\infty),
\)
which completes the proof.

\section{\textbf{Declarations}}

\textbf{Ethical Approval: } It is not applicable.

\textbf{Competing interests:} There are no competing interests.

\textbf{Authors' contributions:} All authors wrote and reviewed this paper.

\textbf{Availability of data and materials:} All the data can be accessed from this article.

\end{document}